\documentclass[11pt]{article}
\usepackage[margin=1in]{geometry}
\usepackage{amsfonts,amsthm,amssymb,amsmath,mathtools}
\usepackage{helvet,framed,verbatim,epsfig,tikz}
\usepackage{booktabs,longtable,array}
\newcolumntype{L}[1]{>{\raggedright\arraybackslash}p{#1}}
\usepackage{microtype}
\usepackage{placeins}
\usepackage{needspace}
\usepackage[colorlinks=true,allcolors=blue]{hyperref}
\hypersetup{pdftitle={On the Harborth Conjecture. Part I. Quadrilaterals},
pdfauthor={Shiva Kintali},bookmarksnumbered=true,bookmarksdepth=3}
\newtheorem{theorem}{Theorem}[section]
\newtheorem{corollary}[theorem]{Corollary}
\newtheorem{proposition}[theorem]{Proposition}
\newtheorem{lemma}[theorem]{Lemma}
\newtheorem{conjecture}[theorem]{Conjecture}

\newtheorem{maintheorem}{Theorem}

\numberwithin{equation}{section}
\theoremstyle{definition}
\newtheorem{definition}[theorem]{Definition}
\newtheorem{example}[theorem]{Example}
\theoremstyle{remark}
\newtheorem{remark}[theorem]{Remark}
\newcommand{\Q}{\mathbb{Q}}
\newcommand{\R}{\mathbb{R}}
\newcommand{\Z}{\mathbb{Z}}
\newcommand{\calS}{\mathcal{S}}
\newcommand{\calO}{\mathcal{O}}
\newcommand{\kerp}{\operatorname{ker}}

\title{On the Harborth Conjecture. Part I \\ Quadrilaterals}

\author{{Shiva Kintali}\footnote{Email: \texttt{shiva.kintali@gmail.com}\ \ \ \ \ Homepage: \href{https://shivakintali.github.io/}{\nolinkurl{shivakintali.github.io}}  }}

\date{August 1, 2026}

\begin{document}
\maketitle
\begin{abstract}
Harborth conjecture states that every planar graph has a crossing-free straight-line drawing in which every edge has an integer length. Kleber's strengthening asks for the vertices themselves to have integer coordinates. In this series of papers, we make progress towards settling these conjectures. Kleber's conjecture reduces to local rational-distance statements for special polygons with at most five vertices. The triangle case is known from the results of Almering and Berry. In this paper, we prove the existence of a rational-distance point on an interior integer diagonal in all the cases for non-degenerate and degenerate quadrilaterals. In upcoming papers, we focus on non-degenerate and degenerate pentagons. \\

\end{abstract}

\section{Introduction}\label{sec:unified-introduction}\label{q1:sec:intro}

Harborth conjecture \cite{KemnitzHarborth2001} states that every finite simple
planar graph has a crossing-free straight-line drawing in which every edge
has integer length. Kleber's strengthening asks for the vertices themselves
to lie on the integer lattice~\cite{Kleber2008}.
In this paper, we call the latter an \emph{integral drawing}: all
vertices lie in \(\Z^2\), all edges are straight-line segments, edge
segments meet only at shared endpoints and contain no non-incident vertex,
and every edge has integer length.

\emph{The Kemnitz--Harborth algorithm} is inductive. It is enough to consider
plane triangulations. Every such triangulation with at least four vertices
has a vertex off the outer face of degree three, four, or five. Delete
such a vertex \(v\), retriangulate the face created by deletion, draw the
smaller triangulation, and then remove the auxiliary diagonals and put
\(v\) back inside that face. Since \(v\) has degree \(3\le d\le5\), the
local face is a triangle, quadrilateral, or pentagon. The auxiliary
diagonals have rational lengths in the inductively constructed drawing.

The graph-theoretic part of the method is elementary. The hard part is
Diophantine geometry. We need a point inside the relevant local polygon
whose distances to all boundary vertices are rational and whose open
segments to those vertices stay strictly inside the face. An interior point
is \emph{admissible} when it satisfies this strict visibility condition.
If it has rational coordinates,
then after a final global scaling all coordinates and all edge lengths
become integers. In this paper, we prove the existence of a rational-distance point
on an interior integer diagonal in all the cases for non-degenerate and degenerate quadrilaterals. 

\subsection{Basics}

Throughout this paper, graphs are finite and simple: loops and multiple
edges are excluded. All distances are Euclidean. For distinct points
\(u,v\), write \([u,v]\) for the closed segment, \((u,v)\) or
\(uv^\circ\) for the open segment, and \(uv=\|u-v\|\) for its length.
For coincident endpoints, set \([u,u]=\{u\}\), \((u,u)=\varnothing\),
and \(uu=0\).

\begin{definition}[Plane straight-line drawing]
Let \(G=(V,E)\) be a finite simple graph. A straight-line drawing of \(G\)
is an injective map $\varphi:V\longrightarrow \R^2$
in which each edge \(uv\in E\) is represented by the segment
\([\varphi(u),\varphi(v)]\). The drawing is plane if two distinct edge segments meet
only when the corresponding graph edges share an endpoint, and only at
that endpoint, and no edge segment contains a non-incident vertex.
\end{definition}

\begin{definition}[Integral drawing]
An integral drawing of \(G\) is a plane straight-line drawing
\(\varphi:V(G)\to\Z^2\) such that
\[
\|\varphi(u)-\varphi(v)\|\in\Z
\qquad\text{for every }uv\in E(G).
\]
\end{definition}

\begin{definition}[Rational drawing]
A rational drawing of \(G\) is a plane straight-line drawing
\(\varphi:V(G)\to\Q^2\) such that
\[
\|\varphi(u)-\varphi(v)\|\in\Q
\qquad\text{for every }uv\in E(G).
\]
Rational drawings are used as an intermediate object. A single scaling turns
them into an integral drawing.
\end{definition}

\begin{lemma}[Scaling]\label{q1:lem:scaling}
If \(G\) has a rational drawing, then \(G\) has an integral drawing.
\end{lemma}

\begin{proof}
Since the graph is finite, one positive integer clears the denominators
of all vertex coordinates and edge lengths. Scaling by this integer
preserves the plane embedding and gives an integral drawing.
\end{proof}

\begin{conjecture}[Harborth \cite{KemnitzHarborth2001}]\label{q1:conj:harborth}
Every finite simple planar graph has a plane straight-line drawing whose edge lengths are
integers.
\end{conjecture}

\begin{conjecture}[Kleber \cite{Kleber2008}]\label{q1:conj:kleber}
Every finite simple planar graph has an integral drawing.
\end{conjecture}

\subsection{Planar graphs}

A \emph{plane graph} has a fixed embedding by simple arcs meeting only at shared
endpoints and avoiding non-incident vertices. Its faces are the complementary
components, including the unbounded \emph{outer face}. A drawing \emph{preserves the embedding} if an
orientation-preserving homeomorphism of the plane carries one labeled
embedding to the other, preserving cyclic orders, the outer face, and
component nesting. A \emph{triangulation} is connected and every face, including the outer face, is
bounded by a cycle of three distinct vertices.

The insertion reduction uses the following elementary facts.

\begin{lemma}[Triangulating a plane graph]\label{q1:lem:triangulate}
Every simple plane graph with at least three vertices can be extended, by
adding only non-crossing edges and no new vertices, to a simple plane
triangulation. The embedding of the original graph is unchanged.
\end{lemma}

\begin{proof}
Choose a simple plane supergraph $H$ with as many edges as possible,
retaining the fixed embedding; a maximum exists because only finitely
many vertex pairs are available. If $H$ were disconnected, two components incident with
a common face could be joined by an arc through that face. Maximality
therefore makes $H$ connected. It also excludes a cut vertex: consecutive
edges $vu,vw$ leading to distinct components of $H-v$ permit the absent
edge $uw$ through their common face corner. Since $H$ has at least three
vertices, it is thus 2-connected, and its facial boundaries are simple
cycles. On a facial cycle $v_1\cdots v_k$ with $k\ge4$, the
alternating chords $v_1v_3,v_2v_4$ cannot both lie on the other side of
the cycle; one is absent and can be inserted through the face. The same
Jordan-curve argument applies to the outer face on the sphere. Thus all
faces are triangles. The added edges are topological arcs, so the
argument preserves the original embedding, including component nesting.
\end{proof}

The added edges are topological arcs. After drawing the triangulated
supergraph, we delete them; no geometric triangulation of a prescribed
straight-line drawing is asserted.

\begin{lemma}[Euler degree bound]\label{q1:lem:euler}
Every simple planar graph with at least one vertex has a vertex of degree
at most five. Moreover, a plane triangulation with at least four vertices
has a vertex off its outer face whose degree is three, four, or five.
\end{lemma}

\begin{proof}
For $n\ge3$, triangulation and Euler's formula give $m\le3n-6$,
so some degree is at most five; $n\le2$ is immediate. A triangulation
with $n\ge4$ has minimum degree three: degrees zero and one are excluded
by connectedness and the triangular facial boundaries, while degree two
would force its two incident facial triangles to exhaust the sphere. Now
$\sum_v(6-\deg v)=12$, while the three outer vertices contribute at
most nine. An interior vertex therefore has degree three, four, or five.
\end{proof}

\begin{lemma}[The polygon around a deleted vertex]\label{q1:lem:deleted-vertex}
Let \(T\) be a plane triangulation with at least four vertices, and let
\(v\) be a vertex of degree \(d\). List the neighbors of \(v\) in cyclic order around \(v\) as
$v_1,v_2,\ldots,v_d$. Then \(v_1v_2\cdots v_dv_1\) is a cycle in \(T-v\), and this cycle bounds
the face created by deleting \(v\). If \(v\) is not on the outer face,
the created face is bounded.
\end{lemma}

\begin{proof}
The distinct cyclic neighbors satisfy $v_iv_{i+1}\in E(T)$ because
every corner at $v$ belongs to the triangle $vv_iv_{i+1}$. These
triangles form a disk on the sphere with boundary $v_1\cdots v_dv_1$.
Deleting $v$ opens this disk into one face. The graph $T-v$ remains
connected, since any path through $v$ can be rerouted along the neighbor
cycle. If $v$ is interior, none
of the merged faces is the outer face, so the new face is bounded.
\end{proof}

\subsection{Rational-distance points}

A line is called \emph{rational} if it has an equation
\(ax+by=c\), where \(a,b,c\in\Q\) and \((a,b)\ne(0,0)\).
Rational endpoints alone do not ensure a rational segment length; for
example, the distance from \((0,0)\) to \((1,1)\) is \(\sqrt2\).

\begin{definition}[Rational point and rational-distance point]
A point \(p=(x,y)\in\R^2\) is rational if \(x,y\in\Q\). Let
\(S=\{p_1,\ldots,p_d\}\subset\R^2\). A point \(q\in\R^2\) is a
rational-distance point for \(S\) if
\[ \|q-p_i\|\in\Q \qquad\text{for every }i=1,\ldots,d.\]
\end{definition}

\begin{lemma}[Three centers force rational coordinates {\cite[Theorem~2.2]{GeelenGuoMcKinnon2008}}]\label{q1:lem:three-centers}
Let \(A,B,C\in\Q^2\) be non-collinear. If \(p\in\R^2\) satisfies
\[
\|p-A\|,\quad \|p-B\|,\quad \|p-C\|\in\Q,
\]
then \(p\in\Q^2\).
\end{lemma}

\begin{proof}
Put \(d_A=\|p-A\|\), and define \(d_B,d_C\) analogously.
Subtracting squared distance equations gives
\[
 \begin{pmatrix}(B-A)^T\\(C-A)^T\end{pmatrix}p
 =\frac12
 \begin{pmatrix}
 \|B\|^2-\|A\|^2+d_A^2-d_B^2\\
 \|C\|^2-\|A\|^2+d_A^2-d_C^2
 \end{pmatrix}.
\]
Both the matrix and right-hand side are rational. The matrix is
invertible because the three centers are non-collinear, so its unique
solution \(p\) has rational coordinates.
\end{proof}

\begin{lemma}[Integer and rational formulations]\label{q1:lem:int-rat-local}
Consider a local insertion statement requiring integer vertex coordinates,
integer side lengths, and possibly integer lengths of specified diagonals.
Suppose its other hypotheses and its conclusion are preserved by positive
rational scaling. Then it is equivalent to the formulation in which all
those coordinates and lengths are required to be rational.
\end{lemma}

\begin{proof}
Clear denominators, apply the integer statement, and scale back. The
scaling is a bijection on the corresponding point sets, preserving the
stated geometric conditions, existence, and infinitude. The converse
uses that integers are rational.
\end{proof}

\subsection{Polygons}

Let \(n\geq 3\), and let $P=v_1v_2\cdots v_n v_1$ be a closed polygonal chain in \(\mathbb{R}^2\). Indices are taken
cyclically modulo \(n\), so that \(v_{n+1}=v_1\). The sides of \(P\)
are the closed line segments $e_i=[v_i,v_{i+1}],\qquad i=1,\dots,n.$
Two sides \(e_i\) and \(e_j\) are called adjacent if they share an
endpoint in the cyclic order, that is, if \(j=i+1\), \(i=j+1\), or
\(\{i,j\}=\{1,n\}\). Otherwise they are called non-adjacent.

\begin{definition}[Simple polygon]
The polygonal chain \(P=v_1v_2\cdots v_n v_1\) is a \emph{simple
polygon} if its boundary has no self-intersections. Equivalently: the vertices \(v_1,\dots,v_n\) are distinct; every side has positive length, i.e. \(v_i\neq v_{i+1}\) for every \(i\); two adjacent sides intersect only in their common endpoint; two non-adjacent sides are disjoint.
Thus a simple polygon may be convex or concave, but it has no repeated
vertices, no crossing sides, no overlapping sides, and no vertex lying in
the relative interior of a non-incident side.
\end{definition}

For a simple polygon \(P\), we also write \(P\) for the compact polygonal
region consisting of its boundary together with its bounded interior.
We write $\partial P$ for its boundary and $\operatorname{int}(P)$ for
its strict interior; all interiors are taken in $\R^2$.

A boundary angle is \emph{flat} if it equals $\pi$ and \emph{reflex} if
it exceeds $\pi$.

\begin{definition}[Non-degenerate polygon]
A polygon \(P=v_1v_2\cdots v_n v_1\) is \emph{non-degenerate} if it is
simple and has no flat boundary angle. Equivalently, \(P\) is simple and
for each $i$, the three consecutive vertices $v_{i-1},v_i,v_{i+1}$
are not collinear. In this convention, every interior angle of \(P\) is strictly between
\(0\) and \(2\pi\), and no interior angle is equal to \(\pi\).
\end{definition}

The \emph{closed kernel} and \emph{admissible insertion points} are defined by
\[
\begin{aligned}
\kerp(P)&=\{p\in P:[p,q]\subset P\text{ for every }q\in P\},\\
\mathcal A(P)&=\{p\in\operatorname{int}(P):(p,v_i)\subset
\operatorname{int}(P)\text{ for every }i\}.
\end{aligned}
\]
Closed-kernel visibility permits boundary contacts~\cite[p.~2]{ORourke1987};
admissibility excludes them, including spokes through another boundary vertex.

\Needspace{11\baselineskip}
\begin{lemma}[Kernel half-planes]\label{q1:lem:kernel-halfplanes}
Orient a simple polygon \(P=v_1\cdots v_nv_1\) counterclockwise, and put
\[
 h_i(x)=\det(v_{i+1}-v_i,x-v_i).
\]
Then
\[
 \kerp(P)=\bigcap_{i=1}^n\{x:h_i(x)\ge0\},\qquad
 \operatorname{int}\kerp(P)=\bigcap_{i=1}^n\{x:h_i(x)>0\}.
\]
Every point of \(\operatorname{int}\kerp(P)\) is admissible.
\end{lemma}

\begin{proof}
We prove the standard weak half-plane characterization
\cite[Section~4.2, p.~117]{ORourke1987}. If $h_i(p)<0$, the segment from
$p$ to a relative interior point of side $i$ approaches that side through
the local exterior of $P$. Thus a kernel point satisfies every weak
inequality.

Conversely, let $p$ satisfy all weak inequalities. Choose
$q\in\operatorname{int}(P)$ outside the finitely many lines through $p$
and the vertices different from $p$. The open segment $(q,p)$ then
avoids all vertices and cannot run along a side. If it first exited $P$
before reaching $p$, it would cross a side transversely from that side's
inward half-plane to its outward half-plane. Along the segment the
corresponding affine function changes from positive to negative, so it
remains negative up to $p$, a contradiction. Hence $[q,p]\subset P$, including the endpoint
$p$ by closedness. Every point of $P$ is a limit of such generic interior
points $q$; taking limits of the segments proves $p\in\kerp(P)$.
Finally, an equality $h_i(p)=0$ prevents a neighborhood of $p$ from lying
in that half-plane, whereas finitely many strict inequalities persist on
a neighborhood. This proves the formula for the kernel's interior.

If $B(p,\varepsilon)\subset\kerp(P)$ and $0<t<1$, then
$B((1-t)p+tv_i,(1-t)\varepsilon)\subset P$: each point of this disk
lies on a segment joining a kernel point to $v_i$. Thus every open
spoke from $p$ lies strictly inside $P$.
\end{proof}

\subsection{Triangles}

\begin{theorem}[Berry {\cite[p.~391]{Berry1992}}]\label{q1:thm:berry}
Let \(A,B,C\in\R^2\) be non-collinear. Suppose that
\[
\|A-B\|\in\Q,\qquad
\|A-C\|^2\in\Q,\qquad
\|B-C\|^2\in\Q.
\]
Then the points whose distances from \(A,B,C\) are all rational are dense in
\(\R^2\).
\end{theorem}

Almering~\cite{Almering1963} proved the classical rational-triangle case.
Berry~\cite{Berry1992} proved the version stated above. We use the
following consequence for triangle insertion.

\begin{corollary}[Triangle insertion]\label{q1:cor:triangle}
Let \(\Delta ABC\) be a non-degenerate triangle with
\[
A,B,C\in\Q^2,
\qquad
\|A-B\|,\|B-C\|,\|C-A\|\in\Q.
\]
Then the interior of \(\Delta ABC\) contains a rational point \(p\in\Q^2\)
such that
\[
\|p-A\|,\quad \|p-B\|,\quad \|p-C\|\in\Q.
\]
\end{corollary}

\begin{proof}
The interior is a nonempty open subset of \(\R^2\), so
Berry's density theorem (Theorem~\ref{q1:thm:berry}) supplies a
rational-distance point there. The three-center lemma
(Lemma~\ref{q1:lem:three-centers}) makes its coordinates rational.
Convexity of the triangle makes the point admissible.
\end{proof}

\subsection{Quadrilaterals}

A quadrilateral is an ordered four-cycle with straight sides.
It is \emph{simple} if its boundary is a simple closed polygonal curve. A
simple quadrilateral is \emph{convex} if the union of its boundary and
bounded interior is a convex subset of \(\R^2\). It is \emph{concave} if it
is simple but not convex; equivalently, one of its interior angles is reflex.
A simple quadrilateral is called non-degenerate if no three of its vertices
are collinear.

\begin{definition}[Interior diagonal]
Let \(Q\) be a simple polygon. A diagonal is a segment joining
two non-adjacent vertices. A diagonal is an interior diagonal if its relative
interior is contained in the interior of \(Q\).
\end{definition}

\begin{definition}[Diagonal rational-distance sets]
Let \(U,V,A,B\in\Q^2\) with $U\ne V$, and suppose that \(UV\) is the chosen diagonal.
We write
\[
\mathcal R^\circ(U,V;A,B)
\]
for the set of rational points \(P\) on the open segment \(UV\) such that $PU, PV, PA, PB$ 
are all rational. We write \(\mathcal R(U,V;A,B)\) for the analogous set on the closed segment \(UV\).
\end{definition}

\begin{lemma}[An interior diagonal sees the quadrilateral]\label{q1:lem:quad-diagonal}
Let \(Q=ABCD\) be a simple quadrilateral, and suppose that \(AC\) is an
interior diagonal. If \(p\) lies on the open segment \(AC\), then the
open segments from \(p\) to \(A,B,C,D\) lie in the strict interior of
\(Q\), and the four closed spokes meet pairwise only at \(p\).
Thus \(p\) is admissible, even if the quadrilateral has a flat vertex.
\end{lemma}

\begin{proof}
The interior diagonal cuts the polygonal region into the two triangles
\(\Delta ABC\) and \(\Delta ACD\), by the Jordan separation of its
two boundary paths. Both triangles are non-degenerate: otherwise a
boundary path would lie on \([A,C]\), contradicting interiority of
the diagonal. Their interiors are disjoint, so \(B\) and \(D\) lie
on opposite sides of its supporting line. The open
segments \((p,B)\) and \((p,D)\) lie in their respective triangle
interiors; \((p,A)\) and \((p,C)\) lie on the open interior diagonal.
Thus all four open segments lie in the strict interior of \(Q\).
The opposite-side condition also shows that their closed versions meet
only at \(p\).
\end{proof}

\begin{remark}[Refuted diagonal insertion proposal]\label{q1:conj:quad}
For comparison, the original unrestricted proposal asserted the following.
Let \(Q=ABCD\) be a simple quadrilateral with
\[
A,B,C,D\in\Z^2,
\qquad
AB,BC,CD,DA\in\Z,
\]
and suppose that \(Q\) has an interior diagonal of integer length. Then the
open segment of that diagonal contains a rational point \(x\in\Q^2\) such
that
\[
\|x-A\|,\quad \|x-B\|,\quad \|x-C\|,\quad \|x-D\|\in\Q.
\]
As stated for all simple quadrilaterals, this assertion is false:
Propositions~\ref{q2:prop:counterexample} and
\ref{q2:prop:tilted-collinear-counterexample} give three-collinear
counterexamples satisfying its arithmetic hypotheses. Theorem~\ref{q1:thm:main}
proves the assertion under non-degeneracy. In the three-collinear case,
Theorem~\ref{q2:thm:main-degenerate} supplies infinitely many strict-interior
rational-distance points, with no universal restriction to the diagonal.
\end{remark}

By Lemma~\ref{q1:lem:int-rat-local}, the proved non-degenerate restriction
and the three-collinear interior theorem have equivalent rational
formulations. These are the forms used in the induction below, whose
intermediate drawings have rational coordinates and edge lengths.

\subsection{Pentagons}

For pentagons, the insertion point must satisfy both the five rational
distance conditions and the strict visibility conditions needed to draw
the five new edges.

\begin{theorem}[Pentagon kernels have interior]\label{q1:thm:pentagon-star}
The closed kernel of every simple pentagon has nonempty interior.
Consequently every simple pentagon has admissible insertion points.
\end{theorem}

\begin{proof}
Suppress any flat vertices. This leaves a polygon with \(m\) vertices,
where \(3\le m\le5\), bounding the same region and having the same
kernel: at a flat vertex the two incident sides lie on the same supporting
line and impose the same inward half-plane condition as the single side
that replaces them. Triangulate this polygon by interior diagonals, using the
polygon triangulation theorem~\cite[Theorem~1.2, p.~12]{ORourke1987}.
There is a vertex \(v\) common to every triangle. This is immediate for
\(m=3,4\). For \(m=5\), adjacency across the two diagonals connects
the three triangles in a chain. The two shared sides of the middle
triangle have a common endpoint, which belongs to all three triangles.

Orient the boundary counterclockwise. For each side not incident with
\(v\), the triangle having that side has \(v\) as its third vertex,
so \(v\) lies strictly in that side's inward half-plane. The two sides
incident with \(v\) have distinct supporting lines, since flat vertices
were suppressed. Their inward normals are therefore not opposite, so there
is a direction pointing strictly into both inward half-planes.
A sufficiently small displacement of \(v\) in that
direction still satisfies every other inward inequality strictly.
Lemma~\ref{q1:lem:kernel-halfplanes} therefore gives a point in the
interior of the kernel, and makes every such point admissible.
\end{proof}

Kernel nonemptiness also follows from the \(n=5\) case of Chv\'atal's
art-gallery theorem~\cite{Chvatal1975}; see
\cite[Theorem~1.1, p.~9]{ORourke1987}. The preceding proof establishes
the stronger interior conclusion, including polygons with flat vertices.
Since \(\Q^2\) is dense in \(\R^2\), admissible points with rational
coordinates exist as well. Requiring rational distances to all five
vertices is the additional arithmetic problem.

\begin{remark}[Closed-kernel visibility is insufficient]\label{q1:rem:kernel-admissibility}
Consider the pentagon with successive vertices
\[
 A=(0,0),\quad B=(546,0),\quad C=(546,728),\quad
 D=(330,440),\quad E=(0,728).
\]
Its side lengths are \(546,728,360,438,728\). The point
\(p=(273,364)\) lies in its strict interior and in its closed kernel,
and
\[
 pA=pB=pC=pE=455,\qquad pD=95.
\]
These statements can be checked directly: the boundary is simple and
counterclockwise. Write \((V_1,\ldots,V_5)=(A,B,C,D,E)\) and
\(V_6=V_1\). The determinants
\(\det(V_{i+1}-V_i,p-V_i)\) for its five oriented edges are
\[
 198744,\quad198744,\quad0,\quad41496,\quad198744.
\]
Thus \(p\in\kerp(P)\) by Lemma~\ref{q1:lem:kernel-halfplanes}. It lies on no boundary
segment: the only supporting line containing it is the line \(CD\),
and \(p\) lies beyond \(D\). Hence it is strictly inside the polygon.
But
\[
 D-p=(57,76),\qquad C-p=(273,364)=\frac{91}{19}(D-p).
\]
The spoke \([p,C]\) passes through \(D\), overlaps \([p,D]\), and
contains the boundary side \([D,C]\). Consequently this rational-distance
point in the strict interior and closed kernel is not admissible.
In fact, \([A,D]\) is an interior diagonal of integer length \(550\)
contained in the closed kernel. Its endpoints satisfy all the weak
inward inequalities, so the whole segment does also; its open segment
meets no boundary side. Since \(A=(0,0)\) and \(D=(55/91)C\), every
point of \((A,D)\) has \([D,C]\) in its spoke to \(C\). Thus even a diagonal segment
contained in the closed kernel need not contain an admissible point.
\end{remark}

\begin{conjecture}[Admissible pentagon insertion]\label{q1:conj:pentagon}
Let \(P=A_1A_2A_3A_4A_5\) be a simple pentagon with
\[
 A_i\in\Z^2,\qquad A_iA_{i+1}\in\Z\qquad(1\le i\le5),
\]
where indices are taken modulo \(5\). Then there exists an admissible
insertion point \(p\in\Q^2\) such that
\[
 \|p-A_i\|\in\Q\qquad(1\le i\le5).
\]
\end{conjecture}

This is a stronger geometric requirement than the original closed-kernel
insertion proposal, as Remark~\ref{q1:rem:kernel-admissibility} explains.
Lemma~\ref{q1:lem:int-rat-local} makes its integer and rational
formulations equivalent. The insertion point must be allowed to range
over the face: the following example rules out a universal restriction
to diagonals, even when two integer-length interior diagonals triangulate
the pentagon.

\begin{example}[An integral pentagon with no admissible diagonal point]\label{q1:ex:pentagon-no-diagonal}
In contrast to the quadrilateral visibility lemma
(Lemma~\ref{q1:lem:quad-diagonal}), a simple integer-coordinate,
integer-sided pentagon need not have an admissible point on any diagonal, even when it
has a triangulation by integer-length diagonals. Consider
\[
 A=(48,0),\quad B=(0,0),\quad C=(0,20),\quad
 D=(-24,10),\quad E=(24,-45).
\]
Its successive side lengths are \(48,20,26,73,51\). The triangles
\(BCD\), \(BDE\), and \(BEA\) have positive oriented twice-areas
\(480,840,2160\). Their interiors are pairwise disjoint:
\(BD\) separates \(C\) and \(E\), \(BE\) separates \(D\) and
\(A\), and the triangles \(BCD\) and \(BEA\) lie in opposite
coordinate quadrants. Their union is therefore the simple
counterclockwise pentagon with the stated boundary. The two interior diagonals used in this triangulation
have lengths \(BD=26\) and \(BE=51\).

Every admissible point \(p=(x,y)\) must satisfy \(y<0\): near the convex
vertex \(A\), the polygon interior lies strictly below the supporting
line \(AB\), so the open spoke to \(A\) must approach it from below.
Likewise \(x<0\), since the polygon interior near the convex vertex
\(C\) lies strictly to the left of the supporting line \(BC\).
But the open diagonals \(AD\) and \(BD\) have \(y>0\), the open
diagonals \(BE\) and \(CE\) have \(x>0\), and the open diagonal
\(AC\) has both coordinates positive. Thus no open diagonal contains
an admissible point, regardless of its distances to the vertices.

The obstruction is specifically the diagonal restriction. Indeed, the
closed inward half-plane description gives
\[
 \kerp(P)=\{(x,y): x\le0,\ y\le0,\ 55x+48y+840\ge0\},
\]
a triangle with vertices $(0,0)$, $(-168/11,0)$, and $(0,-35/2)$.
The other two edge inequalities,
\(10x-24y+480\ge0\) and \(-45x+24y+2160\ge0\), hold strictly
at these three vertices and hence throughout the triangle. For example,
\((-1,-1)\) lies in its interior and is admissible.
The example refutes a universal diagonal strengthening of
Conjecture~\ref{q1:conj:pentagon}; the full interior conjecture remains
an arithmetic question.
\end{example}

\begin{figure}[htbp]
\centering
\begin{tikzpicture}[x=0.065cm,y=0.065cm,line cap=round,line join=round]
\fill[blue!15] (0,0)--({-168/11},0)--(0,{-35/2})--cycle;
\draw[thick] (48,0)--(0,0)--(0,20)--(-24,10)--(24,-45)--cycle;
\draw[dashed] (-24,10)--(0,0)--(24,-45);
\fill (48,0) circle (1.3pt) node[right] {$A$};
\fill (0,0) circle (1.3pt) node[above right] {$B$};
\fill (0,20) circle (1.3pt) node[above] {$C$};
\fill (-24,10) circle (1.3pt) node[left] {$D$};
\fill (24,-45) circle (1.3pt) node[below] {$E$};
\end{tikzpicture}
\caption{The pentagon in Example~\ref{q1:ex:pentagon-no-diagonal}.
The shaded triangle is its closed kernel. Dashed segments show the
integer-length triangulation diagonals \(BD\) and \(BE\).}
\label{q1:fig:pentagon-no-diagonal}
\end{figure}

\subsection{Kemnitz--Harborth induction}

The non-degenerate diagonal theorem and the three-collinear interior
theorem together resolve the degree-four insertion step. The remaining
assumption in the following reduction is the admissible pentagon
conjecture.
%The refuted unrestricted quadrilateral diagonal assertion is not an assumption of this theorem.

\begin{theorem}[Conditional reduction to pentagon insertion]\label{q1:thm:conditional}
Assume Conjecture~\ref{q1:conj:pentagon}. Then every finite simple plane
graph has an integral drawing preserving its embedding, including its
outer face. In particular, every finite simple planar graph has an
integral drawing.
\end{theorem}

\begin{proof}
Graphs of order at most two are immediate. For larger graphs, use
Lemma~\ref{q1:lem:triangulate} and then delete the added edges, so
it suffices to treat plane triangulations. Induct on their order,
maintaining a rational drawing with the prescribed cyclic orders and
outer face. These data determine the plane embedding of a connected
triangulation: the facial identifications extend over their disks on
the sphere and can fix infinity. The triangle with vertices
$(0,0),(3,0),(0,4)$ gives the base case, labeled and reflected as necessary.

For $n\ge4$, choose an interior vertex $v$ of degree $d\in\{3,4,5\}$
by Lemma~\ref{q1:lem:euler}. Deleting $v$ creates the bounded face
$C=v_1\cdots v_dv_1$ of Lemma~\ref{q1:lem:deleted-vertex}.
Triangulate $T-v$ in its fixed embedding. Euler's formula gives
$3(n-1)-6-((3n-6)-d)=d-3$ new edges. All lie in this face, because every
other face is already triangular; the augmentation introduces no duplicate
edge. Draw the resulting smaller
triangulation by induction and delete these auxiliary diagonals.
The face $C$ has rational vertices and sides. The removed diagonals have
rational lengths and their open segments lie strictly inside this face.

For $d=3$, use the triangle-insertion corollary
(Corollary~\ref{q1:cor:triangle}). For $d=4$ without
collinear triples, use the rational version of
Theorem~\ref{q1:thm:main} on the auxiliary interior diagonal and
the interior-diagonal visibility lemma (Lemma~\ref{q1:lem:quad-diagonal}).
If $d=4$ and a collinear triple occurs, exactly three vertices are collinear,
since a simple polygon cannot have all its vertices on a line.
The collinear vertices form a consecutive boundary triple. Its central
vertex must lie between its two neighbors, since otherwise the adjacent
sides would overlap. Thus the face is a triangle with one subdivided side. Its auxiliary
interior diagonal necessarily joins that middle vertex to the off-line
vertex. Its rational length supplies the fifth length in
Theorem~\ref{q2:thm:main-degenerate}, which gives a strict-interior
point; all four open spokes then lie in the triangular interior.
For $d=5$, use the rational form of
Conjecture~\ref{q1:conj:pentagon}.

In each case the open spokes avoid the existing graph. Two spokes
cannot overlap, since the farther would contain a boundary vertex in
its open segment, contradicting admissibility. Their cyclic order
agrees with the boundary order: intersect them with a small circle
about the insertion point and use the order of disjoint arcs in the
resulting annulus. At each boundary vertex the new edge enters the
same face sector. Thus the embedding and outer face are preserved.
The homeomorphism restricts to the original graph after deletion of
augmentation edges, including any nesting of components. Finally apply
Lemma~\ref{q1:lem:scaling}.
\end{proof}

\subsection{Main results and scope}\label{sec:principal-results}

Both principal results assume rational vertex coordinates, rational side
lengths, and a chosen interior diagonal of positive rational length.
Clearing denominators gives the equivalent integer formulations.
A point \emph{qualifies} for a configuration if its distances from all
four vertices are positive rational numbers.

\begin{maintheorem}[Non-degenerate diagonal theorem]\label{thm:intro-nondegenerate}
Let $UAVB$ be a simple quadrilateral with vertices in $\Q^2$ listed in
boundary order, rational side lengths, and no three vertices collinear. If $[U,V]$ is an interior
diagonal of rational length, then $(U,V)$ contains a qualifying point
$P\in\Q^2$.
\end{maintheorem}

\begin{maintheorem}[Three-collinear interior theorem]\label{thm:intro-collinear}
Let $X,Y,Z,C\in\Q^2$, with $X,Y,Z$ distinct and collinear in that
order and $C$ off their line. If $XY,YZ,ZC,CX,YC$ are rational,
then the strict interior of $\triangle XZC$, equivalently of the
quadrilateral $XYZC$, contains infinitely many distinct qualifying
points in $\Q^2$.
\end{maintheorem}

These follow from Theorems~\ref{q1:thm:main}
and~\ref{q2:thm:main-degenerate}, respectively: scale the coordinates and
given lengths to integers, apply the corresponding theorem, and scale back.
This preserves rationality, the required location, and distinctness of
points (Lemma~\ref{q1:lem:int-rat-local}). Admissibility follows from
Lemma~\ref{q1:lem:quad-diagonal} for Theorem~\ref{thm:intro-nondegenerate}
and from convexity of the triangular region for Theorem~\ref{thm:intro-collinear}.
Figure~\ref{fig:configurations} shows the configurations;
Table~\ref{tab:all-results} records the infinitude refinements and the finite
and empty diagonal examples. Rational coordinates are automatic for a point
at rational distance from three non-collinear rational centers
(Lemma~\ref{q1:lem:three-centers}).

\begin{figure}[htbp]
\centering
\begin{tikzpicture}[x=1cm,y=1cm,line cap=round,line join=round,
 every node/.style={font=\small}]
\begin{scope}
\fill[gray!10] (0,0)--(1.4,1.25)--(3,0)--(1.5,-.9)--cycle;
\draw (0,0)--(1.4,1.25)--(3,0)--(1.5,-.9)--cycle;
\draw[dashed] (0,0)--(3,0);
\fill (0,0) circle (1.3pt) node[left] {$U$};
\fill (1.4,1.25) circle (1.3pt) node[above] {$A$};
\fill (3,0) circle (1.3pt) node[right] {$V$};
\fill (1.5,-.9) circle (1.3pt) node[below] {$B$};
\node at (1.5,-1.65) {Convex};
\end{scope}
\begin{scope}[xshift=4.4cm]
\fill[gray!10] (0,0)--(2.8,1.25)--(2,0)--(2.7,-.9)--cycle;
\draw (0,0)--(2.8,1.25)--(2,0)--(2.7,-.9)--cycle;
\draw[dashed] (0,0)--(2,0);
\fill (0,0) circle (1.3pt) node[left] {$U$};
\fill (2.8,1.25) circle (1.3pt) node[above] {$A$};
\fill (2,0) circle (1.3pt) node[left,yshift=-8pt] {$V$};
\fill (2.7,-.9) circle (1.3pt) node[below] {$B$};
\node at (1.5,-1.65) {Concave};
\end{scope}
\begin{scope}[xshift=8.8cm]
\fill[gray!10] (0,0)--(3,0)--(1.5,1.25)--cycle;
\draw (0,0)--(1.3,0)--(3,0)--(1.5,1.25)--cycle;
\draw[dashed] (1.3,0)--(1.5,1.25);
\fill (0,0) circle (1.3pt) node[left] {$X$};
\fill (1.3,0) circle (1.3pt) node[below] {$Y$};
\fill (3,0) circle (1.3pt) node[right] {$Z$};
\fill (1.5,1.25) circle (1.3pt) node[above] {$C$};
\node at (1.5,-1.65) {Three-collinear};
\end{scope}
\end{tikzpicture}
\caption{The principal configurations, drawn schematically. Dashed segments
are the prescribed interior diagonals. Theorems~\ref{thm:intro-nondegenerate}
and~\ref{thm:intro-collinear} distinguish points on these diagonals from
points allowed anywhere in the strict interior.}\label{fig:configurations}
\end{figure}

\subsection{Related work}\label{sec:related-work}

F\'ary's theorem gives straight-line planar drawings without arithmetic
constraints~\cite{Fary1948}. Kemnitz and Harborth developed the integral
insertion program~\cite{KemnitzHarborth2001}. Geelen, Guo, and McKinnon
proved the integer-edge conclusion for cubic planar graphs by rational
approximation; their construction also yields rational coordinates before
scaling~\cite{GeelenGuoMcKinnon2008}. Biedl obtained integer coordinates and
edge lengths on an $O(n)\times O(n)$ grid in the 3-connected cubic planar
case~\cite{Biedl2011}. Sun treated planar $(2,3)$-sparse graphs by rigidity
methods~\cite{Sun2011}, and Benediktovich proved approximation results for
connected graphs of maximum degree four with a vertex of smaller degree,
and for 3-trees~\cite{Benediktovich2013}. Chang and Sun established
Harborth's conjecture for all 4-regular planar
graphs~\cite{ChangSun2024}. F\"orster, Kobourov, Miller, and Zink gave
polynomial-size grid constructions for trees and
cacti~\cite{ForsterEtAl2025}; B\"ungener et al.\ subsequently improved
the grid bounds~\cite{BungenerEtAl2026}. Our results concern prescribed polygonal
faces rather than whole graph classes or bounds on drawing size.

Earlier diagonal constructions occur in Kemnitz--Harborth and in
Sun~\cite[Propositions~11--13]{Sun2013}; the latter includes a convex
quadrilateral with paired side lengths $3,3,4,4$ and diagonal $5$.
Chang and Sun also formulate stability of permissible diagonal
constructions under sufficiently small perturbations preserving rational
coordinates and the required rational edge
lengths~\cite[Section~4.1 and Proposition~6]{ChangSun2024}.
Visibility imposes a further condition: a point in the closed polygon
kernel need not be admissible. We use the classical kernel and
triangulation theory of Chv\'atal and O'Rourke
\cite{Chvatal1975,ORourke1987}, while keeping this distinction explicit.

Almering proved density of rational-distance points for a rational-sided
triangle~\cite{Almering1963}; Berry extended the hypotheses to one
rational side and rational squared side lengths~\cite{Berry1992}.
Corvaja, Turchet, and Zannier established a density criterion for three
non-collinear centers, including all triples of rational
points~\cite{CorvajaTurchetZannier2025}. Berry's square work uses a Kummer
surface and gives obstructions in certain four-vertex
families~\cite{Berry1990}. A rational-sided square has irrational
diagonals and does not satisfy our chosen-diagonal hypothesis.

The principal arithmetic input is Love's detour equation and exact
torsion criterion~\cite{Love}. The special curves connect with concordant
forms and simultaneous Pythagorean conditions~\cite{SelderSpindler,HalbeisenHungerbuhler2022}. We use Mazur's
torsion classification~\cite{Mazur} and standard elliptic-curve methods
\cite{SilvermanAEC,Cremona}; classical quartic obstructions are discussed
in~\cite{Mordell,Sidokhine}. The rank-zero classification of the curve
$48.a4$ supplies an auxiliary torsion obstruction~\cite{Cremona,LMFDB48a4}.
The different fixed curve responsible for the finite examples is
determined by an explicit descent here and identified
in~\cite{LMFDB,LMFDB210e6}.

Love's torsion criterion, together with real-component density, supplies
the generic detour case and identifies the exceptional torsion conditions. Here we resolve the exceptional endpoint
configurations through explicit models and the real location of torsion
values, prove the three-collinear interior theorem, and establish the
special diagonal refinements. The finite and empty diagonal examples
separate existence, infinitude, and location under the stated arithmetic
hypotheses.

\subsection{Results at a glance}\label{sec:results-overview}

Table~\ref{tab:all-results} summarizes the geometric conclusions under the
arithmetic hypotheses above and the conditional graph-insertion consequence.
The arithmetic case map appears beside the
exceptional proofs in Table~\ref{tab:endpoint-cases}; the convex refinements
are collected in Table~\ref{tab:convex-refinements}.

\begin{table}[htbp]
\centering\small
\begin{tabular}{@{}L{.23\linewidth}L{.18\linewidth}L{.32\linewidth}L{.19\linewidth}@{}}
\toprule
Configuration & Location & Conclusion & Result \\
\midrule
Non-degenerate & Chosen open interior diagonal & At least one qualifying point.
& Theorem~\ref{q1:thm:main} \\
Concave & Chosen open interior diagonal & Infinitely many qualifying points.
& Propositions~\ref{q1:prop:generic}, \ref{q1:prop:no-concave-both-torsion} \\
Finite convex examples & Chosen open interior diagonal & An aligned kite and a
symmetric parallelogram each have exactly one.
& Theorems~\ref{q1:thm:finite-aligned}, \ref{q1:thm:finite-example} \\
Three-collinear & Strict triangular interior & Infinitely many; the chosen
open diagonal can contain none.
& Theorem~\ref{q2:thm:main-degenerate}; Propositions~\ref{q2:prop:counterexample},
\ref{q2:prop:tilted-collinear-counterexample} \\
Special three-collinear configurations & Prescribed open diagonal & Infinitude
under the perpendicular or same-side criteria.
& Theorems~\ref{q2:thm:perpendicular}, \ref{q2:thm:same-side} \\
Graph insertion & Admissible interior of each face & Admissible pentagon insertion would
yield an integral drawing preserving the prescribed plane embedding.
& Theorem~\ref{q1:thm:conditional} \\
\bottomrule
\end{tabular}
\caption{Geometric conclusions. In the quadrilateral rows, a qualifying
point has positive rational distances to all four vertices. The special
diagonal criteria have the hypotheses stated in
Sections~\ref{q2:sec:perpendicular-special} and~\ref{q2:sec:same-side-special}.
The graph-insertion row assumes Conjecture~\ref{q1:conj:pentagon}.}\label{tab:all-results}
\end{table}

\subsection{Proof strategy and organization}\label{sec:proof-overview}
\label{sec:organization}

A rational point on the chosen rational-length diagonal is already at
rational distance from its two endpoints. The task is to make its distances
to the other two vertices rational simultaneously. Rational normalization
turns these requirements into the square conditions
$1+\alpha_i^2\in(\Q^\times)^2$ for $\alpha_i\in\Q$, $i=1,2$, related by
$s\alpha_1+\alpha_2=r$ with $s<0$. The parameter $\alpha_1$ varies affinely along the
diagonal, and an open interval $I$ represents its strict interior.
Writing $\alpha=(t-t^{-1})/2$, $t\in\Q_{>0}$, parametrizes one square
condition (Lemma~\ref{q1:lem:S}); elliptic curves allow us to impose both.
Theorem~\ref{q1:thm:love} gives density in the generic case. In the singular overlap
$(r,s)=(0,-1)$, the two square conditions coincide and an elementary
rational parametrization gives density. In the nonsingular exceptional families,
torsion can represent only a few finite detour values; a non-torsion
endpoint lift gives solutions on the interior side. When both endpoint lifts are
torsion, their positions force an elementary interior solution. This is
the single existence argument; convex refinements and finite examples
are consequences or additional classifications.

For three-collinear data, a rational-radius circumcircle and a quarter-angle
parametrization make three distances rational. The fourth gives a
genus-one quartic. Love's criterion identifies the only possible torsion
obstruction for a distinguished rational point on its elliptic model.
Fermat's right-triangle theorem excludes that obstruction, and density
supplies infinitely many points on an interior
circular arc. Separate diagonal models prove the perpendicular and
same-side refinements.

The preceding subsections give the insertion framework and its conditional
pentagon reduction. Section~\ref{sec:arithmetic-tools} collects the shared
arithmetic tools. Section~\ref{q1:sec:main} proves the non-degenerate
theorem through the endpoint-line formulation;
Section~\ref{q1:sec:convex-exceptional} gives its refinements and finite
examples. Sections~\ref{q2:sec:introduction},
\ref{q2:sec:perpendicular-special}, and~\ref{q2:sec:same-side-special}
treat the general three-collinear theorem and the two diagonal
specializations. Appendices~\ref{q1:app:caseIII} and
\ref{q2:app:fixed-curve} contain the quadratic exceptional torsion verification and the
shared fixed-curve descent.
The pentagon arithmetic conjecture remains open; the conditional reduction
does not settle either global graph-drawing conjecture.

\section{Shared arithmetic tools}\label{sec:arithmetic-tools}

\paragraph{Notation guide.}
Distances such as $UV$ denote lengths; $(U,V)$ denotes an open segment.
The following notation recurs in the arithmetic arguments.
\begin{center}\small
\begin{tabular}{@{}L{.22\linewidth}L{.71\linewidth}@{}}
\toprule
Symbol & Meaning \\
\midrule
$\calS$ & Rational $\alpha$ for which $1+\alpha^2$ is a rational square. \\
$\alpha_1,\alpha_2;r,s$ & Detour parameters and coefficients in
$s\alpha_1+\alpha_2=r$; the endpoint-line normalization has $s<0$. \\
$I=(\lambda,\mu),\ \tau$ & The required parameter interval and the
intersection value $\tau=r/(s-1)$ in the endpoint-line setup. \\
$\calO,\ \calO_E$ & The elliptic-curve identity; $E(\R)^0$ is the real
identity component. The superscript $0$ denotes a component, not torsion. \\
$O,\ O_k$ & The geometric vertex $O$ in three-collinear configurations,
and the chosen origin $O_k$ on the quadratic exceptional quartic. \\
\bottomrule
\end{tabular}
\end{center}
Every genus-one group law is transported from the elliptic model using the
origin specified with that model. Other symbols are local to their section;
$\rho,\sigma$ are square-root distance coordinates in
Section~\ref{q2:sec:same-side-special}, and $R$ is a circle radius in
Section~\ref{q2:sec:introduction}.

\subsection{Rational normalization and detour parameters}
\label{q1:sec:normalization}

We first place a segment of rational length on the \(x\)-axis by an
isometry that preserves rational coordinates. We then express rational
distances to an off-axis point through a single set of rational parameters.

\begin{lemma}[Rational normalization]\label{q1:lem:normalization}
Let \(U,V\in\Q^2\) be distinct points with rational distance
\(D=\|V-U\|>0\). Then there is an affine Euclidean isometry
\(F:\R^2\longrightarrow\R^2\) whose matrix and translation vector have
rational entries such that
\[
 F(U)=(0,0),\qquad F(V)=(D,0).
\]
The inverse map \(F^{-1}\) also has rational coefficients. Consequently
\(F\) carries \(\Q^2\) bijectively onto \(\Q^2\) and preserves all Euclidean
distances. It maps the open segment \((U,V)\) bijectively onto
\((0,D)\times\{0\}\).
\end{lemma}

\begin{proof}
Write $V-U=(m,n)$ and put
\[
R=D^{-1}\begin{pmatrix}m&n\\-n&m\end{pmatrix},\qquad
F(X)=R(X-U),\qquad F^{-1}(Y)=U+R^TY.
\]
Since $m^2+n^2=D^2$, the matrix $R$ is rational and orthogonal,
and $R(V-U)=(D,0)$. The formulas prove every assertion; in particular
$F((1-\theta)U+\theta V)=(\theta D,0)$.
\end{proof}

The rationality of \(D\) is necessary for this arithmetic normalization:
an affine map with rational coefficients sends the rational point \(V\)
to a rational point, so its image \((D,0)\) must have \(D\in\Q\).

A \emph{square in $\Q$} has the form $h^2$ with $h\in\Q$; zero
is allowed. Set
\[
\calS=\{\alpha\in\Q:1+\alpha^2\text{ is a square in }\Q\}.
\]
The positive root of $1+\alpha^2$ is then rational.

\begin{lemma}[The elementary distance criterion]\label{q1:lem:distance}
Let \(C=(e,f)\in\Q^2\) with \(f\ne0\), and let \(P=(x,0)\) with
\(x\in\Q\). Then \(PC\in\Q\) if and only if
\[
\frac{x-e}{f}\in\calS.
\]
\end{lemma}

\begin{proof}
Put \(\alpha=(x-e)/f\in\Q\). Then
\[
 PC^2=(x-e)^2+f^2=f^2(1+\alpha^2),
 \qquad PC=|f|\sqrt{1+\alpha^2}.
\]
Because \(|f|\in\Q_{>0}\), \(PC\in\Q\) holds if and only if
\(\sqrt{1+\alpha^2}\in\Q\), which is exactly \(\alpha\in\calS\).
\end{proof}

The absolute value makes the criterion valid on either side of the
\(x\)-axis. If \(f=0\), the separate on-axis formula is simply
\(PC=|x-e|\in\Q\); the quotient in the lemma is then undefined.

\begin{lemma}[Parametrization of \(\calS\)]\label{q1:lem:S}
The map \(\phi:\R^\times\longrightarrow\R\) defined by
\[
 \phi(t)=\frac{t-t^{-1}}2
\]
restricts to a bijection from \(\Q_{>0}\) onto \(\calS\), with inverse
\[
 t=\alpha+\sqrt{1+\alpha^2}\qquad(\alpha\in\calS),
\]
where the square root is positive. In particular,
\[
 \calS=\left\{\frac{t-t^{-1}}2:t\in\Q_{>0}\right\}
      =\left\{\frac{t-t^{-1}}2:t\in\Q^\times\right\}.
\]
The set \(\calS\) is dense in \(\R\), meets every nonempty open interval
in infinitely many points, and satisfies \(\calS=-\calS\).
\end{lemma}

\begin{proof}
For $t\ne0$,
$1+\phi(t)^2=((t+t^{-1})/2)^2$. Conversely, if
$h=\sqrt{1+\alpha^2}\in\Q_{>0}$, then $t=h+\alpha>0$ and
$t^{-1}=h-\alpha$, giving the inverse. On $(0,\infty)$,
$\phi'(t)=(1+t^{-2})/2>0$ and the endpoint limits are $-\infty$
and $+\infty$. Thus $\phi$ is a homeomorphism onto $\R$ and takes
the dense set $\Q_{>0}$ to a set meeting each open interval infinitely
often. Negation preserves the defining square condition; also
$\phi(-1/t)=\phi(t)$, so negative parameters add no values.
\end{proof}

\paragraph{Positive lifts.}\label{q1:par:positive-lifts}
Write
\begin{equation}\label{q1:eq:positive-lift}
\vartheta(\alpha)=\alpha+\sqrt{1+\alpha^2}>0,\qquad
\alpha=\frac{\vartheta(\alpha)-\vartheta(\alpha)^{-1}}2.
\end{equation}
For real $\alpha$, the function $\vartheta$ is real analytic, and
$\vartheta(\alpha)\in\Q$ if and only if $\alpha\in\calS$.
On an open interval of real $\alpha$, suppose $t=\vartheta(\alpha)$
and the auxiliary coordinates define a real analytic branch on a
nonsingular curve. The rational function
$(t-t^{-1})/2$ is regular there and recovers $\alpha$. Its composition
with the branch has derivative $1$, so the branch parametrizes an open
arc locally by the inverse function theorem. In the applications below,
strictly positive radicands ensure analyticity along the branch. At the
endpoints used in the density arguments, the rational-distance hypotheses
and the displayed formulas make all lift coordinates rational.

\subsection{Torsion, halving, and real density}
\label{q2:sec:halving-preparation}

For an elliptic curve $E$, write $\calO$, or $\calO_E$ when the curve needs emphasis, for its identity,
$E(\Q)_{\mathrm{tors}}$ for rational torsion, and $E(\R)^0$ for the
real identity component. A point is \emph{non-torsion} if it has infinite
order. We use standard genus-one and Mordell--Weil
theory as in~\cite{SilvermanAEC,Cremona}.

\begin{theorem}[Mazur's torsion theorem {\cite[Theorem~(8), p.~35]{Mazur}}]
\label{q2:thm:mazur-torsion}
Let $E$ be an elliptic curve over $\Q$. Then $E(\Q)_{\mathrm{tors}}$ is
either cyclic of order $N$, where $1\le N\le10\ \text{or}\ N=12$, or is isomorphic to
\begin{equation*}
\mathbb Z/2\mathbb Z\times\mathbb Z/2r\mathbb Z,
\qquad 1\le r\le4.                                           
\end{equation*}
\end{theorem}

\begin{lemma}[A common duplication formula]\label{lem:common-duplication}
Let $E:y^2=x^3+Ax^2+Bx$ be nonsingular, with $A,B\in\Q$.
For $P=(x,y)\in E$ with $y\ne0$,
\begin{equation}\label{eq:common-duplication}
x(2P)=\left(\frac{x^2-B}{2y}\right)^2.
\end{equation}
\end{lemma}

\begin{proof}
Substitute the tangent slope $m=(3x^2+2Ax+B)/(2y)$ in
$x(2P)=m^2-A-2x$ and use the curve equation.
\end{proof}

\begin{lemma}[Halving criterion]\label{q2:lem:halving}
Let
\[
E:\quad y^{2}=(x-e_{1})(x-e_{2})(x-e_{3}),
\]
where $e_{1},e_{2},e_{3}\in\Q$ are distinct. Let
$P=(x_{0},y_{0})\in E(\Q)$, with $y_{0}\ne0$. Then
$P\in 2E(\Q)$ if and only if each of
\begin{equation*}
x_{0}-e_{1},\qquad x_{0}-e_{2},\qquad x_{0}-e_{3}              
\end{equation*}
is a square in $\Q$.
\end{lemma}

\begin{proof}
This is the usual $2$-descent criterion~\cite[Chapter~X]{SilvermanAEC}.
Here are formulas fixing signs and domains. If $P=2Q$, $Q=(u,v)$,
the tangent slope $\lambda=f'(u)/(2v)$ satisfies
\begin{equation*}
x(2Q)-e_i=\left(\lambda-\frac{v}{u-e_i}\right)^2,
\qquad f(x)=\prod_i(x-e_i).
\end{equation*}
The assumption $y_0\ne0$ excludes $Q=\calO$ and implies $v\ne0$
and $u\ne e_i$, so the three values are rational squares. Conversely, if
all three values are rational squares, choose $r_i\in\Q$ with
$r_i^2=x_0-e_i$. Since their product has square $y_0^2\ne0$, we may
change one sign to arrange
\begin{equation*}
r_1r_2r_3=-y_0,
\end{equation*}
and put $\sigma_1=r_1+r_2+r_3$, $\sigma_2=r_1r_2+r_2r_3+r_3r_1$.
Then
\begin{align}
u&=x_0+\sigma_2,\label{q2:eq:halving-u}\\
v&=-y_0-\sigma_1\sigma_2.\label{q2:eq:halving-v}
\end{align}
Indeed $u-e_i=(r_i+r_j)(r_i+r_k)$ and
$v=-(r_1+r_2)(r_1+r_3)(r_2+r_3)$. No factor vanishes: $r_i+r_j=0$ would imply $e_i=e_j$. Hence
$v\ne0$ and $(u,v)\in E(\Q)$. The identity
$f'(u)=-2\sigma_1v$ gives tangent slope $-\sigma_1$, and
$x(2Q)=\sigma_1^2+e_1+e_2+e_3-2u=x_0$. At abscissa $x_0$
the tangent has ordinate $v-\sigma_1(x_0-u)=-y_0$, so reflection
in the $x$-axis gives $2Q=P$.
\end{proof}

\begin{lemma}[Density on the real identity component]
\label{q2:lem:real-density}
Let $E$ be an elliptic curve over $\R$, and let $T\in E(\R)^0$. If $T$
has infinite order, then the cyclic subgroup $\{nT:n\in\mathbb Z\}$
is dense in $E(\R)^0$.
\end{lemma}

\begin{proof}
The compact connected one-dimensional Lie group $E(\R)^0$ is
isomorphic to $\R/\Z$. Infinite order corresponds to an irrational
class, whose multiples are dense: every proper closed subgroup of
the circle is finite. In particular, every nonempty open arc contains
infinitely many distinct multiples of $T$.
\end{proof}

The group $E(\R)$ has at most two connected components. Thus for any
non-torsion $R\in E(\R)$, the point $2R$ has infinite order and lies
in $E(\R)^0$. Its multiples are dense there, and the odd multiples
$R+2nR$ are dense in the component containing $R$. If $E$ is defined
over $\Q$ and $R\in E(\Q)$, all these points are rational.

\subsection{The quartic obstruction}

\begin{lemma}[Fermat's right-triangle theorem]
\label{q2:lem:fermat-right-triangle}
No non-degenerate rational right triangle has area equal to a nonzero
rational square.
\end{lemma}

\begin{proof}
Clear denominators of the sides and the square root of the area, and
choose an integer right triangle of square
area with least hypotenuse. Dividing out a common divisor preserves
square area, so the triangle is primitive, with sides
$2mn,m^2-n^2,m^2+n^2$ for coprime $m>n>0$ of opposite parity.
Its integer area $mn(m-n)(m+n)$ is a rational square, hence an integer
square. The four factors are pairwise
coprime: the only possible extra common divisor of $m-n,m+n$ is
$2$, excluded by parity. Hence
\[
m=u^2,\quad n=v^2,\quad m+n=w^2,\quad m-n=t^2.
\]
Here $n$ is even: if $m$ were even, then the square $m$ would be
$0\pmod4$ and the odd square $n$ would be $1\pmod4$, contradicting
$m-n=t^2$. Thus $v$ is
even, while $w,t$ are odd and $w>t>0$. The positive integers
$\alpha=(w+t)/2$, $\beta=(w-t)/2$ satisfy
\[
\alpha^2+\beta^2=m=u^2,\qquad
\frac{\alpha\beta}{2}=\frac n4=(v/2)^2.
\]
This is another integer right triangle of nonzero square area, with
hypotenuse $u<m^2+n^2$, contradicting minimality.
\end{proof}

\begin{lemma}[Fermat quartic obstruction]\label{q2:lem:fermat-quartic}
There are no $z,W\in\Q$ with $z\ne0$ satisfying
\begin{equation}
W^2=z^4+6z^2+1.                                             \label{q2:tag:7.4}
\end{equation}
\end{lemma}

\begin{proof}
Replacing $W$ by $|W|$, take $W>0$. Then $W>|z^2-1|$ and the positive rational numbers
\begin{equation*}
W-(z^2-1),\quad W+(z^2-1),\quad 2(z^2+1)
\end{equation*}
are the sides of a right triangle, because
\begin{equation*}
(z^2-1)^2,\quad (z^2+1)^2,\quad W^2
\end{equation*}
are squares in arithmetic progression with difference $4z^2$.
Its area is
\begin{equation*}
\frac{W^2-(z^2-1)^2}{2}=4z^2=(2z)^2,
\end{equation*}
contradicting Lemma~\ref{q2:lem:fermat-right-triangle}.
\end{proof}

\subsection{A common quartic-to-cubic transformation}\label{sec:quartic-model}

Suppose
\begin{equation}
w^{2}=Az^{4}+Bz^{3}+Cz^{2}+Dz+\kappa^{2}.                   \label{q2:tag:6.12}
\end{equation}
Assume $A,B,C,D,\kappa\in\mathbb Q$, $A\kappa\ne0$, and that the quartic
on the right has no repeated root. The double cover of the $z$-line is
branched at four points, so its smooth projective model has genus one;
it contains the rational point $(0,\kappa)$. On the open subset $z\ne0$, define
\begin{equation}
X=\frac{2\kappa(w+\kappa)+Dz}{z^{2}}.                       \label{q2:tag:6.13}
\end{equation}
Solving \eqref{q2:tag:6.13} for \(w\), substituting in \eqref{q2:tag:6.12}, and dividing by
\(z^{2}\), gives a quadratic equation in \(z\):
\begin{equation*}
(X^{2}-4\kappa^{2}A)z^{2}
+(-2DX-4\kappa^{2}B)z
+(D^{2}-4\kappa^{2}X-4\kappa^{2}C)=0.                       
\end{equation*}
On this quadratic relation, the square of
$2(X^{2}-4\kappa^{2}A)z-2DX-4\kappa^{2}B$ equals its discriminant.
Defining
\begin{equation}
Y=\frac{2(X^{2}-4\kappa^{2}A)z-2DX-4\kappa^{2}B}{4\kappa},
\label{q2:tag:6.15}
\end{equation}
one obtains
\begin{equation}
Y^{2}=X^{3}+CX^{2}+(BD-4\kappa^{2}A)X
+\kappa^{2}B^{2}+AD^{2}-4\kappa^{2}AC.                \label{q2:tag:6.16}
\end{equation}
Writing $q(z)$ for the quartic polynomial in \eqref{q2:tag:6.12} and
$g(X)$ for the cubic polynomial in \eqref{q2:tag:6.16}, direct expansion
gives $\operatorname{disc}_X g=\operatorname{disc}_z q$. The Weierstrass discriminant is
$16$ times this nonzero value, so the cubic curve is nonsingular.
Conversely, whenever \(X^{2}-4\kappa^{2}A\ne0\), one recovers
\begin{equation}
z=\frac{4\kappa Y+2DX+4\kappa^{2}B}
{2(X^{2}-4\kappa^{2}A)},\qquad
w=\frac{Xz^{2}-Dz}{2\kappa}-\kappa.                        \label{q2:tag:6.16a}
\end{equation}
These formulas are mutually inverse after also excluding points where
the recovered $z$ is zero, equivalently
$4\kappa Y+2DX+4\kappa^2B=0$. Together with $z=0$ and
$X^2=4\kappa^2A$, these exclude only finitely many points. Thus
\eqref{q2:tag:6.13}, \eqref{q2:tag:6.15}, and \eqref{q2:tag:6.16a}
give a birational equivalence over $\Q$, extending to an isomorphism
of smooth projective models.

The isomorphism sends $(0,\kappa)$ to infinity: the derivative of
$w^2-q(z)$ with respect to $w$ is $2\kappa\ne0$ at that point,
so $z$ is a local parameter there, and the numerator of
\eqref{q2:tag:6.13} tends to $4\kappa^2\ne0$, giving $X$ a double
pole. Thus the model identifies the quartic origin $(0,\kappa)$ with
elliptic infinity and recovers the parameter $z$ through
\eqref{q2:tag:6.16a}; the application specifies which rational point is
used to establish non-torsion.

\Needspace{18\baselineskip}
\section{Non-degenerate quadrilaterals}\label{q1:sec:main}

\begin{theorem}[Main quadrilateral theorem]\label{q1:thm:main}
Let \(Q=UAVB\) be a simple quadrilateral in \(\R^2\), with the vertices
listed in boundary order, and suppose that no three vertices are collinear.
Let \([U,V]\) be an interior diagonal. Assume that
\[
 U,A,V,B\in\Z^2,
 \qquad UA,AV,VB,BU,UV\in\Z_{>0}.
\]
Then there exists \(P\in\Q^2\cap(U,V)\) such that
\[
 PU,PV,PA,PB\in\Q_{>0}.
\]
In particular, \(P\) lies in the strict interior of \(Q\).
The conclusion holds for both convex and concave quadrilaterals.
\end{theorem}

The diagonal is prescribed, and no rationality of the other diagonal
is assumed. The rational version follows by
Lemma~\ref{q1:lem:int-rat-local}; admissibility follows from
Lemma~\ref{q1:lem:quad-diagonal}. We prove existence through the
following endpoint theorem. Refinements and the two finite examples
are collected in Section~\ref{q1:sec:convex-exceptional}.

\subsection{The endpoint-line reduction}
\label{q1:sec:detour}
\begin{theorem}[Endpoint-line theorem]\label{q1:thm:endpoint-line}
Let \(U,V\) be distinct rational points on a rational line \(L\), and let
$A,B\in\Q^2\setminus L$ lie on opposite sides of \(L\). Suppose that
$UV,\ UA,\ VA,\ UB,\ VB\in\Q$. Assume also that neither \(U\) nor \(V\)
lies on the line \(AB\). Then the open segment \(UV\) contains a rational
point \(P\) such that $PA,\ PB\in\Q$.
Consequently \(PU,PV,PA,PB\) are all rational.
\end{theorem}

By rational normalization and reflection, assume
\[
U=(0,0),\quad V=(D,0),\quad A=(a,b),\quad B=(c,d),
\qquad D\in\Q_{>0},\quad a,b,c,d\in\Q,\quad b>0>d.
\]
The five endpoint lengths are rational and neither endpoint is on $AB$.
For $P=(x,0)$, $x\in\Q$, put
\[
\alpha_1=\frac{x-a}{b},\quad \alpha_2=\frac{c-x}{d},\quad
s=\frac bd<0,\quad r=\frac{c-a}{d},\quad
\lambda=-\frac ab,\quad\mu=\frac{D-a}{b},\quad I=(\lambda,\mu).
\]
We call $(\alpha_1,\alpha_2)\in\calS^2$ satisfying
$s\alpha_1+\alpha_2=r$ a \emph{rational detour solution}.

\begin{lemma}[The detour equation]\label{q1:lem:detour}
With the notation above and \(x\in\Q\),
\begin{equation}\label{q1:eq:detour}
s\alpha_1+\alpha_2=r.
\end{equation}
Moreover, \(P=(x,0)\) has rational distance from both \(A\) and \(B\) if and
only if $\alpha_1,\alpha_2\in\calS$.
Finally, \(P\) lies on the open segment \(UV\) if and only if
\(\alpha_1\in I\). In particular, the map
\[
\alpha\longmapsto (a+b\alpha,0)
\]
is a bijection from
\[
\{\alpha\in\calS\cap I:r-s\alpha\in\calS\}
\quad\text{onto}\quad \mathcal R^\circ(U,V;A,B).
\]
\end{lemma}

\begin{proof}
Substitution gives $s\alpha_1+\alpha_2=r$, and
\[
PA=b\sqrt{1+\alpha_1^2},\qquad PB=(-d)\sqrt{1+\alpha_2^2}.
\]
Since $b,-d>0$ are rational, these distances are rational exactly
when both parameters belong to $\calS$. The inverse affine map
$x=a+b\alpha_1$ is rational and injective; $b>0$ makes
$0<x<D$ equivalent to $\alpha_1\in I$. Finally $PU=x$ and
$PV=D-x$ are rational, proving both directions of the bijection.
\end{proof}

The endpoint conditions are exactly
\[
\lambda,\mu,r-s\lambda,r-s\mu\in\calS,\qquad
\lambda,\mu\ne\tau,\qquad \tau=\frac r{s-1}.
\]
Indeed $a+b\tau=(bc-ad)/(b-d)$ is the abscissa where the line $AB$
meets the $x$-axis. No assumption is made
on whether $\tau$ lies in $I$. It remains to find a detour value in $I$.

\subsection{The generic case}
\label{q1:sec:generic}

We combine Love's non-torsion criterion with its consequence for the
density of rational detour solutions. Part~(\ref{q1:part:love-torsion})
records Love's exact torsion criterion; in
part~(\ref{q1:part:love-density}) we give the density argument on the
finite real sign branches needed here. The real interval in this density
statement may be prescribed in advance.

\begin{theorem}[Love's torsion criterion and its density consequence~\cite{Love}]\label{q1:thm:love}
Let \(r,s\in\Q\), \(s\ne0\), \((r,s)\ne(0,\pm1)\), and consider
\[
s\alpha_1+\alpha_2=r,\qquad
\alpha_1,\alpha_2\in\calS.
\]
Associated to this equation are the elliptic curve and rational point
\[
E_{r,s}:\quad y^2=x^3+(1+r^2+s^2)x^2+s^2x,
\qquad R=(-1,r).
\]
\begin{enumerate}
\renewcommand{\labelenumi}{\textup{(\roman{enumi})}}
\renewcommand{\theenumi}{\roman{enumi}}
\item\label{q1:part:love-torsion} \textup{Torsion criterion.}
The point \(R\) has infinite order if and only if
\[
r\ne0,\qquad s\ne\pm1,\qquad
4r^2s\ne\pm(1-s^2)^2,
\]
and its exact order in the complementary nonsingular cases is
\[
2 \quad\text{if } r=0,\qquad
4 \quad\text{if } s=\pm1,\qquad
8 \quad\text{if } 4r^2s=\pm(1-s^2)^2.
\]
\item\label{q1:part:love-density} \textup{Density conclusion.}
When \(R\) has infinite order, the set
\[
\mathcal D_{r,s}
=\{\alpha\in\calS:r-s\alpha\in\calS\}
\]
is dense in \(\R\). Equivalently, rational detour solutions are dense on
the real line \(s\alpha_1+\alpha_2=r\). For each fixed
\((\varepsilon_1,\varepsilon_2)\in\{\pm1\}^2\), rational points are also
dense on the finite real branch obtained by adjoining
\[
h_1=\varepsilon_1\sqrt{1+\alpha_1^2},\qquad
h_2=\varepsilon_2\sqrt{1+\alpha_2^2}.
\]
\end{enumerate}
\end{theorem}

\begin{proof}
\emph{Part~(\ref{q1:part:love-torsion}): torsion criterion.}
The discriminant of the displayed Weierstrass equation is
\[
\Delta(E_{r,s})
=16s^4\bigl(r^2+(s-1)^2\bigr)\bigl(r^2+(s+1)^2\bigr).
\]
It is nonzero under the stated hypotheses. Substitution verifies that
\(R=(-1,r)\) lies on \(E_{r,s}\). Love's exact criterion
\cite[Theorem~1.4]{Love} gives the infinite-order equivalence.
The division-polynomial calculation in \cite[Section~3.5]{Love}
gives orders \(2\), \(4\), and \(8\) in the three respective
exceptional cases and shows that no other finite order occurs.

\emph{Part~(\ref{q1:part:love-density}): density conclusion.}
Assume that $R$ has infinite order. To connect the curve to the detour
equation, take
\[
\eta=\begin{pmatrix}0&s\\1&-r\end{pmatrix}.
\]
Its determinant is \(-s\ne0\). Let \(H\) be the smooth projective model of
the affine quartic
\[
H_\eta:\quad w^2=q(t),\qquad
q(t)=4s^2t^2+(1-t^2-2rt)^2.
\]
Indeed,
\[
\operatorname{disc}_t(q)
=4096s^4\bigl(r^2+(s-1)^2\bigr)\bigl(r^2+(s+1)^2\bigr)\ne0,
\]
so the quartic has four distinct roots and its smooth projective model
has genus one. It has the rational point \(O_H=(0,1)\). By
\cite[equation~(9)]{Love}, its Jacobian is \(E_{r,s}\). Using \(O_H\) as
origin therefore gives an isomorphism over \(\Q\)
\[
\iota:E_{r,s}\longrightarrow H
\]
\nopagebreak[4]
sending the elliptic-curve identity \(\calO\) to \(O_H\). The recovered
parameter is the detour coordinate $\alpha_1$ below; when $R=(-1,r)$
is non-torsion, $2R$ is the point used for density on this model.

On the affine chart \(t\ne0\), the detour projection of
\cite[Proposition~5.1]{Love} takes the form
\[
\alpha_2=\frac{1-t^2}{2t},\qquad
\alpha_1=\frac{r-\alpha_2}{s}.
\]
The two square identities are
\[
1+\alpha_2^2=\left(\frac{1+t^2}{2t}\right)^2,\qquad
1+\alpha_1^2=\left(\frac{w}{2st}\right)^2.
\]
Thus every rational point in this chart gives a rational detour solution.
Conversely, given a rational detour solution, put
\[
t=\sqrt{1+\alpha_2^2}-\alpha_2>0,\qquad
w=2|s|t\sqrt{1+\alpha_1^2}.
\]
Both coordinates are rational; substitution gives \(w^2=q(t)\) and
recovers the original detour parameters. In particular, every finite
rational detour solution has a lift with \(t>0\) and \(w>0\).

We now verify the density consequence of \cite[Lemma~5.3]{Love} on the
specific real branch needed here. Let \(C\) be the connected component of
\(H(\R)\) containing \(O_H\). The component group of \(E_{r,s}(\R)\) has
order at most two, so \(2R\) lies in its identity component and remains
non-torsion. By the real-component density lemma
(Lemma~\ref{q2:lem:real-density}), the multiples of \(2R\) are dense in
that identity component. The rational points
\(\iota(2nR)\), \(n\in\Z\), are therefore dense in \(C\).

For every real \(t\), \(q(t)>0\): at \(t=0\) it equals \(1\), and for
\(t\ne0\) its term \(4s^2t^2\) is strictly positive. Hence the open arc
\[
\{(t,\sqrt{q(t)}):t>0\}
\]
lies in \(C\), since the same continuous branch extends to \(O_H\) at
\(t=0\). On \((0,\infty)\), the function
\(t\mapsto(1-t^2)/(2t)\) has derivative \(-(1+t^{-2})/2<0\) and limits
\(+\infty\) and \(-\infty\) at the two ends. It is therefore a
homeomorphism onto \(\R\). Since \(s\ne0\), the coordinate
\(\alpha_1=(r-\alpha_2)/s\) also maps this arc homeomorphically onto
\(\R\). Every nonempty open interval of \(\alpha_1\)-values thus has a
nonempty open preimage in \(C\), containing infinitely many of the
rational points \(\iota(2nR)\). Their \(\alpha_1\)-values are distinct
and belong to \(\mathcal D_{r,s}\). This proves the claimed density.

Finally, both radicands \(1+\alpha_i^2\) are strictly positive on the
finite real detour line. Each fixed choice of signs defines a continuous
graph over that line, with continuous inverse given by projection to
\(\alpha_1\). At every rational detour solution both square roots are
rational, independently of their chosen signs. The density just proved
therefore gives density on each of the four finite real sign branches.
\end{proof}

\begin{proposition}[Generic endpoint-line theorem]\label{q1:prop:generic}
Assume the normalized endpoint-line setup. If
\[
r\ne0,\qquad s\ne\pm1,\qquad
4r^2s\ne\pm(1-s^2)^2,
\]
then there are infinitely many rational points \(P\) on the open segment
\(UV\) such that \(PA,PB\in\Q\).
\end{proposition}

\begin{proof}
Parts~(\ref{q1:part:love-torsion}) and~(\ref{q1:part:love-density}) of
Theorem~\ref{q1:thm:love} make
$\{\alpha\in\calS:r-s\alpha\in\calS\}$ dense in $\R$, hence infinite
in the nonempty interval $I$. Apply the bijection of
Lemma~\ref{q1:lem:detour}.
\end{proof}

\subsection{The shared equal-height model}

For $T\in\Q^\times$, both the symmetric convex case and the exceptional
endpoint case $s=-1$, $r=-T$, require
$\alpha,\alpha-T\in\calS$. We establish their common model and torsion
input before treating the geometric consequences.

\begin{lemma}[The symmetric detour curve]\label{q1:lem:symmetric-detour-curve}
Assume \(T\in\Q^\times\). For \(\alpha\in\R\), the simultaneous rational-square conditions
\[
1+\alpha^2\in(\Q^\times)^2,
\qquad
1+(\alpha-T)^2\in(\Q^\times)^2
\]
are represented by rational points on
\begin{equation}\label{q1:eq:gamma-T}
\Gamma_T:\quad (t-u)(tu+1)=2Ttu,
\end{equation}
through
\[
\alpha=\frac{t-t^{-1}}2,
\qquad
\alpha-T=\frac{u-u^{-1}}2,
\qquad t,u\in\Q^\times.
\]
The curve \(\Gamma_T\) is birational over \(\Q\) to
\begin{equation}\label{q1:eq:E-T}
E_T:\quad Y^2=X^3+(T^2+2)X^2+X.
\end{equation}
The birational map \(\phi:\Gamma_T\dashrightarrow E_T\) is
\begin{equation}\label{q1:eq:symmetric-map}
X=-1-\frac{2T}{t-u-2T},
\qquad
Y=\frac{T(t+u)}{t-u-2T},
\end{equation}
and the inverse map, on \(X\ne-1\), is
\begin{equation}\label{q1:eq:symmetric-inverse}
t=\frac{TX-Y}{X+1},
\qquad
u=-\frac{TX+Y}{X+1}.
\end{equation}
More precisely, these maps restrict to an isomorphism from the affine
curve \(\Gamma_T\) to the affine part of \(E_T\) with \(X\ne-1\).
The locus \(t,u\ne0\) corresponds exactly to
\[
E_T\setminus\{\calO,(0,0),(-1,T),(-1,-T)\}.
\]
\end{lemma}

\begin{proof}
The two square conditions force $\alpha\in\Q$, since their difference
plus $T^2$ is $2T\alpha$. Parametrize each by
Lemma~\ref{q1:lem:S}; subtracting yields \eqref{q1:eq:gamma-T}.
Conversely the two parametrizations give
$1+\alpha^2=((t+t^{-1})/2)^2$ and the analogous identity for $u$.
The discriminant of $E_T$ is $16T^2(T^2+4)\ne0$. Substitution gives
\begin{equation*}
Y^2-X^3-(T^2+2)X^2-X
=\frac{4T^2((t-u)(tu+1)-2Ttu)}{(t-u-2T)^3}.
\end{equation*}
On $\Gamma_T$, the denominator cannot vanish: $t-u=2T$ would give
$2T(tu+1)=2Ttu$. The inverse formulas are checked by substitution
and give $tu=(Y^2-T^2X^2)/(X+1)^2=X$. Thus $t,u\ne0$ is exactly
$X\ne0$ in this chart. The only points with $X=0$ or $-1$ are
$(0,0)$ and $(-1,\pm T)$; together with infinity these are precisely
the omitted points. The two open curves are isomorphic, over $\Q$
and real analytically over $\R$.
On the smooth projective completion of $\Gamma_T$, we transport the group
law from $E_T$. In homogeneous coordinates $(t:u:v)$, the origin is
$(1:1:0)$: the inverse formulas give $t/u\to1$ and $1/u\to0$
at $\calO$. The recovered detour parameter
is $\alpha=(t-t^{-1})/2$. The endpoint lifts on this model are the points
whose torsion is tested in the later applications.
\end{proof}

\begin{lemma}[Torsion with an irrational second diagonal]
\label{q1:lem:symmetric-torsion}
Assume \(T\in\Q^\times\) and $T^2+4\notin(\Q^\times)^2$. Then
\begin{equation*}
E_T(\Q)_{\mathrm{tors}}
=\{\calO,\ (0,0),\ (-1,T),\ (-1,-T)\}.
\end{equation*}
\end{lemma}

\begin{proof}
The point $(-1,T)$ doubles to $(0,0)$ (its tangent slope is $-T$),
so it has order $4$. The other $2$-torsion abscissae solve
$X^2+(T^2+2)X+1=0$, whose discriminant is $T^2(T^2+4)$.
Under the nonsquare hypothesis, Mazur's theorem leaves only cyclic
orders $4,8,12$. Two calculations, valid for every $T\ne0$, exclude
the latter possibilities.

First, for a possible half $R=(X,Y)$ of $(-1,\pm T)$, one has
$Y\ne0$. Lemma~\ref{lem:common-duplication} with $B=1$ gives
\begin{equation*}
X(2R)+1=\frac{(X+1)^4+4T^2X^2}{4Y^2}.
\end{equation*}
The numerator is strictly positive for every real $X$: its vanishing
would require both $X=-1$ and $TX=0$, impossible for $T\ne0$.
Thus neither $(-1,T)$ nor $(-1,-T)$ has even a real half, excluding
cyclic order $8$.
Second, rational $3$-torsion would give a rational root of
\begin{equation*}
\psi_3(X)=3X^4+4(T^2+2)X^3+6X^2-1,
\end{equation*}
equivalently
\begin{equation}\label{q1:eq:symmetric-psi3-rearranged}
4T^2X^3=(1-3X)(X+1)^3.
\end{equation}
Neither $X=0$ nor $-1$ is a root. Put $Z=3X$,
$W=6TX^2/(X+1)$; then
\begin{equation*}
W^2=Z(1-Z)(Z+3).
\end{equation*}
The substitution $x=-Z-1$, $y=W$ gives
\begin{equation}\label{q1:eq:symmetric-C}
C:\quad y^2=(x+2)(x+1)(x-2).
\end{equation}
This is LMFDB $48.a4$, Cremona $48a1$, of rank zero and torsion
$(\Z/2\Z)^2$~\cite{Cremona,LMFDB48a4}. Consequently
$C(\Q)=\{\calO,(-2,0),(-1,0),(2,0)\}$. Its affine pullbacks give
$X=1/3,0,-1$; the latter two are excluded, while $X=1/3$ in
\eqref{q1:eq:symmetric-psi3-rearranged} forces $T=0$. This excludes
$3$-torsion and cyclic order $12$, proving the stated group.
We will also use the two exclusions just established without the
nonsquare hypothesis.
\end{proof}

\subsection{The exceptional cases}
\label{q1:sec:exceptional}

For $s<0$, failure of the generic hypotheses means exactly
\begin{equation}\label{q1:eq:exceptional-list}
r=0,\qquad s=-1,\qquad 4r^2s=-(1-s^2)^2.
\end{equation}
The other sign is possible only at the overlap $(r,s)=(0,-1)$.
Table~\ref{tab:endpoint-cases} separates that overlap and gives a disjoint
route through the endpoint-line proof. In the nonsingular exceptional
rows, a non-torsion positive endpoint lift yields infinitude by the
endpoint-density lemma (Lemma~\ref{q1:lem:density}); the torsion alternatives are handled in the
cited propositions.

\begin{table}[htbp]
\centering\small
\begin{tabular}{@{}L{.18\linewidth}L{.39\linewidth}L{.35\linewidth}@{}}
\toprule
Case & Conditions in the normalized setup & Argument and conclusion \\
\midrule
Generic & $r\ne0$, $s\ne-1$, $4r^2s\ne-(1-s^2)^2$
& Proposition~\ref{q1:prop:generic}: infinitude. \\
Singular overlap & $r=0$, $s=-1$
& The square conditions coincide; Lemma~\ref{q1:lem:S} gives infinitude. \\
Aligned & $r=0$, $s\ne-1$
& Proposition~\ref{q1:prop:r-zero}: existence via the common-leg model. \\
Equal heights & $r\ne0$, $s=-1$
& Proposition~\ref{q1:prop:s-minus-one}: existence via the symmetric model. \\
Quadratic exceptional & $r\ne0$, $s\ne-1$, $4r^2s=-(1-s^2)^2$
& Proposition~\ref{q1:prop:third}: infinitude; Appendix~\ref{q1:app:caseIII}
controls the torsion values. \\
\bottomrule
\end{tabular}
\caption{The endpoint-line proof for $s<0$. The rows are mutually exclusive
and exhaustive.}\label{tab:endpoint-cases}
\end{table}
\FloatBarrier

All group laws on affine genus-one models below refer to their smooth
projective completions and the specified origins. Endpoint lifts use
positive square roots and positive detour parameters, so the endpoint
lengths make them rational. The following lemma will be applied to
these lifts.

\begin{lemma}[Density from one non-torsion endpoint]\label{q1:lem:density}
Let \(\Gamma\) be a smooth projective genus-one curve over \(\Q\) with
a chosen rational origin, and let \(\alpha\in\Q(\Gamma)\) be a
nonconstant rational function. Let \(\lambda<\mu\) be rational numbers.
Suppose that a real analytic branch \(Q(v)\in\Gamma(\R)\) is defined
on an open interval containing \([\lambda,\mu]\), that \(\alpha\) is
regular along this branch, and that \(\alpha(Q(v))=v\).
Assume that \(Q(\lambda),Q(\mu)\in\Gamma(\Q)\), with at least one of
these two points non-torsion for the chosen group law.
Then infinitely many distinct rational values of \(\alpha\) in
\((\lambda,\mu)\) are attained by points of \(\Gamma(\Q)\).
\end{lemma}

\begin{proof}
Let $R$ be a non-torsion endpoint lift. Its even multiples are dense
in the identity component by Lemma~\ref{q2:lem:real-density}; if $R$
is on the other component, its odd multiples are dense there.
Thus rational multiples are dense in the real component of $R$.
Differentiating $\alpha(Q(v))=v$ shows that $Q$ is an immersion,
with local inverse $\alpha$. A small one-sided interval inside
$(\lambda,\mu)$ therefore maps to a nonempty open arc adjacent to $R$
in its real component.
It contains infinitely many rational multiples, with distinct rational
$\alpha$-values because $\alpha$ is defined over $\Q$ and injective
on the arc.
\end{proof}

\subsubsection{\texorpdfstring{The aligned case: \(r=0\)}{The aligned case: r=0}}

\begin{proposition}[The aligned case]\label{q1:prop:r-zero}
Assume the normalized endpoint-line setup and \(r=0\). Then the open segment
\(UV\) contains a rational point \(P\) with \(PA,PB\in\Q\). Writing
\(\rho=-s\), if \(\rho=1\), then there are infinitely many such points. If
\(\rho\ne1\) and the endpoint lifts are not both torsion on the associated
genus-one curve, then there are infinitely many such points.
\end{proposition}

\begin{proof}
Put $\rho=-s>0$. The detour conditions are
\begin{equation}\label{q1:eq:r-zero-condition}
\alpha_1\in\calS,\qquad \rho\alpha_1\in\calS.
\end{equation}
If $\rho=1$, density of $\calS$ and Lemma~\ref{q1:lem:detour}
give infinitude. Assume $\rho\ne1$.

For
\(h,k\in\Q_{>0}\) with \(h\ne k\), consider the smooth intersection of quadrics
\begin{equation}\label{eq:common-leg-quadrics}
\mathcal Q_{h,k}:\quad X_0^2+h^2X_1^2=X_2^2,\qquad
X_0^2+k^2X_1^2=X_3^2.
\end{equation}
Here the coordinates are homogeneous in $\mathbb P^3$. Neither equation
gradient vanishes on the intersection. A dependence between them would
force $X_2=X_3=0$, and then $h^2\ne k^2$ forces $X_1=X_0=0$,
which is impossible in projective space. Thus the intersection is smooth.
In the chart \(X_1\ne0\), put \(z=X_0/X_1\). Thus the two square
conditions are \(z^2+h^2=Z^2\) and \(z^2+k^2=W^2\).
The Selder--Spindler isomorphism \cite[Theorem~2]{SelderSpindler}
identifies this curve with
\[
E_{h,k}:\quad y^2=x(x+h^2)(x+k^2).
\]
The cited theorem is stated for distinct nonzero integer coefficients.
For rational \(h,k\), choose \(N\in\Z_{>0}\) with \(Nh,Nk\in\Z\).
The changes \((z,Z,W)\mapsto(Nz,NZ,NW)\) on the quadric chart and
\((x,y)\mapsto(N^2x,N^3y)\) on the elliptic curve reduce to the integer
case; their rational inverses give the stated isomorphism over \(\Q\).
The model sends the quadric origin $(1:0:1:1)$ to elliptic infinity.
Its recovered coordinate is $z=X_0/X_1$, and the rational endpoint lifts
are the points tested for torsion. At an affine point with $y\ne0$,
the inverse gives
\begin{equation}\label{eq:common-leg-coordinate}
z=\frac{h^2k^2-x^2}{2y}.
\end{equation}
The point at infinity and the three points of order \(2\) correspond to
\((1:0:\pm1:\pm1)\), so they give no finite value of \(z\).

The curve has full rational \(2\)-torsion. The point
$(hk,hk(h+k))$ doubles to $(0,0)$ by
Lemma~\ref{lem:common-duplication} with $B=h^2k^2$, so has order $4$. Mazur's theorem therefore leaves only
\(\Z/2\Z\times\Z/4\Z\) and \(\Z/2\Z\times\Z/8\Z\) as possible
rational torsion groups.
For every point \(P\) of order \(4\), its double belongs to the real
identity component and hence must be \((0,0)\): the other two nonzero
points of order \(2\) lie on the bounded real component, whose
\(x\)-coordinates range between \(-\max(h^2,k^2)\) and
\(-\min(h^2,k^2)\). Thus
\(x(P)=\pm hk\), and the inverse formula gives \(z=0\).
If \(P\) has order \(8\), then \(2P\) has order \(4\) and lies on the
identity component. Consequently \(x(2P)=hk\), because the order-\(4\)
points with \(x=-hk\) lie on the bounded component. The same duplication
formula now gives \(z^2=hk\).
It follows that every finite torsion value is contained in
\begin{equation}\label{eq:common-leg-torsion}
\{0\}\ \cup\ \{z\in\Q:z^2=hk\}.
\end{equation}
This is a bound on the possible values; it does not assert that the
nonzero pair occurs whenever \(hk\) is a rational square. In particular,
there is at most one possible finite torsion value on each side of \(0\).

Apply this calculation with $h=1$, $k=1/\rho$, $z=\alpha_1$.
The conditions \eqref{q1:eq:r-zero-condition} are exactly those defining
the rational points of this quadric chart. The branch
\[
Q(\alpha)=\left(\alpha:1:\sqrt{1+\alpha^2}:
\sqrt{\rho^{-2}+\alpha^2}\right)
\]
is real analytic and nonsingular: the two square-root coordinates are
positive, so the Jacobian minor in those coordinates is nonzero.
The regular coordinate $X_0/X_1=\alpha$ is its inverse, and rational
endpoint distances make $Q(\lambda),Q(\mu)$ rational. Transport the group
law from $E_{1,1/\rho}$. If either endpoint lift is non-torsion, the
endpoint-density lemma (Lemma~\ref{q1:lem:density}) and the detour bijection
give infinitude.
Otherwise neither endpoint value is zero, by nonincidence, and
\eqref{eq:common-leg-torsion} gives at most one value on each side
of zero. Thus $\lambda<0<\mu$, and $\alpha_1=\alpha_2=0$
gives the required interior solution.

For comparison, write $Z=X_2/X_1$, $Z'=X_3/X_1$. The substitution
$t=\alpha_1+Z$, $W=2tZ'$ identifies this chart with the open part
$t\ne0$ of the Jacobi quartic
\begin{equation*}
C_\rho:\quad W^2=t^4+\left(\frac4{\rho^2}-2\right)t^2+1,
\end{equation*}
whose polynomial discriminant is $4096(1-\rho^2)^2/\rho^8\ne0$.
The inverse is $\alpha_1=(t-t^{-1})/2$, $Z=(t+t^{-1})/2$,
$Z'=W/(2t)$; nonvanishing of $t$ follows from
$(Z+\alpha_1)(Z-\alpha_1)=1$.
The scaling $X=\rho^2x$, $Y=\rho^3y$ gives its elliptic model
\begin{equation}\label{q1:eq:r-zero-elliptic}
E_\rho:\quad Y^2=X(X+\rho^2)(X+1).
\end{equation}
Here the origin is transported from $(1:0:1:1)$ on the common-leg curve,
and $\alpha_1=(t-t^{-1})/2$ is the recovered detour parameter. The positive
endpoint lifts are preserved. Their possible finite torsion values are
$0$ and, only when $\rho$ is a rational square, $\pm1/\sqrt\rho$.
\end{proof}

\subsubsection{\texorpdfstring{The equal-height case: \(s=-1\)}{The equal-height case: s=-1}}

\begin{proposition}[The opposite equal-height case]\label{q1:prop:s-minus-one}
Assume the normalized endpoint-line setup, \(s=-1\), and \(r\ne0\). Then
the open segment \(UV\) contains a rational point \(P\) with \(PA,PB\in\Q\).
If the endpoint lifts are not both torsion on the associated genus-one
curve, then there are infinitely many such points.
\end{proposition}

\begin{proof}
Put $T=-r\in\Q^\times$. The detour equation becomes $\alpha_2=\alpha_1-T$.
Thus we need
\begin{equation*}
\alpha_1\in\calS,\qquad \alpha_1-T\in\calS.
\end{equation*}
The special value is
\[
\tau=\frac T2=-\frac r2.
\]
Lemma~\ref{q1:lem:symmetric-detour-curve} identifies these conditions
with the nonzero-parameter locus of
\begin{equation*}
\Gamma_T:\quad (t-u)(tu+1)=2Ttu,
\end{equation*}
where $\alpha_1=(t-t^{-1})/2$, and gives the nonsingular model
\begin{equation*}
E_T:\quad Y^2=X^3+(T^2+2)X^2+X.
\end{equation*}
We use the inverse formulas
\begin{equation}\label{q1:eq:s-minus-one-inverse}
t=\frac{TX-Y}{X+1},\qquad u=-\frac{TX+Y}{X+1},
\end{equation}
together with the forward map and the precise omitted-point description
proved in that lemma.

The proof of Lemma~\ref{q1:lem:symmetric-torsion} shows, for every
$T\ne0$, that $(-1,\pm T)$ have order $4$ and no rational halves,
and that there is no rational $3$-torsion. Mazur's theorem therefore gives
\[
E_T(\Q)_{\mathrm{tors}}\in
\{\Z/4\Z,\ \Z/2\Z\times\Z/4\Z,\ \Z/2\Z\times\Z/8\Z\}.
\]
If $T^2+4$ is not a rational square, Lemma~\ref{q1:lem:symmetric-torsion}
and the omitted-point description give no finite torsion detour values.
Otherwise $w=\sqrt{T^2+4}\in\Q$ and $w>2$. The possible finite torsion
detour values are listed below; the order-$8$ rows require rational
halves of the indicated points.
\[
\renewcommand{\arraystretch}{1.15}
\begin{array}{@{}clc@{}}
\toprule
\text{Order}&\text{Points or condition}&\text{Detour value }\alpha_1\\
\midrule
2&Y=0,\ X^2+(T^2+2)X+1=0&T/2\\
4&(1,\pm w)&T/2\\
8&2R=(1,w)\quad\text{(if rational halves exist)}&(T-w)/2\\
8&2R=(1,-w)\quad\text{(if rational halves exist)}&(T+w)/2\\
\bottomrule
\end{array}
\]
The remaining torsion points $\calO,(0,0),(-1,\pm T)$ give no
finite nonzero-parameter detour value. For the first two rows,
substitution in \eqref{q1:eq:s-minus-one-inverse} gives $t-t^{-1}=T$.
The four points $(-1,\pm T),(1,\pm w)$ exhaust order $4$ in either
possible group with full rational $2$-torsion. Thus every order-$8$
point is a half of $(1,\pm w)$, since $(-1,\pm T)$ have no rational
halves.

For completeness, $X(2R)=1$ factors as
\[
\bigl(X^2-2(1+w)X+1\bigr)
\bigl(X^2-2(1-w)X+1\bigr)=0.
\]
For a root of the first quadratic, set $Y=(w+1)X-1$; for a root
of the second, set $Y=(w-1)X+1$. Substitution in the curve equation
and both duplication coordinates gives $2R=(1,w)$.
Since $w>2$, the quadratics have four distinct roots, with $X\ne0,-1$
and $Y\ne0$. The inverse map gives $(T-w)/2$ in both cases.
These are all halves of $(1,w)$:
for each root the two possible ordinates are opposites, and their
doubles are $(1,w)$ and $(1,-w)$, respectively.
Negating $R$ changes $t$ to $-u$
and $\alpha_1$ to $T-\alpha_1$, proving the final row.
In particular there is at most one finite torsion value on each side
of $\tau=T/2$.

The positive branch $t=\vartheta(\alpha)$,
$u=\vartheta(\alpha-T)$ satisfies the hypotheses of the endpoint-density
lemma (Lemma~\ref{q1:lem:density}) by \eqref{q1:eq:positive-lift} and the shared model. A non-torsion
endpoint lift therefore gives infinitude. If both endpoint lifts are torsion,
nonincidence excludes the central value, so they must be the pair
$(T-w)/2,(T+w)/2$. Their midpoint $T/2$ lies in $I$ and
$1+(T/2)^2=(w/2)^2$. Taking $\alpha_1=T/2$, $\alpha_2=-T/2$
and applying Lemma~\ref{q1:lem:detour} proves existence.
\end{proof}

\subsubsection{The quadratic exceptional case}

\begin{proposition}[The quadratic exceptional identity]\label{q1:prop:third}
Assume the normalized endpoint-line setup,
\[
r\ne0,\qquad s\ne-1,\qquad
4r^2s=-(1-s^2)^2.
\]
Then the open segment \(UV\) contains infinitely many rational points
\(P\) with \(PA,PB\in\Q\).
\end{proposition}

\begin{proof}
The exceptional identity and the assumptions \(s<0\), \(r\ne0\) imply
\[
-s=\left(\frac{1-s^2}{2r}\right)^2.
\]
Thus \(-s\) is a positive rational square. Write
\[
s=-k^2,\qquad k\in\Q_{>0}.
\]
Since \(s\ne-1\), we have \(k\ne1\). The exceptional identity gives
\[
r=\varepsilon\frac{1-k^4}{2k},
\qquad \varepsilon\in\{\pm1\}.
\]
The reflection $(x,y)\mapsto(D-x,y)$, followed by interchanging
the endpoint labels, preserves the hypotheses and sends
$(r,\alpha_1,\alpha_2,\tau,I)$ to
$(-r,-\alpha_1,-\alpha_2,-\tau,(-\mu,-\lambda))$.
It therefore suffices to take $r=(1-k^4)/(2k)$. Then
\begin{equation*}
\tau=\frac r{s-1}=\frac{k^2-1}{2k}\in\calS,\qquad r-s\tau=-\tau.
\end{equation*}
Thus $\tau\in I$ gives an elementary solution. We prove infinitude
regardless of whether $\tau$ lies in $I$.

For infinitude, use the following model and torsion calculation.
Parametrize
\[
\alpha_1=\frac{t-t^{-1}}2.
\]
The condition \(\alpha_2=r-s\alpha_1\in\calS\) gives the quartic
\begin{equation}\label{q1:eq:third-quartic}
\begin{aligned}
C_k:\quad W^2={}&k^6t^4+2k^3(1-k^4)t^3\\
&+(k^8-2k^6-2k^4+4k^2+1)t^2\\
&+2k^3(k^4-1)t+k^6.
\end{aligned}
\end{equation}
The polynomial on the right can also be written as
\[
\bigl(k^3t^2+(1-k^4)t-k^3\bigr)^2+4k^2t^2.
\]
Its positive branch is
\[
t=\vartheta(\alpha),\qquad
W=2k\,\vartheta(\alpha)\sqrt{1+(r-s\alpha)^2}.
\]
The affine curve is nonsingular along this branch, since the partial
derivative of $W^2-f_k(t)$ with respect to $W$ is $2W>0$, where
$f_k$ denotes the displayed quartic polynomial. The convention
\eqref{q1:eq:positive-lift} verifies regularity of the recovered
parameter and rationality of the endpoint lifts.
The special point \(\alpha_1=\tau\) corresponds to
\[
O_k=(t,W)=\bigl(k,\ k(k^2+1)\bigr).
\]
Using \(O_k\) as the origin, \eqref{q1:eq:third-quartic} is birational to
\begin{equation}\label{q1:eq:third-elliptic}
\begin{aligned}
E_k:\quad
Y^2&-2(k^4-2k^2-1)XY+8k^6(k^4-1)Y\\
&=X^3-4k^4(k^2+1)X^2.
\end{aligned}
\end{equation}
Appendix~\ref{q1:app:caseIII} proves nonsingularity, gives the
birational maps with this origin and their extensions at every omitted
point, and establishes the complete torsion group
\[
 E_k(\Q)_{\mathrm{tors}}=\langle(0,0)\rangle\cong\Z/8\Z.
\]
With $O_k$ corresponding to elliptic infinity, the parameter to recover
is $\alpha_1=(t-t^{-1})/2$ and the torsion generator is $T=(0,0)$.
For affine points with $Y\ne0$, the inverse first recovers $t$ by
\begin{equation}\label{q1:eq:third-t-map}
t=k+\frac{2k(k^2+1)X}{Y}.
\end{equation}
The appendix's final table, obtained from this formula and the omitted-point
charts, shows that the only finite detour value represented by rational
torsion is $\alpha_1=\tau$. In particular, every rational lift with
finite $\alpha_1\ne\tau$ is non-torsion.

Neither endpoint value equals \(\tau\), by the nonincidence hypothesis.
Hence both endpoint lifts are non-torsion. The positive branch just
constructed satisfies the hypotheses of the endpoint-density lemma
(Lemma~\ref{q1:lem:density}),
which gives infinitely many rational detour
solutions with \(\alpha_1\in I\), and Lemma~\ref{q1:lem:detour} gives the
desired points \(P\) on \(UV\).
\end{proof}

\begin{lemma}[Diagonal criterion for simple quadrilaterals]
\label{q1:lem:diagonal-criterion}
Let \(Q\) be a simple non-degenerate quadrilateral with vertices
$P_1,P_2,P_3,P_4$ in boundary order. Then \(Q\) is convex if and only if the two diagonals
$P_1P_3,\ P_2P_4$ meet in their relative interiors.
\end{lemma}

\begin{proof}
For a convex quadrilateral, each diagonal separates the other two
vertices, so the diagonals cross internally. Conversely, four extreme
vertices of a simple quadrilateral must occur in cyclic hull order;
any other order crosses two sides. A nonconvex simple quadrilateral
therefore has one vertex strictly inside the triangle of the other
three. Its diagonal to the opposite vertex stays inside that triangle
except at its endpoint and cannot meet the opposite side, which is
the other diagonal. Non-degeneracy excludes collinear boundary cases.
\end{proof}

\subsubsection{The concave torsion alternative}

The phrase ``both endpoint lifts are torsion'' refers to the two positive
endpoint lifts fixed above, one over each endpoint parameter of \(I\).
This auxiliary condition can occur in the
convex symmetric parallelogram with rational other diagonal, but it does not
occur in the simple non-degenerate concave application.

\begin{proposition}[Exclusion of the concave torsion alternative]
\label{q1:prop:no-concave-both-torsion}
Assume the normalized endpoint-line setup. Suppose that the vertices
\(U,A,V,B\), in this boundary order, form a simple non-degenerate
concave quadrilateral \(Q\) with \(UV\) as its interior diagonal. Thus
\[
U=(0,0),\qquad V=(D,0),\qquad
A=(a,b),\qquad B=(c,d),\qquad b>0>d.
\]
In each of the exceptional endpoint-line cases
\[
\begin{gathered}
r=0;\\
r\ne0,\quad s=-1;\\
r\ne0,\quad s\ne-1,\quad 4r^2s=-(1-s^2)^2,
\end{gathered}
\]
the set \(\mathcal R^\circ(U,V;A,B)\) is infinite.
In the nonsingular subcases, at least one of the two positive endpoint
lifts on the corresponding auxiliary genus-one curve is non-torsion.
\end{proposition}

\begin{proof}
The diagonals of a simple non-degenerate concave quadrilateral do
not meet in their relative interiors
(Lemma~\ref{q1:lem:diagonal-criterion}). Since $A,B$ are on opposite sides of the line $UV$, the segment
$(A,B)$ meets that line exactly once, at the point with parameter
$\tau$. The diagonal criterion excludes $\tau\in(\lambda,\mu)$,
and endpoint nonincidence excludes $\tau=\lambda,\mu$. Thus both endpoints lie on the same side of $\tau$.

For $r=0$, $-s\ne1$, Proposition~\ref{q1:prop:r-zero} gives at most
one finite torsion value on each side of $\tau=0$. For $s=-1$,
$r\ne0$, Proposition~\ref{q1:prop:s-minus-one} gives at most one
on each side of $\tau=-r/2$. The two distinct endpoint lifts cannot both
be torsion in either case. For the quadratic identity,
Appendix~\ref{q1:app:caseIII} gives only the finite torsion value
$\tau$, so neither endpoint lift is torsion. The established positive
branches and the endpoint-density lemma (Lemma~\ref{q1:lem:density}) give
infinitude in each
nonsingular case. At $(r,s)=(0,-1)$ the conditions coincide and
density of $\calS$ gives infinitude directly.
\end{proof}

\paragraph{Completion of the exceptional analysis.}
Propositions~\ref{q1:prop:r-zero}, \ref{q1:prop:s-minus-one},
and~\ref{q1:prop:third} prove existence in the three alternatives of
\eqref{q1:eq:exceptional-list}, which exhaust the exceptional cases because
$s<0$. They also give infinitude whenever at least one positive endpoint
lift is non-torsion. In the singular overlap $(r,s)=(0,-1)$, the two square
conditions coincide, and Lemma~\ref{q1:lem:S} gives infinitude directly.

\begin{proof}[Proof of the endpoint-line theorem]
By Lemma~\ref{q1:lem:normalization} and a reflection if needed, reduce to
the normalized setup of Section~\ref{q1:sec:detour}. If the
generic hypotheses hold, Proposition~\ref{q1:prop:generic} gives the result.
If not, since \(s<0\), one of the exceptional alternatives in
\eqref{q1:eq:exceptional-list} holds. Proposition~\ref{q1:prop:r-zero}
applies when $r=0$; Proposition~\ref{q1:prop:s-minus-one} applies when
$r\ne0$ and $s=-1$; and Proposition~\ref{q1:prop:third} applies in the
remaining exceptional case. In every case, the resulting point \(P=(x,0)\) has
\(x\in\Q\), \(0<x<D\), and \(PA,PB\in\Q\). Hence \(PU=x\) and
\(PV=D-x\) are rational as well. Transforming back by the inverse rational
isometry preserves rational coordinates, all distances, and the open
segment \(UV\), proving the theorem.
\end{proof}

\subsection{Completion of the quadrilateral theorem}
\label{q1:sec:application}

We now deduce Theorem~\ref{q1:thm:main} from the endpoint-line theorem,
whose proof was completed in Section~\ref{q1:sec:exceptional}. The only
remaining geometric step is to verify the opposite-side hypothesis for
an interior diagonal. The resulting argument applies simultaneously to
convex and concave quadrilaterals.

\begin{lemma}[Geometry of an interior diagonal]\label{q1:lem:opposite-sides}
Let \(Q\) be a simple non-degenerate quadrilateral, convex or concave, and let
\([U,V]\) be an interior diagonal. If \(A\) and \(B\) are the other two
vertices, then \(A\) and \(B\) lie on opposite sides of the line through
\(U,V\).
\end{lemma}

\begin{proof}
As in Lemma~\ref{q1:lem:quad-diagonal}, the interior diagonal cuts
the polygon into two non-degenerate triangles with disjoint interiors.
Their third vertices $A,B$ therefore lie on opposite sides of its line.
Indeed, two non-degenerate triangles sharing a base and having their
third vertices on the same side overlap near every interior point of
that base.
\end{proof}

\begin{proof}[Proof of Theorem~\ref{q1:thm:main}]
The line through the rational points $U,V$ is rational. Its four
endpoint distances are the side lengths, and $UV$ is rational by
hypothesis. Lemma~\ref{q1:lem:opposite-sides} gives the opposite-side
condition, while non-degeneracy excludes either endpoint from $AB$.
Apply Theorem~\ref{q1:thm:endpoint-line}; its four distances are positive
because the point is on $(U,V)$ and $A,B$ are off its line.
Interiority follows from the chosen-diagonal hypothesis.
\end{proof}

\section{Convex refinements and finite examples}
\label{q1:sec:convex-exceptional}

Use the normalized coordinates $U=(0,0)$, $V=(D,0)$, $A=(a,b)$,
$B=(c,d)$ with $D,b>0>d$, all coordinates rational, and
$s=b/d$, $r=(c-a)/d$. Throughout this section $UAVB$ is simple,
non-degenerate, and convex. Thus its two diagonals meet in their
relative interiors. Rational side lengths are imposed only where
stated; the first two elementary constructions do not require them.

Table~\ref{tab:convex-refinements} records the infinitude refinements and the
finite cases; the normalized coefficients are those fixed above.
\begin{table}[htbp]
\centering\small
\begin{tabular}{@{}L{.32\linewidth}L{.36\linewidth}L{.24\linewidth}@{}}
\toprule
Convex configuration & Conclusion & Result \\
\midrule
Quadratic exceptional identity & Infinitely many diagonal points.
& Proposition~\ref{q1:prop:convex-third-infinite} \\
Equal heights, non-symmetric & Infinitely many diagonal points.
& Proposition~\ref{q1:prop:convex-s-minus-one-nonsym-infinite} \\
Symmetric parallelogram, irrational other diagonal & Infinitely many diagonal points.
& Proposition~\ref{q1:prop:convex-symmetric} \\
Symmetric parallelogram, rational other diagonal & Existence; for $T\ne0$,
finiteness is equivalent to rank zero of the associated curve.
& Propositions~\ref{q1:prop:convex-symmetric}, \ref{q1:prop:rank-zero} \\
Aligned kite; symmetric parallelogram examples & Exactly one diagonal point
in each displayed example.
& Theorems~\ref{q1:thm:finite-aligned}, \ref{q1:thm:finite-example} \\
\bottomrule
\end{tabular}
\caption{Convex refinements under the rational-side hypotheses. For the
symmetric configuration, $c=D-a$, $d=-b$, and $T=(D-2a)/b$.}
\label{tab:convex-refinements}
\end{table}

\subsection{Elementary constructions and infinitude}

\begin{proposition}[Elementary convex constructions]
\label{q1:prop:convex-r-zero}\label{q1:prop:convex-third-exception}
Assume the normalized convex quadrilateral setup. The open segment $UV$
contains a rational point at rational distance from $U,V,A,B$ in
either of the following cases, without a rational-side assumption:
\begin{enumerate}
\item $r=0$;
\item the exceptional identity
\begin{equation*}
r\ne0,\qquad s\ne-1,\qquad 4r^2s=-(1-s^2)^2.
\end{equation*}
\end{enumerate}
\end{proposition}

\begin{proof}
Use the diagonal intersection, which lies in $(U,V)$ by convexity.
In the first case it is $(a,0)$; convexity gives $0<a<D$, so its
distances $a,D-a,b,-d$ are positive rational numbers.
In the second case its parameters are
\begin{equation*}
\alpha_1=\frac r{s-1},\qquad\alpha_2=-\alpha_1.
\end{equation*}
Since $s<0$ and $s\ne-1$, the rational number $t=(1-s^2)/(2r)$
is nonzero. It satisfies $t^2=-s$ and
$(t-t^{-1})/2=r/(s-1)$. Either sign of $t$ is permitted in the
parametrization of $\calS$, so $\alpha_1,\alpha_2\in\calS$,
and Lemma~\ref{q1:lem:detour} applies.
\end{proof}

\begin{proposition}[Infinitude in the quadratic exceptional case]
\label{q1:prop:convex-third-infinite}
Assume the normalized convex quadrilateral setup
\[
\begin{gathered}
U=(0,0),\qquad V=(D,0),\\
A=(a,b),\qquad B=(c,d),\qquad b>0>d,
\end{gathered}
\]
with
\[
s=\frac bd,\qquad r=\frac{c-a}{d}.
\]
Suppose
\[
r\ne0,\qquad s\ne-1,\qquad
4r^2s=-(1-s^2)^2.
\]
Assume also that the side lengths $UA,\ AV,\ VB,\ BU$ are rational.
Then the open diagonal segment \(UV^\circ\) contains infinitely many
rational points whose distances from \(U,V,A,B\) are rational.
\end{proposition}

\begin{proof}
Rational sides give the endpoint-distance hypotheses, and
non-degeneracy excludes either endpoint from the line $AB$.
The opposite-side condition holds because $b>0>d$.
Proposition~\ref{q1:prop:third} therefore applies directly.
\end{proof}

\begin{proposition}[Non-symmetric equal-height infinitude]
\label{q1:prop:convex-s-minus-one-nonsym-infinite}
Assume the normalized convex quadrilateral setup
\[
U=(0,0),\qquad V=(D,0),\qquad
A=(a,b),\qquad B=(c,-b),\qquad b>0.
\]
Assume $a+c\ne D$. Assume also that the side lengths $UA,\ AV,\ VB,\ BU$ are rational.
Then the open diagonal segment \(UV^\circ\) contains infinitely many
rational points whose distances from \(U,V,A,B\) are rational.
\end{proposition}

\begin{proof}
For an elementary point put $h=a+c$. The diagonal intersection is
$(h/2,0)$, so convexity gives $0<h<2D$.
Choose $e=0$ if $h<D$, and $e=D$ if $h>D$. Then $x=h-e$ lies in
$(0,D)$ and
\[
P_0=(x,0),\quad E=(e,0),\qquad P_0A^2=EB^2,\quad P_0B^2=EA^2.
\]
Thus $P_0$ has all four rational distances.

For infinitude put $T=(c-a)/b=-r$. If $T=0$, the two detour
conditions coincide, so density of $\calS$ in the nonempty interval
$(\lambda,\mu)$ gives infinitude.
Otherwise Proposition~\ref{q1:prop:s-minus-one} gives at most one
finite torsion value on each side of $T/2$, with a possible noncentral
pair summing to $T$. Convexity and non-degeneracy give
$\lambda<T/2<\mu$. If both endpoint lifts were torsion, then
$\lambda+\mu=T$. But
$\lambda+\mu-T=(D-a-c)/b\ne0$, a contradiction.
The non-torsion endpoint lift and the branch already verified there give
infinitude by the endpoint-density lemma (Lemma~\ref{q1:lem:density})
and the detour bijection (Lemma~\ref{q1:lem:detour}).
\end{proof}

\subsection{Symmetric parallelograms}
\label{q1:sec:convex-symmetric}

Assume $s=-1$ and $a+c=D$, so
$A=(a,b)$, $B=(D-a,-b)$, with $a\in\Q$, $b,D\in\Q_{>0}$.
No restriction $0<a<D$ is imposed. Put
\[
p=a/b,\qquad q=(D-a)/b,\qquad T=q-p.
\]

\begin{lemma}[Symmetric line reduction]\label{q1:lem:symmetric-line-reduction}
Assume the symmetric convex parallelogram setup. Put
\[
p=\frac ab,\qquad
q=\frac{D-a}{b},\qquad
T=q-p=\frac{D-2a}{b}.
\]
For a rational point \(P=(x,0)\), with \(x\in\Q\), set
\[
\alpha=\frac{x-a}{b}.
\]
Then
\begin{equation*}
0<x<D
\quad\Longleftrightarrow\quad
-p<\alpha<q.
\end{equation*}
Moreover \(PA\) and \(PB\) are rational if and only if $1+\alpha^2, \ 1+(\alpha-T)^2$
are squares in \(\Q\). If the other diagonal \(AB\) has rational length,
then the midpoint of \(UV\) is a required point. If \(AB\notin\Q\), then
\begin{equation*}
T^2+4\notin(\Q^\times)^2.
\end{equation*}
In particular \(T\ne0\).
\end{lemma}

\begin{proof}
The identities $x=a+b\alpha$,
$PA^2=b^2(1+\alpha^2)$, and
$PB^2=b^2(1+(\alpha-T)^2)$ prove the parameter assertions.
The midpoint has distances $D/2,D/2,AB/2,AB/2$.
Finally $AB^2=b^2(T^2+4)$, proving the nonsquare condition and
$T\ne0$ when $AB\notin\Q$.
\end{proof}

\begin{lemma}[The endpoint lift is non-torsion]\label{q1:lem:symmetric-endpoint}
Assume the symmetric convex parallelogram setup, assume the four side lengths are
rational, and assume \(AB\notin\Q\). Let
\[
h_p=\sqrt{1+p^2},
\qquad
h_q=\sqrt{1+q^2},
\]
and put
\[
Q_0=(t_0,u_0):=(h_p-p,\ h_q-q)\in\Gamma_T(\Q).
\]
Let \(R_0=\phi(Q_0)\in E_T(\Q)\), where \(\phi\) is the birational map in
\eqref{q1:eq:symmetric-map}. Then \(R_0\) is non-torsion.
\end{lemma}

\begin{proof}
Rational sides give $h_p=UA/b$, $h_q=UB/b\in\Q_{>0}$.
The parameters $t_0=h_p-p$, $u_0=h_q-q$ are positive and represent
$\alpha=-p$, $\alpha-T=-q$, so $Q_0\in\Gamma_T(\Q)$.
The affine isomorphism in Lemma~\ref{q1:lem:symmetric-detour-curve}
makes $R_0$ finite, with $X(R_0)=t_0u_0>0$. If $AB\notin\Q$,
Lemma~\ref{q1:lem:symmetric-torsion} lists only finite torsion
abscissae $0,-1$; hence $R_0$ is non-torsion.
\end{proof}

\begin{proposition}[Symmetric convex parallelogram case]\label{q1:prop:convex-symmetric}
Assume the symmetric convex parallelogram setup above, and assume that the side
lengths $UA,\ AV,\ VB,\ BU$
are rational. Then the open segment \(UV\) contains a rational point whose
distances from \(U,V,A,B\) are rational. If the other diagonal \(AB\) is not
rational, then there are infinitely many such rational points.
\end{proposition}

\begin{proof}
If $AB\in\Q$, use the midpoint. Otherwise the preceding lemma
gives a non-torsion endpoint lift on the equal-height model.
The positive branch in Proposition~\ref{q1:prop:s-minus-one}, with
interval $(-p,q)$ of length $D/b>0$, satisfies the hypotheses of the
endpoint-density lemma (Lemma~\ref{q1:lem:density}). The resulting infinitely many detour
values give the required points by Lemma~\ref{q1:lem:symmetric-line-reduction}.
\end{proof}

The next criterion retains the symmetric setup, with rational side
lengths imposed only for the converse.

\begin{proposition}[Rank criterion for finiteness]\label{q1:prop:rank-zero}
Assume \(T\ne0\). If \(E_T(\Q)\) has rank \(0\), then
\(\mathcal R(U,V;A,B)\), and hence \(\mathcal R^\circ(U,V;A,B)\), is finite,
without any rational-side assumption.
If, in addition, $UA,AV,VB,BU\in\Q$, then finiteness of either set is
equivalent to rank zero. Under this additional hypothesis, positive rank
gives infinitely many qualifying points on the open diagonal.
\end{proposition}

\begin{proof}
A point of $\mathcal R(U,V;A,B)$ with parameter $\alpha$ has the unique positive
lift $t=\alpha+\sqrt{1+\alpha^2}$,
$u=\alpha-T+\sqrt{1+(\alpha-T)^2}$. The isomorphism in
Lemma~\ref{q1:lem:symmetric-detour-curve} maps these lifts injectively
to $E_T(\Q)$. By the Mordell--Weil theorem, rank zero makes this
group finite.

Now assume the side lengths are rational and the rank is positive.
Put $h_p=\sqrt{1+p^2}=UA/b$ and $h_q=\sqrt{1+q^2}=UB/b$;
both are rational. Let
\(R_U\in E_T(\Q)\) be the lift of
the endpoint \(U\) defined by the positive parameters
\(t=h_p-p\), \(u=h_q-q\), and choose a non-torsion point
\(S\in E_T(\Q)\). Then $2S$ is non-torsion in the real identity
component, so Lemma~\ref{q2:lem:real-density} makes the rational coset
\(R_U+\langle2S\rangle\) dense in the component containing \(R_U\).
The positive branch $t=\vartheta(\alpha)$,
$u=\vartheta(\alpha-T)$ has the regular inverse
$\alpha=(t-t^{-1})/2$. Under the isomorphism of
Lemma~\ref{q1:lem:symmetric-detour-curve}, a short interval of
$\alpha$-values just above $-p$ therefore gives an open arc in that
component, inside the required interval $(-p,q)$.
Its intersection with the coset gives infinitely many distinct points
of $\mathcal R^\circ(U,V;A,B)$, hence also of $\mathcal R(U,V;A,B)$.
Together with the first implication, this proves both equivalences.
\end{proof}

\begin{remark}[Finiteness with a rational second diagonal]
\label{q1:rem:symmetric-rational-ab-finite}
The midpoint construction in Proposition~\ref{q1:prop:convex-symmetric}
gives existence when $AB\in\Q$, but does not force infinitude.
Theorem~\ref{q1:thm:finite-example} has $AB=50$ and the unique qualifying
point $(25,0)$ on the open diagonal; uniqueness follows from the complete
rational-point classification for its rank-zero elliptic curve.

When \(T=0\), the two distance conditions coincide, and the density of
\(\calS\) (Lemma~\ref{q1:lem:S}) gives infinitely many points directly in
the nonempty interval \((-p,q)\); the singular cubic obtained by setting
\(T=0\) is not used as an elliptic curve.
\end{remark}

\subsection{Two finite diagonal sets}

Both examples reduce to the fixed curve
$E:Y^2=X(X+81)(X+256)$, whose sixteen rational points are determined by
the fixed-curve classification (Lemma~\ref{q2:lem:fixed-curve-points}).

\Needspace{14\baselineskip}
\begin{theorem}[A singleton diagonal set in a kite]
\label{q1:thm:finite-aligned}\label{q1:ex:finite-aligned}
Let
\[
\begin{aligned}
U&=(0,0), & A&=(192,144),\\
V&=(384,0), & B&=(192,-256).
\end{aligned}
\]
Then $UAVB$ is a convex kite with side lengths
$UA=AV=240$, $VB=BU=320$, and interior diagonal $UV=384$.
All vertex coordinates and these lengths are integers, and
\[
\mathcal R^\circ(U,V;A,B)=\{(192,0)\}.
\]
\end{theorem}

\begin{proof}
Put $\alpha=(x-192)/144$ and $\rho=9/16$, so the open diagonal
corresponds to $-4/3<\alpha<4/3$. The two off-diagonal distances are
rational exactly when $\alpha,\rho\alpha\in\calS$. Thus every
qualifying point has a rational lift to the aligned model $E_\rho$, which is
identified with $E:Y^2=X(X+81)(X+256)$ by
$(X,Y)=(256X_\rho,4096Y_\rho)$.
The fixed-curve classification (Lemma~\ref{q2:lem:fixed-curve-points})
gives rank zero, so every rational lift is torsion.
The bound \eqref{eq:common-leg-torsion} in the
aligned calculation allows only
$\alpha\in\{0,\pm1/\sqrt\rho\}=\{0,\pm4/3\}$.
Only zero lies in the open interval. It gives $(192,0)$, whose
distances to $U,V,A,B$ are $192,192,144,256$.
\end{proof}

\begin{remark}[Finiteness in the aligned case]\label{q1:rem:convex-r-zero-finite}
Theorem~\ref{q1:thm:finite-aligned} shows that rational side lengths do
not force infinitude in the convex aligned case: its open diagonal has
exactly one qualifying point. Thus the existence construction in
Proposition~\ref{q1:prop:convex-r-zero} need not yield infinitude.
When $\rho=-s=1$, the two square conditions coincide, so density of
$\calS$ gives infinitely many points directly. For $\rho\ne1$, the
concordant-form model is nonsingular. Positive-rank specializations can
give infinitely many points on suitable real intervals; in the rank-zero
example above, the torsion table leaves only the two endpoints and the
single interior point $(192,0)$.
\end{remark}

\Needspace{14\baselineskip}
\begin{theorem}[A singleton diagonal set in a parallelogram]
\label{q1:thm:finite-example}\label{q1:ex:finite-parallelogram}
Let
\[
\begin{aligned}
U&=(0,0), & A&=(18,24),\\
V&=(50,0), & B&=(32,-24).
\end{aligned}
\]
Then $UAVB$ is a symmetric convex parallelogram with side lengths
$UA=VB=30$, $AV=BU=40$, and diagonals $UV=AB=50$.
All vertex coordinates and these lengths are integers, and
$\mathcal R^\circ(U,V;A,B)=\{(25,0)\}$.
Thus its midpoint $P=(25,0)$ is the only rational point on the open
diagonal $UV$ at rational distance from all four vertices.
\end{theorem}

\begin{proof}
Here $T=7/12$, $w=\sqrt{T^2+4}=25/12$, and
$\alpha=(x-18)/24\in(-3/4,4/3)$.
The positive lift in Lemma~\ref{q1:lem:symmetric-detour-curve}
sends every qualifying point to a rational point of $E_T$.
The scaling $(X_1,Y_1)=(144X,1728Y)$ identifies $E_T$ with
$E:Y_1^2=X_1(X_1+81)(X_1+256)$, of rank zero by the fixed-curve
classification (Lemma~\ref{q2:lem:fixed-curve-points}). Thus all rational lifts are torsion.
The table in Proposition~\ref{q1:prop:s-minus-one} allows only
\[
\alpha=\frac{T-w}{2}=-\frac34,\qquad
\frac T2=\frac7{24},\qquad
\frac{T+w}{2}=\frac43.
\]
Only $7/24$ is interior, giving $x=25$; all four distances there are $25$.
\end{proof}

\section{Three-collinear quadrilaterals}\label{q2:sec:introduction}

For distinct collinear $X,Y,Z$ with $Y\in(X,Z)$ and $C$ off their
line, the simple cycle $XYZC$ bounds $\triangle XZC$, with one flat
angle at $Y$. Its only interior diagonal joins $Y$ to $C$; $XZ$ lies
on the boundary. We also use $O=Y$, $A=X$, $B=Z$, so $OACB$ traverses
the same boundary in reverse, up to cyclic permutation.

\begin{theorem}[Three-collinear interior theorem]\label{q2:thm:main-degenerate}
Let \(X,Y,Z,C\in\Z^2\) be four distinct points. Assume that
\(X,Y,Z\) are collinear in that order and that \(C\) does not lie
on their common line. Suppose that
\[
XY,\ YZ,\ ZC,\ CX,\ YC\in\Z_{>0}.
\]
Then there are infinitely many distinct points
\(P\in\Q^2\cap\operatorname{int}(\triangle XZC)\) such that
\[
PX,\ PY,\ PZ,\ PC\in\Q_{>0}.
\]
Thus all these points lie in the strict interior of the three-collinear
degenerate quadrilateral \(XYZC\).
\end{theorem}

Scaling gives the rational-coordinate and rational-length version.
The fifth length $YC$ is the interior diagonal; the boundary length
$XZ=XY+YZ$ is already rational. The diagonal counterexamples in
Section~\ref{q2:sec:diagonal-counterexamples} show why the conclusion
allows the whole interior. We prove infinitude on an interior circular
arc, without asserting density in the triangular region.

\Needspace{18\baselineskip}
\subsection{The normalized theorem}
\label{q2:sec:general-three-collinear}

The following normalized formulation covers every orientation of the
collinear line. Its proof uses the circumcircle construction outlined in
Section~\ref{sec:proof-overview}.

\begin{theorem}[Normalized three-collinear interior theorem]
\label{q2:thm:tilted-interior}
Let
\[
O=(0,0),\qquad A=(a_x,a_y),\qquad C=(c,0),\qquad
B=(b_x,b_y)
\]
be four points. Assume that \(a_x,a_y,b_x,b_y,c\in\Q\), that \(c>0\),
that \(a_y>0>b_y\), that \(A,O,B\) are collinear, and that the four
boundary lengths $OA,\ AC,\ CB,\ BO$ are rational.
Then the polygonal cycle \(O\!-\!A\!-\!C\!-\!B\!-\!O\)
is a three-collinear degenerate quadrilateral, and its strict interior
contains infinitely many distinct points with rational coordinates whose
distances from \(O,A,B,C\) are all positive rational numbers.
\end{theorem}

The signs $a_y>0>b_y$ and collinearity imply $O\in(A,B)$.
The line $AB$ meets the $x$-axis only at $O$, so $C\notin AB$ and
the cycle bounds $\triangle ABC$. The equivalence with the preceding
theorem follows by a rational isometry along $YC$; the reduction is
completed below.

\subsubsection{Normalization along the base}

\begin{proof}[Proof of Theorem~\ref{q2:thm:tilted-interior}]
Because $a_y>0$, $b_y<0$, and $A,O,B$ are collinear, $O$ lies strictly
between $A$ and $B$. Put
\begin{equation*}
p=AO,\qquad q=OB,\qquad L=p+q=AB.                         
\end{equation*}
These are positive rational numbers.

We first pass from the coordinates in the theorem to coordinates along
the collinear base. In the following isometry and its inverse, the vertex
symbols still denote the original points.
Translate $A$ to the origin. The unit vector from $A$ toward $B$
has rational coordinates: it is the rational vector $O-A$ divided
by its rational length $p$. Write this vector as \(e_0\), and put
\(f_0=(-e_{0,2},e_{0,1})\). Then
\[
W\longmapsto\bigl((W-A)\mathbin{\cdot}e_0,\,
(W-A)\mathbin{\cdot}f_0\bigr)
\]
is a Euclidean isometry defined over \(\mathbb Q\), with inverse also
defined over \(\mathbb Q\): the inverse sends \((x,y)\) to
\(A+xe_0+yf_0\). Replacing \(f_0\) by \(-f_0\) if necessary
puts the configuration in the form
\begin{equation}
A=(0,0),\qquad O=(p,0),\qquad B=(L,0),\qquad C=(u,v),        \label{q2:tag:7.9}
\end{equation}
where
\begin{equation*}
0<p<L,\qquad u,v\in\Q,\qquad v>0.                       
\end{equation*}
From \eqref{q2:tag:7.9} onward, the vertex labels denote their images in
these base coordinates. Rationality of points and all Euclidean distances
is preserved, and $c$ still denotes the unchanged length $OC$.

The four side lengths give $AO,AC,CB,BO\in\Q$, while
$OC=c\in\Q$ and $AB=AO+OB$. Thus all six mutual
distances among $A,O,B,C$ are rational. Since $O$ lies on the open
segment $AB$, the strict interior of the quadrilateral $OACB$ is
exactly the strict interior of the triangle $ABC$.

\subsubsection{Choosing a circumcircle with nonzero center height}

Let $\Gamma_1$ be the circle through $A,O,C$, and let $\Gamma_2$ be
the circle through $O,B,C$. Their centers have respective
$y$-coordinates
\begin{equation}
\kappa_1=\frac{u(u-p)+v^2}{2v},
\qquad
\kappa_2=\frac{(u-p)(u-L)+v^2}{2v}.                        \label{q2:tag:7.11}
\end{equation}
Indeed, the perpendicular bisectors of \(AO\) and \(OB\) put the two
center abscissae at \(p/2\) and \((p+L)/2\), respectively; equating the
distance from each center to \(C\) with its distance to a base endpoint
then gives \eqref{q2:tag:7.11}.
They cannot both vanish: $2v(\kappa_2-\kappa_1)=-L(u-p)$,
so simultaneous vanishing would give $u=p$.
The numerator defining $\kappa_1$ would then be $v^2$,
contradicting $v>0$.

If \(\kappa_2=0\), then \(\kappa_1\ne0\). In that case interchange the
labels \(A,B\) and apply the reflection \((x,y)\mapsto(L-x,y)\).
This carries the triangular region isometrically to the relabeled region
and preserves all rational-distance conditions. It sends the old circle
\(AOC\) to the new circle \(OBC\) and replaces \((p,u,v)\) by
\((L-p,L-u,v)\). Substitution in \eqref{q2:tag:7.11} shows that the
new \(\kappa_2\) is the old \(\kappa_1\), so it is nonzero.
After this optional reflection, we retain the labels and
parameters of \eqref{q2:tag:7.9} for the transformed configuration. We may
therefore assume that the circle $\Gamma=\Gamma_2$ through $O,B,C$ has nonzero center height
\begin{equation*}
\kappa:=\kappa_2\ne0.                                     
\end{equation*}
Its equation, center, and radius are
\begin{equation}
(x-p)(x-L)+y^2-2\kappa y=0,                                 \label{q2:tag:7.13}
\end{equation}
\begin{equation}
K=\left(\frac{p+L}{2},\kappa\right),
\qquad
R^2=\frac{(L-p)^2}{4}+\kappa^2.                            \label{q2:tag:7.14}
\end{equation}
The radius is rational. Indeed, triangle $OBC$ has rational sides
and area $(L-p)v/2$, so its circumradius is
\begin{equation*}
R=\frac{OB\cdot OC\cdot BC}{4[OBC]}
=\frac{OC\cdot BC}{2v}\in\Q_{>0}.                             
\end{equation*}

Figure~\ref{fig:interior-circle} shows the geometric part of the construction.
The circle parametrization will make three distances rational, and the
quartic will select the parameters making the fourth rational.

\begin{figure}[htbp]
\centering
\begin{tikzpicture}[x=.2cm,y=.2cm,line cap=round,line join=round,
 every node/.style={font=\small}]
\fill[gray!10] (0,0)--(25,0)--(9,12)--cycle;
\draw (0,0)--(25,0)--(9,12)--cycle;
\draw[gray] (17,6) circle (10);
\draw[dashed] (17,6)--(9,12) node[midway,above] {$R$};
\draw[very thick] (17,6) ++(200:10) arc (200:216.86989765:10);
\fill (0,0) circle (1.4pt) node[below] {$A$};
\fill (9,0) circle (1.4pt) node[below] {$O$};
\fill (25,0) circle (1.4pt) node[below] {$B$};
\fill (9,12) circle (1.4pt) node[above left] {$C$};
\fill (17,6) circle (1.2pt) node[right] {$K$};
\node at (26,12) {$\Gamma$};
\node[align=center] at (1,6) {interior arc\\$t<0$};
\draw[->] (3,5)--(7.7,2.3);
\end{tikzpicture}
\caption{The circumcircle $\Gamma$ through $O,B,C$ after normalization along
$AB$. A sufficiently short one-sided arc from $O$ lies inside $\triangle ABC$.
The bold arc illustrates its location; qualifying rational points on it are
selected by the quartic condition for the distance to $A$.}
\label{fig:interior-circle}
\end{figure}

\subsubsection{A rational-distance parametrization of the circle}

Define the rational unit vectors
\begin{equation}
e=\frac{O-K}{R},\qquad f=(-e_2,e_1).                       \label{q2:tag:7.16}
\end{equation}
They are rational because \(O,K\in\mathbb Q^2\) and
\(R\in\mathbb Q_{>0}\).
For $t\in\Q$ put
\begin{equation}
Z(t)=K+\frac{R}{(1+t^2)^2}
\left((t^4-6t^2+1)e+4t(1-t^2)f\right).                   \label{q2:tag:7.17}
\end{equation}
The same formula defines a real analytic map for all $t\in\R$, since
$1+t^2>0$. This is the quarter-angle parametrization of $\Gamma$: if
$t=\tan(\theta/4)$, then the central angle from $O$ to $Z(t)$ is
\(\theta\). The formula extends continuously to
\(\mathbb P^1(\mathbb R)\), with \(Z(0)=Z(\infty)=O\).
It covers the circle twice: for finite nonzero \(t\),
\(Z(t)=Z(-1/t)\). We will restrict it to an injective interval when
proving distinctness of the interior points. For two finite
parameters \(t_1,t_2\), the chord formula is
\begin{equation}
|Z(t_1)-Z(t_2)|
=\frac{4R\,|(t_1-t_2)(1+t_1t_2)|}
{(1+t_1^2)(1+t_2^2)}.                             \label{q2:tag:7.18}
\end{equation}
This follows directly from
\[
|Z(t_1)-Z(t_2)|=2R\left|\sin\left(
2(\arctan t_1-\arctan t_2)\right)\right|.
\]

We give explicit rational parameters for \(B\) and \(C\).
For \(J\in\{B,C\}\), put
\[
a_J=\frac{(J-K)\cdot e}{R},\qquad
b_J=\frac{(J-K)\cdot f}{R},\qquad
h_J=\frac{OJ}{2R}>0.
\]
These quantities are rational, and the circle and chord identities give
\(a_J^2+b_J^2=1\) and \(h_J^2=(1-a_J)/2\).
Set \(g_J=b_J/(2h_J)\). Then \(g_J\in\Q\) and
\(g_J^2+h_J^2=1\). Since \(h_J>0\), we have \(|g_J|<1\),
so the denominator \(1+g_J\) is strictly positive.
Define
\begin{equation*}
\tau_J=\frac{h_J}{1+g_J}\in\Q_{>0},\qquad J\in\{B,C\}.
\end{equation*}
Indeed,
\[
\frac{1-\tau_J^2}{1+\tau_J^2}=g_J,\qquad
\frac{2\tau_J}{1+\tau_J^2}=h_J.
\]
The two coefficients of \(e,f\) in the unit vector
\((Z(\tau_J)-K)/R\) are therefore
\(g_J^2-h_J^2=a_J\) and \(2g_Jh_J=b_J\).
Thus \(Z(\tau_J)=J\), including when \(J\) is antipodal to \(O\).

It follows from \eqref{q2:tag:7.18}, using the parameters
\(0,\tau_B,\tau_C\), that for every \(t\in\Q\) the three
distances
\begin{equation*}
Z(t)O,\qquad Z(t)B,\qquad Z(t)C                             
\end{equation*}
are rational. We must impose only the remaining condition that
$Z(t)A$ be rational.

Direct substitution of \eqref{q2:tag:7.17}, using \eqref{q2:tag:7.13}--\eqref{q2:tag:7.16}, gives
\begin{equation}
AZ(t)^2=\frac{N(t)}{(1+t^2)^2},                              \label{q2:tag:7.21}
\end{equation}
where
\begin{equation}
\begin{aligned}
N(t)={}&p^2t^4-8\kappa pt^3
+(4L^2+16\kappa^2-2p^2)t^2\\
&\quad+8\kappa pt+p^2.
\end{aligned}                                                \label{q2:tag:7.22}
\end{equation}
For a transparent verification, put
\[
c_t=\frac{t^4-6t^2+1}{(1+t^2)^2},\qquad
s_t=\frac{4t(1-t^2)}{(1+t^2)^2},\qquad d=L-p.
\]
Equations \eqref{q2:tag:7.14}, \eqref{q2:tag:7.16}, and \eqref{q2:tag:7.17} give
\[
Z(t)=\left(
\frac{p+L}{2}-\frac d2c_t+\kappa s_t,
\ \kappa(1-c_t)-\frac d2s_t
\right).
\]
Squaring these two coordinates, using $c_t^2+s_t^2=1$, and collecting
powers of $t$ gives exactly \eqref{q2:tag:7.21}--\eqref{q2:tag:7.22}.

The remaining distance condition is therefore equivalent to lifting
\(t\in\Q\) to a rational point on the quartic
\begin{equation}
\mathcal H:\quad w^2=N(t).                                 \label{q2:tag:7.23}
\end{equation}
The point \((t,w)=(0,p)\) is already rational, but it gives the boundary
point \(Z(0)=O\). We will find infinitely many rational lifts whose
parameters map into the strict interior.
The polynomial has discriminant
\[
\operatorname{disc}_t N
=4096L^4p^4\bigl((L-p)^2+4\kappa^2\bigr)
             \bigl((L+p)^2+4\kappa^2\bigr)>0.
\]
Its four roots are therefore distinct. Since its leading coefficient
is $p^2\ne0$, the smooth projective model of \(w^2=N(t)\) has genus one. We write
\(\overline{\mathcal H}\) for the normalization of the projective closure
of \eqref{q2:tag:7.23}, that is, for its smooth projective model.

\subsubsection{A non-torsion point on the elliptic model}

Apply the common transformation of Section~\ref{sec:quartic-model}
to \eqref{q2:tag:7.23} with
\[
A=p^2,\quad B=-8\kappa p,\quad
C=4L^2+16\kappa^2-2p^2,\quad D=8\kappa p,\quad \kappa_0=p.
\]
Here the coefficient letters in the common transformation are independent
of the geometric vertex labels. Using $(\xi,\eta)$ for its cubic
coordinates, the specialization is
\begin{equation}\label{q2:tag:7.26}
\eta^2=(\xi-2p^2)
\left(\xi^2+(4L^2+16\kappa^2)\xi
+8L^2p^2-32\kappa^2p^2-4p^4\right).
\end{equation}
The formulas in Section~\ref{sec:quartic-model} give inverse
isomorphisms between the affine open subsets where
$t\ne0$, $\xi^2\ne4p^4$ on the quartic and where
$\xi^2\ne4p^4$, $\eta+4\kappa(\xi-2p^2)\ne0$ on the cubic.
The last condition ensures that the recovered $t$ is nonzero.
Only finitely many points are omitted from these affine charts.
Rescale only the elliptic-model coordinates by
\begin{equation*}
\xi=2p^2+4L^2X,\qquad \eta=8L^3Y.
\end{equation*}
The geometric coordinates and circle parameter $t$ remain unchanged.
Equation \eqref{q2:tag:7.26} becomes
\begin{equation}
Y^2=X^3+(1+r^2+s^2)X^2+s^2X,                               \label{q2:tag:7.29}
\end{equation}
where we choose
\begin{equation}
r=-\frac{2\kappa}{L},
\qquad
s=-\frac pL.                                               \label{q2:tag:7.30}
\end{equation}
We have
\begin{equation}
r\ne0,\qquad -1<s<0.                                    \label{q2:tag:7.31}
\end{equation}
The parameters satisfy the nonsingularity hypotheses of
Theorem~\ref{q1:thm:love}. The displayed maps extend to an isomorphism
$\jmath:\overline{\mathcal H}\to E$ of smooth projective curves over $\Q$.
The model has origin $H_0=(0,p)$ corresponding to $\calO_E$ by
Section~\ref{sec:quartic-model}. The recovered coordinate is the circle
parameter $t$, and the point tested for non-torsion is the rational point
$T=(-1,r)\in E(\Q)$.

By Theorem~\ref{q1:thm:love}(\ref{q1:part:love-torsion}), the only possible obstruction to
the point $(-1,r)$ having infinite order is
\begin{equation}
4r^2s=-(1-s^2)^2.                                          \label{q2:tag:7.32}
\end{equation}
Indeed, \(4r^2s<0\) by \eqref{q2:tag:7.31}, whereas
\((1-s^2)^2>0\), so the positive sign in the exceptional equality of
Theorem~\ref{q1:thm:love}(\ref{q1:part:love-torsion}) cannot occur.
Substitution of \eqref{q2:tag:7.30} turns \eqref{q2:tag:7.32} into
\begin{equation}
16pL\kappa^2=(L^2-p^2)^2.                                 \label{q2:tag:7.33}
\end{equation}

To exclude \eqref{q2:tag:7.33}, put $d=L-p>0$. It would imply
\[
\sqrt{pL}=\frac{d(L+p)}{4|\kappa|}\in\Q,\qquad
z=\frac{\sqrt{pL}}p=\sqrt{L/p}\in\Q,\quad z>1.
\]
The rational radius gives $W=4R/d\in\Q$, and
\[
W^2=4+\frac{16\kappa^2}{d^2}
=4+\frac{(L+p)^2}{pL}=6+z^2+z^{-2}.
\]
Thus $(zW)^2=z^4+6z^2+1$, contradicting
the Fermat quartic obstruction (Lemma~\ref{q2:lem:fermat-quartic}).
Therefore $T$ is non-torsion.
In the perpendicular notation of Section~\ref{q2:sec:perpendicular-special},
this holds also when $c^2=ab$:
the circle curve here differs from the diagonal curve
\eqref{q2:eq:perp-curve}, where a distinguished point has order $4$.

\subsubsection{Density and the interior arc}

Since $N(0)=p^2>0$, the positive branch
$H(t)=(t,\sqrt{N(t)})$ is real analytic for $t$ near zero and passes
through $H_0$. The coordinate $t$ is a local parameter there because
the derivative of $w^2-N(t)$ with respect to $w$ is $2p\ne0$.
The smooth projective isomorphism $\jmath$ is regular at $H_0$, even
though its affine formulas have poles there. Thus $\jmath(H(t))$
parametrizes an open arc of $E(\R)^0$ near $\calO_E$.

The component group of a real elliptic curve has order at most two,
so $2T\in E(\R)^0$ and remains non-torsion. By the real-component
density lemma (Lemma~\ref{q2:lem:real-density}), its rational multiples are dense in
that component. Every sufficiently small one-sided branch arc
$H(( -\varepsilon,0))$ therefore contains infinitely many rational
points: apply $\jmath^{-1}$ to the multiples lying on its image.
Their affine coordinates $t,w$ are rational, since $\jmath^{-1}$ is
defined over $\Q$. On the positive branch distinct points have distinct $t$.

The two non-base side inequalities for $\triangle ABC$ are
$vx-uy>0$ and $v(L-x)+(u-L)y>0$. At $O=(p,0)$ their left sides
are $vp>0$ and $v(L-p)>0$. Thus a small disk about $O$ satisfies
both inequalities, and its upper half lies inside $\triangle ABC$.
Differentiating the circle parametrization gives
\[
Z'(0)=(4\kappa,-2(L-p)).
\]
Its second coordinate is negative. After decreasing $\varepsilon$,
$Z((-\varepsilon,0))$ therefore lies in that upper half-disk and
the ordinate of $Z$ is strictly decreasing there. Thus $Z$ is
injective on the interval. The density argument applies to this smaller
open branch arc as well. For each rational branch point take $Q=Z(t)$.
These are infinitely many distinct strict-interior rational points.
All four vertices lie on the boundary, so the distances are nonzero;
the chord and quartic identities give
\begin{equation*}
QO,QB,QC\in\Q_{>0},\qquad QA=\frac{|w|}{1+t^2}\in\Q_{>0}.
\end{equation*}
Undo first the optional base reflection and then the rational isometry,
returning to the original coordinates of the theorem. These maps preserve
the points' distinctness, distances, and strict interiority. This proves
Theorem~\ref{q2:thm:tilted-interior}.
\end{proof}

\begin{proof}[Proof of Theorem~\ref{q2:thm:main-degenerate}]
Let the degenerate quadrilateral $XYZC$ satisfy the hypotheses of the
main theorem, and relabel $O=Y,\ A=X,\ B=Z$.
Because $Y$ lies strictly between $X$ and $Z$, the points $A,O,B$ are
collinear and $O$ lies strictly between $A$ and $B$. Moreover, the boundary
order $O,A,C,B$ is the reverse, up to cyclic permutation, of the boundary
order $X,Y,Z,C$, so both descriptions determine the same triangular region.

Put $d=YC\in\Z_{>0}$. Lemma~\ref{q1:lem:normalization} gives a
rational isometry sending $O,C$ to $(0,0),(d,0)$. The images of
$A,B$ have nonzero second coordinates because their line through $O$
does not contain $C$; these coordinates have opposite signs because
$O\in(A,B)$. Reflecting in the $x$-axis if necessary makes the
second coordinate of $A$ positive. All coordinates remain rational
and all side lengths remain integers, so every hypothesis of
Theorem~\ref{q2:thm:tilted-interior} holds.

That theorem supplies infinitely many rational points in the strict interior
at rational distance from all four transformed vertices. The inverse
rational isometry preserves rational coordinates, strict interior, Euclidean
distances, and distinctness, and therefore gives infinitely many required
points for $X,Y,Z,C$.
\end{proof}

\section{Perpendicular diagonal refinements and obstructions}
\label{q2:sec:perpendicular-special}

Let $O=(0,0)$, $A=(0,a)$, $B=(0,-b)$, $C=(c,0)$, where
$a,b,c\in\Z_{>0}$, and assume
\begin{equation*}
m=\sqrt{a^2+c^2},\qquad n=\sqrt{b^2+c^2}\quad\text{belong to }\Z_{>0}.
\end{equation*}
The cycle $OACB$ is $\triangle ABC$ with $AB$ subdivided at $O$;
its boundary lengths are $a,m,n,b$ and its interior diagonal has
length $c$. The boundary lines
\begin{equation*}
y=a(1-x/c),\qquad y=-b(1-x/c)
\end{equation*}
give the exact interior
\begin{equation}\label{q2:eq:perpendicular-interior}
0<x<c,\qquad -b(1-x/c)<y<a(1-x/c).
\end{equation}

\begin{theorem}[Perpendicular opposite-side theorem]
\label{q2:thm:perpendicular}
Let $a,b,c\in\mathbb Z_{>0}$ and suppose that
$m=\sqrt{a^2+c^2},\ n=\sqrt{b^2+c^2}$ belong to $\mathbb Z_{>0}$. For
$O=(0,0),\ A=(0,a),\ C=(c,0),\ B=(0,-b)$,
the strict interior of the degenerate quadrilateral $OACB$, equivalently
the strict interior of $\triangle ABC$, contains infinitely many distinct
points $Q\in\Q^2$, each satisfying
$QO,\ QA,\ QC,\ QB\in\Q_{>0}$.
If $c^2\ne ab$, infinitely many such points $Q$ can be chosen on the open
diagonal $OC$.
\end{theorem}

\subsection{Proof of the diagonal refinement}
\label{q2:sec:perpendicular-proof}

The interior-infinitude conclusion follows from
Theorem~\ref{q2:thm:main-degenerate}. We prove the additional diagonal
assertion, including rational versions by scaling.

\begin{proof}
If $a=b$, take $Q=(x,0)$ with
\begin{equation*}
x=\frac{2aN}{N^2-1},\qquad N\in\Z,\quad N>1.
\end{equation*}
For all sufficiently large $N$, $0<x<c$, and
\begin{align*}
QA=QB&=\frac{a(N^2+1)}{N^2-1},\\
QO&=x,\qquad QC=c-x.
\end{align*}
These give infinitely many distinct points since $x$ strictly decreases
to zero. Assume now $a\ne b$ and $c^2\ne ab$.
Use the common-leg curve \eqref{eq:common-leg-quadrics} with $h=a$, $k=b$,
whose elliptic model is
\begin{equation}\label{q2:eq:perp-curve}
E:\quad Y^2=X(X+a^2)(X+b^2).
\end{equation}
The roots $0,-a^2,-b^2$ are distinct because $a,b>0$ and $a\ne b$,
so $E$ is nonsingular. This is the common-leg model with $(h,k)=(a,b)$: the origin
$(1:0:1:1)$ maps to $\calO_E$, the coordinate $z=X_0/X_1$ recovers
$Q=(z,0)$, and the endpoint lift at $z=c$ is tested for non-torsion.
The branch
\[
Q_z=(z:1:\sqrt{z^2+a^2}:\sqrt{z^2+b^2}),\qquad 0\le z\le c,
\]
is real analytic on a neighborhood of this interval. Both square-root
coordinates are positive, so the Jacobian minor in these coordinates is
$4\sqrt{z^2+a^2}\sqrt{z^2+b^2}>0$; $z=X_0/X_1$ is a regular inverse
coordinate. The endpoints are $Q_0=(0:1:a:b)$ and $Q_c=(c:1:m:n)$,
both rational. The torsion bound \eqref{eq:common-leg-torsion} shows that
$Q_c$ is non-torsion, since $c>0$ and $c^2\ne ab$.
The endpoint-density lemma (Lemma~\ref{q1:lem:density}) now gives infinitely
many distinct rational
$z\in(0,c)$ for which both square roots are rational. The points
$Q=(z,0)$ have distances
\begin{equation*}
QO=z,\quad QC=c-z,\quad QA=\sqrt{z^2+a^2},\quad QB=\sqrt{z^2+b^2},
\end{equation*}
all positive and rational, and \eqref{q2:eq:perpendicular-interior}
proves interiority.

This also recovers the distinguished point and its halving:
\begin{equation}\label{q2:eq:perp-point}
P=(c^2,cmn)=\pm2S_0\in E(\Q),
\end{equation}
where $S_0$ is the image of $Q_c$ under the common-leg isomorphism.
Since $S_0$ is non-torsion, it is affine, say $S_0=(u,v)$, with
$v\ne0$. By \eqref{eq:common-leg-coordinate} and duplication,
\[
c=\frac{a^2b^2-u^2}{2v},\qquad
X(2S_0)=\frac{(u^2-a^2b^2)^2}{4v^2}=c^2.
\]
The curve equation then gives $Y(2S_0)=\pm cmn$. Since $S_0$ is
non-torsion, so is $P$.
\end{proof}

If
\begin{equation*}
c^2=ab,
\end{equation*}
then $a\ne b$: otherwise $c=a$ and $m=a\sqrt2$ would be irrational.
Moreover, $mn=c(a+b)$, so the same distinguished point is
$P=(ab,ab(a+b))$. The duplication formula (Lemma~\ref{lem:common-duplication}) gives
$X(2P)=0$, and the curve equation then forces $2P=(0,0)$.
This is a nonidentity point of order $2$, so $P$ has exact order $4$.
This alone does not classify diagonal points, but the examples below
show that the open diagonal can contain no point at rational distance
from all four vertices. Interior infinitude still follows from
Theorem~\ref{q2:thm:main-degenerate}.

\subsection{Diagonal obstructions and counterexamples}
\label{q2:sec:diagonal-counterexamples}

The two examples establish failure on a prescribed diagonal under
all arithmetic hypotheses of the interior theorem.

\subsubsection{The basic perpendicular counterexample}

We begin by testing the stronger requirement that the point lie on the
open segment \(OC\).

Take
\begin{equation*}
O=(0,0),\qquad A=(0,9),\qquad C=(12,0),\qquad B=(0,-16).      
\end{equation*}
All four side lengths and the diagonal \(OC\) are integral:
\begin{equation}
\begin{aligned}
OA&=9,& BO&=16,& OC&=12,\\
AC&=\sqrt{12^2+9^2}=15,&
CB&=\sqrt{12^2+16^2}=20.
\end{aligned}                                                 \label{q2:tag:5.2}
\end{equation}
Also $12^{2}=9\cdot16$, so this is the exceptional case
$c^{2}=ab$.

\begin{proposition}[Perpendicular diagonal counterexample]
\label{q2:prop:counterexample}
There is no point in the open segment $OC$ whose distances
from $O,A,B,C$ are all rational.
\end{proposition}

\begin{proof}
\medskip
\noindent\textit{Reduction to a specific elliptic curve.}

Suppose that a point $Q=(q,0)$, where $0<q<12$, had all four
distances rational. Since $QO=q$, we necessarily have $q\in\Q$;
then $QC=12-q$ is rational as well.
The other two would give $r,s\in\Q_{>0}$ satisfying
\begin{equation}
r^{2}=q^{2}+81,\qquad s^{2}=q^{2}+256.                        \label{q2:tag:5.3}
\end{equation}
It follows that
\begin{equation}
(X,Y)=(q^{2},qrs)                                             \label{q2:tag:5.4}
\end{equation}
is a rational point on
\begin{equation*}
E:\quad Y^{2}=X(X+81)(X+256)
=X^{3}+337X^{2}+20736X.                                     
\end{equation*}
The fixed-curve classification (Lemma~\ref{q2:lem:fixed-curve-points})
determines \(E(\Q)\) completely.
By \eqref{q2:tag:5.4}, a positive solution of \eqref{q2:tag:5.3}
requires \(X=q^2>0\). The lemma shows that the positive \(X\)-coordinates
on \(E(\Q)\) are \(24,144,864\). Of these, only \(144\) is a square in
\(\Q\): an integer that is a rational square is an integer square, while
\(24=2^3\cdot3\) and \(864=2^5\cdot3^3\). Hence \(q^2=144\), and
positivity gives \(q=12\). This is the endpoint \(C\), not a point in the
open segment \(OC\). This proves Proposition~\ref{q2:prop:counterexample}.
\end{proof}

The same example also disproves a prescribed-diagonal strengthening
of the general three-collinear theorem. Indeed, $A,O,B$ are collinear with $O$ between the other
two, and the five required lengths in \eqref{q2:tag:5.2} are integral.
Thus Proposition~\ref{q2:prop:counterexample} already satisfies all
hypotheses of the general three-collinear setting.

\subsubsection{The tilted counterexample}

The preceding counterexample has \(AB\perp OC\). One might therefore
hope that requiring the two lines not to be perpendicular would
guarantee the desired point on the open diagonal. The following
example shows that it does not. In the normalized three-collinear setup
of Theorem~\ref{q2:thm:tilted-interior},
\(AB\not\perp OC\) is equivalent to \(a_xb_x\ne0\): the line
\(OC\) is horizontal, and the nonzero vectors \(A\) and \(B\)
are scalar multiples with opposite signs.

\begin{proposition}[Tilted collinear counterexample]
\label{q2:prop:tilted-collinear-counterexample}
Let
\begin{equation*}
O=(0,0),\qquad A=(7,24),\qquad C=(25,0),\qquad B=(-7,-24).
\end{equation*}
This quadrilateral satisfies all the hypotheses of the normalized
three-collinear theorem (Theorem~\ref{q2:thm:tilted-interior}) and also has
\[
a_xb_x\ne0\qquad\text{and}\qquad AB\not\perp OC,
\]
but its open diagonal \(OC\) contains no point whose
distances from all four vertices are rational.
\end{proposition}

\begin{proof}
All four vertices have integral coordinates. Since \(B=-A\), the
points \(A,O,B\) are collinear and \(O\) lies strictly between \(A\)
and \(B\). Also, \(A\) is above the \(x\)-axis and \(B\) is below
it. The four side lengths, as well as the diagonal \(OC\), are
integers:
\[
OA=\sqrt{7^2+24^2}=25,\qquad
AC=\sqrt{18^2+24^2}=30,
\]
\begin{equation*}
CB=\sqrt{32^2+24^2}=40,\qquad
BO=25,\qquad OC=25.                                          
\end{equation*}
Furthermore, \(a_x=7\) and \(b_x=-7\) are nonzero. The line \(AB\)
is spanned by \((7,24)\), whereas the line \(OC\) is spanned by
\((1,0)\). Their scalar product is \(7\ne0\), so
\(AB\not\perp OC\).

Suppose, for a contradiction, that a point in the open
diagonal has rational distances from all four vertices. Write it as
\[
Q=(x,0),\qquad x\in\mathbb Q,\qquad 0<x<25.
\]
Here \(x\in\Q\) follows from \(QO=x\); the other endpoint distance
\(QC=25-x\) is then rational as well.
Put
\begin{equation}
t=\frac{x}{25},\qquad
q=12\frac{1-t}{1+t}.                                        \label{q2:tag:5.34}
\end{equation}
Then \(t\in\mathbb Q\), \(0<t<1\), and hence $q\in\mathbb Q,\ 0<q<12$.
In fact, \eqref{q2:tag:5.34} is a rational bijection between these two open
intervals, with inverse
\[
t=\frac{12-q}{12+q}.
\]

Direct substitution in the two remaining distances gives
\begin{equation*}
QA^2=\frac{25}{9}(1+t)^2(q^2+81),
\end{equation*}
\begin{equation*}
QB^2=\frac{25}{16}(1+t)^2(q^2+256).
\end{equation*}
Since $1+t>0$, equivalently
\begin{equation*}
QA=\frac53(1+t)\sqrt{q^2+81},\qquad
QB=\frac54(1+t)\sqrt{q^2+256}.
\end{equation*}
Thus rational $QA,QB$ would give a point $(q,0)$ on the preceding
counterexample's open diagonal with all four distances rational,
contrary to Proposition~\ref{q2:prop:counterexample}.
\end{proof}

Figure~\ref{fig:finite-empty} contrasts the two limitations on a prescribed
diagonal: a single qualifying point can be the full set in a convex
configuration, and the set can be empty in a three-collinear configuration.

\begin{figure}[htbp]
\centering
\begin{minipage}[c]{.44\linewidth}\centering
\begin{tikzpicture}[x=.85cm,y=.85cm,line cap=round,line join=round,
 every node/.style={font=\small}]
\fill[gray!10] (0,0)--(1.5,1.125)--(3,0)--(1.5,-2)--cycle;
\draw (0,0)--(1.5,1.125)--(3,0)--(1.5,-2)--cycle;
\draw[dashed] (0,0)--(3,0);
\fill (0,0) circle (1.4pt) node[left] {$U$};
\fill (3,0) circle (1.4pt) node[right] {$V$};
\fill (1.5,1.125) circle (1.4pt) node[above] {$A$};
\fill (1.5,-2) circle (1.4pt) node[below] {$B$};
\fill (1.5,0) circle (2pt) node[above] {$P$};
\end{tikzpicture}

\small Exactly one point on $(U,V)$.
\end{minipage}\hfill
\begin{minipage}[c]{.44\linewidth}\centering
\begin{tikzpicture}[x=.55cm,y=.55cm,line cap=round,line join=round,
 every node/.style={font=\small}]
\fill[gray!10] (0,2.25)--(3,0)--(0,-4)--cycle;
\draw (0,2.25)--(3,0)--(0,-4)--cycle;
\draw[dashed] (0,0)--(3,0);
\fill (0,0) circle (1.4pt) node[left] {$O$};
\fill (0,2.25) circle (1.4pt) node[above] {$A$};
\fill (0,-4) circle (1.4pt) node[below] {$B$};
\fill (3,0) circle (1.4pt) node[right] {$C$};
\end{tikzpicture}

\small No point on $(O,C)$.
\end{minipage}
\caption{Qualifying points on prescribed diagonals. Left: the kite of
Theorem~\ref{q1:thm:finite-aligned}, whose only point is $P=(192,0)$.
Right: the configuration $(a,b,c)=(9,16,12)$ in
Proposition~\ref{q2:prop:counterexample}; its open diagonal contains none,
although its strict triangular interior contains infinitely many.
Both drawings are uniform rescalings of their respective examples.}
\label{fig:finite-empty}
\end{figure}

\section{The same-side diagonal theorem}
\label{q2:sec:same-side-special}

The vertices $A$ and $B$ lie on the same side of the line $OC$, and
$A$ is the middle collinear vertex. After clearing denominators and
relabeling, the configuration falls under
Theorem~\ref{q2:thm:main-degenerate};
the separate argument gives the stronger conclusion of infinitely
many points on the prescribed open diagonal $AC$.

Let \(a,b,c\in\mathbb Q_{>0}\) with \(0<a<b\), and consider the
configuration
\[
O=(0,0),\qquad A=(0,a),\qquad B=(0,b),\qquad C=(c,0).
\]
The vertices \(O,A,B\) are collinear and occur in that order. The boundary
cycle is $O$--$A$--$B$--$C$--$O$: its enclosed region is the triangle
$OBC$, with the side $OB$ subdivided at $A$. The additional arithmetic
hypotheses are that
\begin{equation}
m=\sqrt{a^{2}+c^{2}}\in\mathbb Q,\qquad
n=\sqrt{b^{2}+c^{2}}\in\mathbb Q.                           \label{q2:tag:6.1}
\end{equation}
The corresponding integer-coordinate problem is the special case
in which \(a,b,c,m,n\) are positive integers.

The conditions in \eqref{q2:tag:6.1} make both $OAC$ and $OBC$
rational-sided triangles. A rational affine parameter on $AC$ makes the
distances to $A$ and $C$ rational; the remaining task is to make the
distances to $O$ and $B$ rational simultaneously.

\begin{theorem}[Same-side diagonal theorem]
\label{q2:thm:same-side}
Let \(a,b,c\in\mathbb Q_{>0}\) satisfy \(a<b\), and suppose that the two
numbers \(m,n\) defined in \eqref{q2:tag:6.1} are rational. Then there are infinitely
many distinct rational points in the open segment \(AC\) whose distances
from \(O,A,B,C\) are all positive rational numbers.
\end{theorem}

The proof converts the two remaining square-distance conditions into a
genus-one quartic. Its elliptic model contains two useful rational points:
the image of the endpoint $C$ and a point on the other real component. An
exact torsion argument shows that at least one of these points has infinite
order. Density on the identity component then supplies infinitely many
rational points with parameter in the open interval $0<t<1$.
We first construct the model and its distinguished points, then prove two
auxiliary lemmas, and finally complete the density argument.

\subsection{The equations along the diagonal}

\begin{proof}[Proof of Theorem~\ref{q2:thm:same-side}: reduction to an elliptic model]
\renewcommand{\qedsymbol}{}
Put $d=b-a>0$.
A point in the open segment \(AC\) has the form
\begin{equation}
Q(t)=A+t(C-A)=(ct,a(1-t)),\qquad 0<t<1.                      \label{q2:tag:6.2}
\end{equation}
Every such point lies strictly inside the triangular region $OBC$.
Indeed,
\[
Q(t)=(1-t)\left(1-\frac ab\right)O
     +(1-t)\frac ab B+tC,
\]
and these three barycentric coefficients are positive and sum to $1$.
Because \(AC=m\), its two distances along the diagonal are
\begin{equation}
Q(t)A=tm,\qquad Q(t)C=(1-t)m.                               \label{q2:tag:6.3}
\end{equation}
The other two squared distances are
\begin{equation*}
\begin{aligned}
Q(t)O^{2}
&=c^{2}t^{2}+a^{2}(1-t)^{2}
 =m^{2}t^{2}-2a^{2}t+a^{2},\\
Q(t)B^{2}
&=c^{2}t^{2}+\bigl(b-a(1-t)\bigr)^{2}
 =m^{2}t^{2}+2ad\,t+d^{2}.
\end{aligned}
\end{equation*}
Consequently it is enough to find infinitely many distinct rational
parameters $t\in(0,1)$ admitting $\rho,\sigma\in\mathbb Q$ such that
\begin{equation}
\rho^{2}=m^{2}t^{2}-2a^{2}t+a^{2},\qquad
\sigma^{2}=m^{2}t^{2}+2ad\,t+d^{2}.                             \label{q2:tag:6.6}
\end{equation}

\subsection{The genus-one curve}

Let \(\mathcal C\) be the normalization of the projective closure of the
affine curve \eqref{q2:tag:6.6}; thus \(\mathcal C\) denotes its smooth projective model.
It has the rational points
\begin{equation}
(t,\rho,\sigma)=(0,\pm a,\pm d),\qquad
(t,\rho,\sigma)=(1,\pm c,\pm n).                                    \label{q2:tag:6.7}
\end{equation}
At each of these eight affine points, $\rho$ and $\sigma$ are nonzero. The
Jacobian minor of \eqref{q2:tag:6.6} with respect to $(\rho,\sigma)$ is $4\rho\sigma\ne0$,
so the points are nonsingular and lift uniquely to $\mathcal C$.
We use \((0,a,d)\) as the origin.

Parametrize the first conic in \eqref{q2:tag:6.6} by drawing the line
$\rho=a+ut$ through \((t,\rho)=(0,a)\). Substitution gives
$a^{2}+2aut+u^{2}t^{2}=m^{2}t^{2}-2a^{2}t+a^{2}$.
Apart from the base point \(t=0\), this yields
\begin{equation}
t=\frac{2a(u+a)}{m^{2}-u^{2}}.                              \label{q2:tag:6.8}
\end{equation}
Set
\begin{equation}
z=u+a,\qquad
w=\sigma\bigl(m^{2}-(z-a)^{2}\bigr).                             \label{q2:tag:6.9}
\end{equation}
Using \(m^{2}=a^{2}+c^{2}\), substitution of \eqref{q2:tag:6.8} into the second
equation of \eqref{q2:tag:6.6} gives the quartic
\begin{equation}
\begin{aligned}
w^{2}={}&d^{2}z^{4}-4abd\,z^{3}\\
&+2\bigl(2a^{2}b^{2}+a^{2}c^{2}
+2abc^{2}-b^{2}c^{2}\bigr)z^{2}\\
&+4abc^{2}d\,z+c^{4}d^{2}.
\end{aligned}                                                \label{q2:tag:6.10}
\end{equation}
The reverse substitution is explicit. Writing
$D_z=m^{2}-(z-a)^{2}=c^{2}+2az-z^{2}$, we recover
\[
t=\frac{2az}{D_z},\qquad \rho=a+(z-a)t,\qquad \sigma=\frac{w}{D_z}
\quad (D_z\ne0).
\]
Conversely, for $t\ne0$ we have $z=(\rho-a)/t+a$ and $w=\sigma D_z$.
These maps are mutually inverse between the open subsets $t\ne0$
and $zD_z\ne0$. Indeed, the conic equation gives $tD_z=2az$, and
$z=0$ would force $c^2t^2=0$. If $F(z)$ is
the quartic on the right side of \eqref{q2:tag:6.10}, then
\[
\operatorname{disc}_z F
=4096a^{4}b^{2}c^{8}d^{4}(a^{2}+c^{2})
  \bigl(a^{2}b^{2}+c^{2}(2a-b)^{2}\bigr)>0.
\]
Thus $F$ is squarefree of degree $4$, so $w^2=F(z)$ is geometrically
irreducible and its smooth projective model has genus one. The fiber
$t=0$ consists of four nonsingular points, so removing it loses no
component of the distance curve. The inverse maps above therefore show
that $\mathcal C$ is also geometrically irreducible and has genus one.
At \(z=0\), the right side is \((c^{2}d)^{2}\), so \eqref{q2:tag:6.10} has the
rational point
\begin{equation}
(z,w)=(0,c^{2}d).                                           \label{q2:tag:6.11}
\end{equation}
This is the point corresponding to the chosen origin \((0,a,d)\) on
\(\mathcal C\). Indeed, on the branch \(\rho>0\) through that point,
\(\rho=a-at+O(t^{2})\); hence \(u=(\rho-a)/t\longrightarrow-a\), so
\(z\longrightarrow0\) and \(w\longrightarrow c^{2}d\).

Applying Section~\ref{sec:quartic-model} to \eqref{q2:tag:6.10}
with $\kappa=c^2d>0$, and then translating
\[
x=X-2c^{2}d^{2},\qquad y=Y,
\]
gives the elliptic curve
\begin{equation}
E:\quad
y^{2}=x\left(x^{2}+4Hx+16a^{2}c^{4}d^{2}\right),             \label{q2:tag:6.17}
\end{equation}
where
\begin{equation*}
H=a^{2}b^{2}+c^{2}(a^{2}+d^{2}).                            
\end{equation*}
The curve is nonsingular. Indeed,
\begin{equation}
\begin{aligned}
H+2ac^{2}d&=b^{2}(a^{2}+c^{2})=b^{2}m^{2},\\
H-2ac^{2}d&=a^{2}b^{2}+c^{2}(2a-b)^{2}>0.
\end{aligned}                                                \label{q2:tag:6.19}
\end{equation}
Consequently \(H^{2}-4a^{2}c^{4}d^{2}>0\). The quadratic factor in
\eqref{q2:tag:6.17} therefore has two distinct real roots; their product is
\(16a^{2}c^{4}d^{2}>0\) and their sum is \(-4H<0\), so both roots are
negative. In particular, neither equals the third root \(x=0\), and the
cubic on the right side of \eqref{q2:tag:6.17} has no repeated root. This proves
nonsingularity. The birational equivalence above consequently extends to
an isomorphism over $\Q$ of smooth projective curves. Under that isomorphism the
point \eqref{q2:tag:6.11}, and hence \((0,a,d)\in\mathcal C\), maps to the point at
infinity: on the branch \(w\to\kappa\), the numerator in \eqref{q2:tag:6.13} tends to
\(4\kappa^{2}\ne0\) while \(z^{2}\to0\), so \(X\) has a pole. Thus
\eqref{q2:tag:6.17} identifies $(\mathcal C,(0,a,d))$ with $(E,\calO_E)$.
The coordinate to recover is $t$, which selects $Q(t)$ on $AC$.
The endpoint image $P$ and auxiliary point $S$ introduced below are the
points tested for non-torsion; $R$ is the order-$4$ reference point.

\subsection{The non-torsion alternative}

The point in \eqref{q2:tag:6.7} with \(t=1,\rho=c,\sigma=n\) corresponds, under the
birational map above, to
\begin{equation}
P=\left(
4acd(b+n),\
8acd(b+n)(an+bc-ac)
\right)\in E(\mathbb Q).                                   \label{q2:tag:6.20}
\end{equation}
Indeed, at this endpoint \eqref{q2:tag:6.8}--\eqref{q2:tag:6.9} give $z=c,\ w=2acn$,
and substitution in \eqref{q2:tag:6.13}, \eqref{q2:tag:6.15}, followed by
$x=X-2c^2d^2$, $y=Y$, gives \eqref{q2:tag:6.20}.

We shall use two further points on \(E\). The first is
\begin{equation}
R=(4ac^{2}d,\,8abc^{2}dm).                                  \label{q2:tag:6.21}
\end{equation}
Substitution using $H+2ac^{2}d=b^{2}m^{2}$ from
\eqref{q2:tag:6.19} verifies that $R\in E(\mathbb Q)$.
Lemma~\ref{lem:common-duplication}, with $B=16a^2c^4d^2$, gives
$2R=(0,0)$: the numerator vanishes while $y(R)\ne0$.
Thus $R$ has exact order $4$. The second point is
\begin{equation}
S=(-4c^{2}d^{2},\,8abc^{2}d^{2}).                           \label{q2:tag:6.23}
\end{equation}
Substitution using \eqref{q2:tag:6.19} verifies $S\in E(\Q)$.

The torsion behavior of \(S\) follows from the complete criterion in
Theorem~\ref{q1:thm:love}(\ref{q1:part:love-torsion}): after the rescaling used below, \(S\) becomes
the distinguished point \((-1,r_0)\) in that theorem.
This completes the reduction. The next two lemmas establish the non-torsion
alternative needed for density.
\end{proof}

The following quartic obstruction will exclude the remaining torsion
possibility in the real-arc argument.

\begin{lemma}[Auxiliary quartic obstruction]\label{q2:lem:quartic-p4-p2}
The affine rational points on
\begin{equation}
W^{2}=p^{4}+p^{2}-1                                         \label{q2:tag:6.26}
\end{equation}
are exactly $(p,W)=(\pm1,\pm1)$. In particular, \eqref{q2:tag:6.26} has no rational
solution with \(p>1\).
\end{lemma}

\begin{proof}
\emph{The elliptic model and its rank.}
Put \(z=p-1\). Equation \eqref{q2:tag:6.26} becomes
\begin{equation}
W^{2}=z^{4}+4z^{3}+7z^{2}+6z+1.                           \label{q2:tag:6.27}
\end{equation}
This quartic has discriminant $-400$, so it has no repeated root.
Applying \eqref{q2:tag:6.13}--\eqref{q2:tag:6.16} with
\[
A=1,\quad B=4,\quad C=7,\quad D=6,\quad \kappa=1
\]
and then setting \(x=X+3\) gives
\begin{equation}
E_0:\quad Y^{2}=x^{3}-2x^{2}+5x.                           \label{q2:tag:6.28}
\end{equation}

This model sends $(p,W)=(1,1)$ to elliptic infinity and recovers
$p=z+1$. We determine its full rational point set. Apply the
$2$-isogeny descent recalled in the proof of
Lemma~\ref{q2:lem:fixed-curve-points}
with $A=-2$ and $B=5$; its nonsingularity hypothesis holds because
$B(A^{2}-4B)=-80\ne0$. As there, the descent maps take the squareclass
of the nonzero $x$-coordinate, the class of $1$ at the identity,
and the class of the linear coefficient at $(0,0)$ on each curve.
Writing their images as $\alpha(E_0)$ and $\alpha'(E_0')$, we have
\begin{equation}
2^{\operatorname{rank}E_0(\mathbb Q)}
=\frac{|\alpha(E_0)|\,|\alpha'(E_0')|}{4}.                \label{q2:tag:6.29}
\end{equation}
Here
\begin{equation}
E_0':\quad y^{2}=x^{3}+4x^{2}-16x.                        \label{q2:tag:6.30}
\end{equation}

For \(E_0\), the possible descent classes are represented by
\(\pm1,\pm5\). Since $x^{3}-2x^{2}+5x=x\bigl((x-1)^{2}+4\bigr)$,
every affine real point has \(x\ge0\). Thus the negative classes do not
occur,
whereas the classes \(1\) and \(5\) occur at the identity and \((0,0)\).
Consequently \(|\alpha(E_0)|=2\).

For \(E_0'\), the possible classes are \(\pm1,\pm2\), and \(1,-1\)
occur at the identity and \((0,0)\), respectively. A class
\(\delta=\pm2\) could occur only if there were coprime integers \(U,V\)
and an integer \(N\) satisfying
\begin{equation*}
N^{2}=\delta U^{4}+4U^{2}V^{2}-\frac{16}{\delta}V^{4}.     
\end{equation*}
If \(U\) is odd, the right side is \(2\pmod4\). If \(U\) is even,
coprimality makes \(V\) odd; the last term on the right side then has
exact \(2\)-adic valuation \(3\), while the first two terms are divisible
by \(16\). Hence the entire right side has exact \(2\)-adic valuation
\(3\). Neither alternative is possible for a square. Therefore
\(|\alpha'(E_0')|=2\), and \eqref{q2:tag:6.29} gives
\begin{equation}
\operatorname{rank}E_0(\mathbb Q)=0.                      \label{q2:tag:6.32}
\end{equation}

\emph{The rational torsion subgroup.}
The tangent slope at $T=(1,2)$ is $1$, so $2T=(1,-2)=-T$ and $T$
has order $3$. The point $(0,0)$ has order $2$, and
$T+(0,0)=(5,-10)$. Thus these points generate the six points
\begin{equation}
\mathcal O,\quad (0,0),\quad (1,\pm2),\quad (5,\pm10).    \label{q2:tag:6.33}
\end{equation}
The quadratic factor \(x^{2}-2x+5\) has discriminant \(-16\), so
\((0,0)\) is the only nonzero rational point of order \(2\). Mazur's
theorem therefore forces the torsion subgroup to be cyclic, since every
noncyclic group in Mazur's list contains three nonzero points of order
\(2\). A cyclic torsion group containing a point of order \(6\) could be
larger only if it had order \(12\). In that case its unique subgroup of
order \(4\) would contain a rational point \(Q\) with
\(2Q=(0,0)\). Lemma~\ref{lem:common-duplication}, now with \(B=5\), would force
\(x(Q)^{2}=5\), which is impossible over \(\mathbb Q\). Therefore
the torsion subgroup has order \(6\). Together with the rank-zero
calculation \eqref{q2:tag:6.32} and the Mordell--Weil theorem, this proves that \eqref{q2:tag:6.33}
is all of \(E_0(\mathbb Q)\).

\emph{Returning to the quartic.}
It remains to pull these points back to \eqref{q2:tag:6.27}. Every affine point of
\eqref{q2:tag:6.27} with \(z\ne0\) maps to an affine point of \(E_0\), so its
\(x\)-coordinate is one of \(0,1,5\). The original cubic coordinate is
\(X=x-3\), and \eqref{q2:tag:6.13} gives
\begin{equation}
W=\frac{Xz^{2}-6z-2}{2}.                                  \label{q2:tag:6.34}
\end{equation}
For \(x=0,1,5\), substitution of \eqref{q2:tag:6.34} into \eqref{q2:tag:6.27} reduces,
respectively, to
\begin{equation*}
\frac54z^{2}(z+2)^{2}=0,\qquad
2z^{2}(z+2)=0,\qquad
-10z^{3}=0.                                                
\end{equation*}
Thus a point with \(z\ne0\) has \(z=-2\), hence \(p=-1\). The remaining
affine case \(z=0\) gives \(p=1\). In either case $W=\pm1$, and all
four pairs $(p,W)=(\pm1,\pm1)$ satisfy the original equation.
This proves the lemma.
\end{proof}

We can now prove the rank statement needed for the same-side theorem.

\begin{lemma}[Non-torsion alternative for the same-side model]
\label{q2:lem:same-side-positive-rank}
At least one of the two points \(P\) and \(S\) in \eqref{q2:tag:6.20} and \eqref{q2:tag:6.23} has
infinite order on \(E(\mathbb Q)\).
\end{lemma}

\begin{proof}
\emph{Restricting the torsion possibilities.}
If $S$ is non-torsion there is nothing to prove. Assume it is torsion.
Under
the rational change of variables
\begin{equation*}
x=4c^{2}d^{2}X,\qquad y=8c^{3}d^{3}Y,                     
\end{equation*}
equation \eqref{q2:tag:6.17} becomes
\begin{equation}
Y^{2}=X^{3}+(1+r_0^{2}+s_0^{2})X^{2}+s_0^{2}X,
\qquad
r_0=\frac{ab}{cd},\quad s_0=\frac ad,                     \label{q2:tag:6.44}
\end{equation}
and \(S\) becomes \((-1,r_0)\). Here \(r_0,s_0>0\), so
Theorem~\ref{q1:thm:love}(\ref{q1:part:love-torsion}) says that a torsion point \(S\) has order
\(4\) or \(8\).

Let $G=E(\Q)_{\mathrm{tors}}$ and $G^0=G\cap E(\R)^0$.
The quadratic factor has two distinct negative roots by
\eqref{q2:tag:6.19}, so $S$, whose abscissa is negative, lies on
the bounded component, while $R$ lies on the identity component.
Thus $|G^0|=|G|/2$. If $S$ has order $4$, a cyclic $G$ is impossible:
its unique order-$4$ subgroup would be $\langle R\rangle\subset E(\R)^0$.
If $S$ has order $8$, Mazur's theorem allows cyclic order $8$ as well.
Consequently
\[
G\in\{\Z/8\Z,\ \Z/2\Z\times\Z/4\Z,\ \Z/2\Z\times\Z/8\Z\},
\qquad |G^0|\in\{4,8\}.
\]
\emph{Locating $P$ on the real component.}
Both $P$ and $R$ have positive coordinates, and
\begin{equation*}
x(P)=4acd(b+n)>4ac^2d=x(R),
\end{equation*}
since $b+n>c$. On the upper branch from infinity toward $(0,0)$
their order is therefore $\calO_E,P,R,(0,0)$. Choose a topological
group isomorphism $E(\R)^0\cong\R/\Z$ oriented in this direction.
The unique nonidentity point of order $2$ in this component, $(0,0)$,
has parameter $1/2$. Since $2R=(0,0)$,
$R$ has parameter $1/4$, while $P$ has parameter strictly between
$0$ and $1/4$. As a finite subgroup of $\R/\Z$, $G^0$ is cyclic;
its order is $4$ or $8$, so it has at most one point in this arc,
with parameter $1/8$.
Thus a torsion $P$ would have order $8$ and satisfy $2P=R$.

\emph{Excluding $2P=R$.}
Introduce the rational right-triangle parameters:
\begin{equation}
p=\frac{m+a}{c},\qquad q=\frac{n+b}{c}.                   \label{q2:tag:6.36}
\end{equation}
The numbers \(p,q\) are rational by \eqref{q2:tag:6.1}. Moreover, \(1<p<q\):
positivity gives \(p>1\), and the function
\(x\mapsto (\sqrt{x^{2}+c^{2}}+x)/c\) is strictly increasing for
\(x>0\). Solving \eqref{q2:tag:6.36} for the triangle sides gives
\begin{equation}
a=c\frac{p^{2}-1}{2p},\quad
m=c\frac{p^{2}+1}{2p},\quad
b=c\frac{q^{2}-1}{2q},\quad
n=c\frac{q^{2}+1}{2q}.                                   \label{q2:tag:6.37}
\end{equation}

Since $an+cd>0$, the duplication formula at $P$ has no zero
denominator. Equation~\eqref{eq:common-duplication} simplifies to
\begin{equation*}
x(2P)=\left(\frac{2abcd}{an+cd}\right)^{2}.               
\end{equation*}
After substituting \eqref{q2:tag:6.37}, the difference \(x(2P)-x(R)\) factors as
\begin{equation}
\frac{c^{4}(p-1)(p+1)(p-q)(pq+1)F_1F_2}
{p^{2}q^{2}(p^{2}q-p^{2}+2pq+2p-q+1)^{2}},                \label{q2:tag:6.40}
\end{equation}
where
\begin{equation*}
\begin{aligned}
F_1&=2p^{2}q-pq^{2}+p+q^{2}+1,\\
F_2&=p^{2}q^{2}+p^{2}+pq^{2}-p+2q.
\end{aligned}                                               
\end{equation*}
For \(1<p<q\), put
\[
L=p^{2}q-p^{2}+2pq+2p-q+1.
\]
As a function of \(q\), its slope is \(p^{2}+2p-1>0\), while
\(L(p,p)=(p+1)(p^{2}+1)>0\); hence \(L>0\), and the denominator in
\eqref{q2:tag:6.40} is positive. Every numerator factor other than \(F_1\) is
nonzero, and
\[
F_2=p^{2}q^{2}+pq^{2}+(p^{2}-p)+2q>0.
\]
Thus \(2P=R\) would imply
\begin{equation}
(p-1)q^{2}-2p^{2}q-(p+1)=0.                              \label{q2:tag:6.42}
\end{equation}
The discriminant of this quadratic in \(q\) is
\(4(p^{4}+p^{2}-1)\). More explicitly, \eqref{q2:tag:6.42} implies
$W^{2}=p^{4}+p^{2}-1$ for the rational number $W=(p-1)q-p^{2}$.
This would give a rational solution of \eqref{q2:tag:6.26} with $p>1$, contrary to
the auxiliary quartic obstruction (Lemma~\ref{q2:lem:quartic-p4-p2}).
Therefore $2P=R$ is impossible, so $P$ is non-torsion when $S$ is torsion.
An order-$8$ point $P$ in $E(\R)^0$ would force $|G^0|=8$ even
without assuming that $S$ is torsion: $G^0$ is cyclic, and Mazur's
list permits no larger cyclic rational torsion subgroup containing it.
The same arc argument would then give $2P=R$, which has just been
excluded. Thus $P$ cannot have order $8$ in any case.

\end{proof}

\subsection{Density and completion of the proof}

\begin{proof}[Completion of the proof of Theorem~\ref{q2:thm:same-side}]
The real points
\begin{equation}
\gamma(t)=
\left(
t,\
\sqrt{m^{2}t^{2}-2a^{2}t+a^{2}},\
\sqrt{m^{2}t^{2}+2ad\,t+d^{2}}
\right),\qquad 0\le t\le1,                                  \label{q2:tag:6.45}
\end{equation}
form a continuous real arc of \(\mathcal C(\mathbb R)\) joining
\((0,a,d)\) to \((1,c,n)\). Under the birational identification with
\(E\), which was shown above to extend to an isomorphism of smooth
projective curves, this is an arc joining \(\calO_E\) to \(P\).
Because it starts at \(\calO_E\), the entire connected arc lies in
the identity component \(E(\mathbb R)^{0}\).

By the non-torsion alternative (Lemma~\ref{q2:lem:same-side-positive-rank}),
either \(P\) or \(S\) has
infinite order. If \(P\) has infinite order, its multiples are dense in
the circle group \(E(\mathbb R)^{0}\). If \(S\) has infinite order,
then \(2S\) has infinite order and belongs to the identity component,
since the component group has order $2$. Its multiples are therefore
dense in that component as well. Both assertions follow from the
real-component density lemma (Lemma~\ref{q2:lem:real-density}).
The portion of \eqref{q2:tag:6.45} with $0<t<1,\ \rho>0,\ \sigma>0$
is a nonempty open arc. Indeed, on \(0\le t\le1\) the radicands can be
written as $c^{2}t^{2}+a^{2}(1-t)^{2}\ \text{and}\ c^{2}t^{2}+(d+at)^{2}$,
so both are strictly positive. The Jacobian minor $4\rho\sigma$ is nonzero
along this arc, so the implicit function theorem makes $t$ a local
coordinate and proves the asserted openness. Therefore infinitely many rational points
of \(E\) belonging to the relevant dense cyclic subgroup lie on the open
arc. Indeed, if the intersection were finite, a smaller nonempty open
subarc disjoint from that finite set would contradict density. The
isomorphism is defined over \(\mathbb Q\), so their inverse images supply
infinitely many distinct triples $t,\rho,\sigma\in\mathbb Q,\ 0<t<1$,
satisfying \eqref{q2:tag:6.6}; the triples have distinct $t$-coordinates because the arc
in~\eqref{q2:tag:6.45} is parametrized by its first coordinate.

For each such triple take $Q=(ct,a(1-t))$.
These are distinct rational points in the open segment \(AC\), and
equations \eqref{q2:tag:6.3}--\eqref{q2:tag:6.6} give $QA=tm,\ QC=(1-t)m,\ QO=\rho,\ QB=\sigma$.
All four distances are positive rational numbers, and the barycentric
calculation following \eqref{q2:tag:6.2} proves strict interiority.
This proves Theorem~\ref{q2:thm:same-side}.
\end{proof}

\paragraph{Perspective.}
Theorems~\ref{q1:thm:main} and~\ref{q2:thm:main-degenerate} settle the
quadrilateral step in the local insertion program: the required rational
point is admissible, on a prescribed interior diagonal in the non-degenerate
case and in the triangular interior in the three-collinear case. The finite
and empty diagonal examples explain why existence, infinitude, and location
must be distinguished. The remaining step is
Conjecture~\ref{q1:conj:pentagon}: for every pentagon satisfying its arithmetic
hypotheses, find a rational-distance point whose open spokes to all five
vertices stay strictly inside the face. The conditional reduction in
Theorem~\ref{q1:thm:conditional} identifies this admissible pentagon problem
as sufficient for an integral drawing preserving any prescribed plane embedding.

\appendix

\section{The quadratic exceptional model and its torsion table}
\label{q1:app:caseIII}

This appendix verifies the birational model, the complete rational torsion
subgroup, and the torsion-value table used in
Proposition~\ref{q1:prop:third}. Throughout \(k\in\Q_{>0}\) and \(k\ne1\).
Let $f_k$ be the polynomial on the right of
\eqref{q1:eq:third-quartic}, so $C_k:W^2=f_k(t)$.
Its polynomial
discriminant is
\[
\operatorname{disc}_t f_k
=256k^{16}(k-1)^2(k+1)^2(k^2+1)^4(k^4+6k^2+1)>0.
\]
Thus $f_k$ is squarefree of degree $4$, with leading coefficient $k^6\ne0$.
Consequently $C_k$ is geometrically irreducible and nonsingular, and its
smooth projective model $\overline C_k$ has genus one.
All statements about points at infinity or the group law on the quartic
refer to $\overline C_k$. The chosen origin is
\[
O_k=(k,k(k^2+1)).
\]
It lies on $C_k$ because $f_k(k)=k^2(k^2+1)^2$.
Writing $u=t-k$, define
\begin{equation*}
H_k(u)=k(k^2+1)+(-k^4+2k^2+1)u+k^3u^2.         
\end{equation*}
A direct expansion gives the identity
\begin{equation}
H_k(u)^2=f_k(k+u)-4k^5(k^2-1)u^3.             \label{q1:tag:A.2}
\end{equation}

\subsection*{The birational maps}
\addcontentsline{toc}{subsection}{The birational maps}

Let \(E_k\) be the generalized Weierstrass curve
\begin{equation}
\begin{aligned}
E_k:\quad
Y^2&-2(k^4-2k^2-1)XY+8k^6(k^4-1)Y\\
&=X^3-4k^4(k^2+1)X^2.
\end{aligned} \label{q1:tag:A.3}
\end{equation}
Write $\mathcal E_k(X,Y)$ for the left side minus the right side
of \eqref{q1:tag:A.3}.
For this generalized Weierstrass equation the discriminant is
\begin{equation*}
\Delta(E_k)
=4096\,k^{16}(k-1)^2(k+1)^2(k^2+1)^4
(k^4+6k^2+1).                                 
\end{equation*}
The hypotheses $k>0$, $k\ne1$ make every factor nonzero,
so $E_k$ is nonsingular, with identity $\calO$.

Define a rational map
\[
\Phi:C_k\dashrightarrow E_k
\]
on the open set \(u\ne0\) by
\begin{equation}
X=\frac{2k(k^2+1)(W+H_k(u))}{u^2}, \label{q1:tag:A.5}
\end{equation}
\begin{equation}
Y=\frac{4k^2(k^2+1)^2(W+H_k(u))}{u^3}. \label{q1:tag:A.6}
\end{equation}
Substitution of \eqref{q1:tag:A.5} and \eqref{q1:tag:A.6} into the defining
polynomial $\mathcal E_k$ gives
\begin{equation*}
\mathcal E_k(\Phi(t,W))
=\frac{8k^3(k^2+1)^3(W+H_k(u))(f_k(k+u)-W^2)}
{u^6}.                                   
\end{equation*}
Hence \(\Phi\) maps \(C_k\) to \(E_k\) wherever it is defined.

Define rational coordinate functions on the affine open set $Y\ne0$ by
\begin{equation}\label{q1:tag:A.8}
t=k+\frac{2k(k^2+1)X}{Y},
\end{equation}
\begin{equation}\label{q1:tag:A.9}
W=\frac{k(k^2+1)X^3+8k^7(k^2-1)(k^2+1)^2Y}{Y^2}.
\end{equation}
These define the candidate inverse $\Psi$. To verify it, use
\eqref{q1:tag:A.2} on $C_k$:
\[
(W+H_k(u))(W-H_k(u))=4k^5(k^2-1)u^3.
\]
Since $k>0$ and $k\ne1$, the right side is nonzero when $u\ne0$.
Thus $W+H_k(u)$ and hence $XY$ are nonzero on this open set.
Equations \eqref{q1:tag:A.5}--\eqref{q1:tag:A.6} give
$2k(k^2+1)X/Y=u$, so \eqref{q1:tag:A.8} recovers $t$;
substitution in \eqref{q1:tag:A.9}, using the displayed identity,
recovers $W$. Thus $\Psi\Phi$ is the identity on $u\ne0$.

The ratio $X/Y=u/(2k(k^2+1))$ also makes $\Phi$ nonconstant,
hence dominant between the irreducible curves $C_k$ and $E_k$.
It follows that the rational identity $W^2=f_k(t)$ for $\Psi$
and the composition identity $\Phi\Psi=\mathrm{id}$ hold on $E_k$:
both hold on the dense image of $\Phi$. All coordinates in these
identities are regular where $XY\ne0$, and \eqref{q1:tag:A.8}
there gives $u\ne0$. Therefore $\Phi$ and $\Psi$ are inverse
isomorphisms on the affine open subsets $u\ne0$ and $XY\ne0$.
They extend uniquely to inverse isomorphisms of smooth projective
curves $\overline C_k\simeq E_k$ over $\Q$.

\subsection*{Extension at the omitted points}
\addcontentsline{toc}{subsection}{Extension at the omitted points}

As $u\to0$ along the branch through
\[
O_k=(k,k(k^2+1)),
\]
$W+H_k(u)$ tends to $2k(k^2+1)\ne0$. Since $2W(O_k)\ne0$, the coordinate
$u$ is a local parameter there. Thus $X$ and $Y$ have poles of
orders $2$ and $3$, respectively, and $O_k$ maps to $\calO$. Conversely,
at $\calO$ one has $X/Y\to0$, $X^3/Y^2\to1$, and $1/Y\to0$, as follows from the
usual local parameter $-X/Y$ on a Weierstrass cubic. Equations
\eqref{q1:tag:A.8}--\eqref{q1:tag:A.9} therefore give $t\to k$ and
$W\to k(k^2+1)$. The isomorphism consequently respects the chosen
origins and transports the elliptic group law to $\overline C_k$.

At the other point $O_k^-=(k,-k(k^2+1))$ above $u=0$, the factor
$W-H_k(u)$ tends to $-2k(k^2+1)\ne0$. Rationalizing with
\eqref{q1:tag:A.2} gives
\[
\frac{W+H_k(u)}{u^3}
=\frac{4k^5(k^2-1)}{W-H_k(u)}
\longrightarrow -\frac{2k^4(k^2-1)}{k^2+1}.
\]
It follows from \eqref{q1:tag:A.5}--\eqref{q1:tag:A.6} that
\[
\Phi(O_k^-)=(0,-8k^6(k^4-1)).
\]

There are exactly two points at infinity on $\overline C_k$. To see this,
put $v=1/u$ and $\omega=v^2W$. The equation in this chart is
\[
\omega^2=v^4f_k(k+v^{-1}),
\]
whose right side is a polynomial in $v$ with value $k^6$ at $v=0$.
Thus the two points are $I_\pm=(v,\omega)=(0,\pm k^3)$; they are rational
and nonsingular because $2\omega\ne0$. In this chart the forward map is
\[
\begin{aligned}
X&=2k(k^2+1)\bigl(\omega+v^2H_k(v^{-1})\bigr),\\
Y&=2k(k^2+1)vX.
\end{aligned}
\]
Since $v^2H_k(v^{-1})\to k^3$, we obtain
\[
\Phi(I_-)=(0,0),\qquad
\Phi(I_+)=\bigl(4k^4(k^2+1),0\bigr).
\]
The function $t=k+v^{-1}$ has a pole at each of $I_-$ and $I_+$.
These are all four points omitted on $\overline C_k$. On $E_k$,
$X=0$ gives $Y=0,-a_3$, where $a_3=8k^6(k^4-1)$, while
$Y=0$ gives $X=0,4k^4(k^2+1)$. Together with $\calO$, these are
exactly the four corresponding points outside the common affine open set.

\subsection*{The multiples of the distinguished torsion point}
\addcontentsline{toc}{subsection}{The multiples of the distinguished torsion point}

We now verify the torsion table on \(E_k\). The curve \eqref{q1:tag:A.3} is in the
generalized Weierstrass form
\[
Y^2+a_1XY+a_3Y=X^3+a_2X^2,
\]
with
\[
a_1=-2(k^4-2k^2-1),\qquad
a_3=8k^6(k^4-1),\qquad
a_2=-4k^4(k^2+1).
\]
The negation formula is therefore
\begin{equation}
-(X,Y)=
\bigl(X,\ -Y+2(k^4-2k^2-1)X-8k^6(k^4-1)\bigr). \label{q1:tag:A.12}
\end{equation}

The following points lie on \(E_k\), by direct substitution:
\begin{equation}
\begin{aligned}
P_1&=(0,0),\\
P_2&=\bigl(4k^4(k^2+1),-8k^4(k^2+1)^2\bigr),\\
P_3&=\bigl(4k^2(k^2-1)(k^2+1),
-8k^4(k^2-1)(k^2+1)\bigr),\\
P_4&=(4k^6,-8k^8),\\
P_5&=\bigl(4k^2(k^2-1)(k^2+1),
-8k^2(k^2-1)(k^2+1)^2\bigr),\\
P_6&=\bigl(4k^4(k^2+1),0\bigr),\\
P_7&=\bigl(0,-8k^6(k^4-1)\bigr).
\end{aligned}                                          \label{q1:tag:A.13}
\end{equation}
Equation \eqref{q1:tag:A.12} gives
\[
-P_1=P_7,\qquad -P_2=P_6,\qquad -P_3=P_5,\qquad -P_4=P_4.
\]
Thus \(P_4\) is a point of order \(2\).

At $P_1$ the partial derivatives of $\mathcal E_k$ are $0,a_3$,
with $a_3\ne0$, so its tangent is $Y=0$. For $k\in\Q_{>0}$, $k\ne1$,
the table gives the complete intersection divisor for each displayed
line and the resulting group identity. Substitution gives coefficient
$-1$ for $X^3$, so all three intersections are counted with multiplicity.
\[
\renewcommand{\arraystretch}{1.15}
\begin{array}{@{}ccc@{}}
\toprule
\text{Line}&\text{Intersection divisor}&\text{Group identity}\\
\midrule
Y=0&2P_1+P_6&2P_1=P_2\\
Y=-2(k^2+1)X&P_1+P_2+P_5&P_1+P_2=P_3\\
Y=-2k^2X&P_1+P_3+P_4&P_1+P_3=P_4\\
\bottomrule
\end{array}
\]
The chord-and-tangent law and the negation relations therefore give
\[
nP_1=P_n\quad(1\le n\le7),\qquad 8P_1=\calO.
\]
Since $4P_1=P_4\ne\calO$ and $P_4$ has order $2$, $T=P_1$
has exact order $8$.

\subsection*{Completeness of the rational torsion subgroup}
\addcontentsline{toc}{subsection}{Completeness of the rational torsion subgroup}

To use the table for every rational torsion point, we must also exclude
torsion outside $\langle T\rangle$. By Mazur's theorem \cite{Mazur}, a
rational point of order $8$ leaves only the possibilities
\[
E_k(\Q)_{\mathrm{tors}}\cong\Z/8\Z
\quad\text{or}\quad \Z/2\Z\times\Z/8\Z.
\]
The second possibility would give additional rational points of order $2$.
By \eqref{q1:tag:A.12}, an affine point of order $2$ satisfies
\[
Y=(k^4-2k^2-1)X-4k^6(k^4-1).
\]
Substitution in $\mathcal E_k(X,Y)=0$ gives
\[
(X-4k^6)Q_k(X)=0,
\]
where
\[
Q_k(X)=X^2+(k^8-4k^6-2k^4+4k^2+1)X
       -4k^6(k^4-1)^2.
\]
The first factor gives the known point $P_4$, and
$Q_k(4k^6)=16k^8\ne0$, so this root is distinct from the roots of $Q_k$.
The discriminant of $Q_k$ is
\[
(k-1)^2(k+1)^2(k^2+1)^4(k^4+6k^2+1).
\]
Its first three factors form a nonzero rational square. Thus extra
rational $2$-torsion would require
\[
k^4+6k^2+1=h^2
\]
for some $h\in\Q$. This contradicts the self-contained Fermat quartic
obstruction in Lemma~\ref{q2:lem:fermat-quartic}, because $k>0$.
Hence $P_4$ is the only nonzero rational point of
order $2$, and
\[
E_k(\Q)_{\mathrm{tors}}=\langle T\rangle
=\{\calO,P_1,\ldots,P_7\}.
\]

\subsection*{The \texorpdfstring{\(t\)}{t}-values of the torsion points}
\addcontentsline{toc}{subsection}{The \texorpdfstring{\(t\)}{t}-values of the torsion points}

The inverse map \(\Psi\) gives
\begin{equation}
t=k+\frac{2k(k^2+1)X}{Y}                         \label{q1:tag:A.14}
\end{equation}
for affine points with \(Y\ne0\). The point \(\calO\) corresponds to the
chosen origin \(O_k\) on \(C_k\), hence to \(t=k\). For the points with
$Y\ne0$, substitute in \eqref{q1:tag:A.14}; for the remaining points, use
the extensions computed above. For $k\in\Q_{>0}$, $k\ne1$, the table
lists every rational torsion point of $E_k$ and its parameter and detour
values. All eight rows occur for every such $k$.
\[
\renewcommand{\arraystretch}{1.15}
\begin{array}{@{}ccc@{}}
\toprule
\text{Point on }E_k & \text{Parameter }t &
\text{Detour value }\alpha_1=(t-t^{-1})/2\\
\midrule
\calO & k & (k-k^{-1})/2\\
P_1=T & \infty & \text{undefined}\\
P_2=2T & 0 & \text{undefined}\\
P_3=3T & -1/k & (k-k^{-1})/2\\
P_4=4T & -1/k & (k-k^{-1})/2\\
P_5=5T & 0 & \text{undefined}\\
P_6=6T & \infty & \text{undefined}\\
P_7=7T & k & (k-k^{-1})/2\\
\bottomrule
\end{array}
\]
The entries $t=0$ and $t=\infty$ give poles of $\alpha_1$, hence no
finite detour value. At $P_1=\Phi(I_-)$ and $P_6=\Phi(I_+)$, the
function $t$ itself has a pole. The two finite points above $t=k$ are
exactly $O_k$ and $O_k^-$;
their images are $\calO$ and $P_7$, respectively. Thus no point excluded
by the affine formulas contributes an additional finite $t$-value.

Finally, $t=k$ and $t=-1/k$ both give
\[
\frac{k-k^{-1}}2=\frac{k^2-1}{2k}=\tau,
\]
so this is the only finite $\alpha_1$-value represented by rational
torsion. This proves the torsion-value assertion used in
Proposition~\ref{q1:prop:third}.

\section{Rational points on a fixed elliptic curve}
\label{q2:app:fixed-curve}

The same rank-zero elliptic curve occurs in two finite examples and in the
diagonal counterexample. We determine its rational points once and use the
result in all three places. The origin is elliptic infinity. The points
$S=(24,840)$ and $H=(-81,0)$ will generate the rational point set;
each geometric application gives its coordinate recovery explicitly.

\begin{lemma}[Rational points on the fixed curve]
\label{q2:lem:fixed-curve-points}
Let
\[
E:\quad Y^2=X(X+81)(X+256).
\]
Then
\[
\operatorname{rank}E(\Q)=0,
\qquad
E(\Q)\cong \Z/8\Z\times\Z/2\Z,
\]
and
\[
\begin{split}
E(\Q)=\{&
\calO,(0,0),(-81,0),(-256,0),\\
&(-216,\mathord\pm1080),(-144,\mathord\pm1008),
(-96,\mathord\pm480),\\
&(24,\mathord\pm840),(144,\mathord\pm3600),
(864,\mathord\pm30240)\}.
\end{split}
\]
In particular, its nonnegative finite $X$-coordinates are
$0,24,144,864$.
\end{lemma}

\begin{proof}
The roots $0,-81,-256$ are distinct, so $E$ is nonsingular.
For the isogenous pair below, both origins are at infinity and the
$x$-squareclasses are the coordinates used to compute rank.
We use the standard $2$-isogeny descent in the following explicit form;
see \cite[Chapter~X, \S4]{SilvermanAEC} or
\cite[\S3.6, Method~1]{Cremona}. For
\[
F:\quad y^2=x^3+Ax^2+Bx
\]
with $A,B\in\Z$ and $B(A^2-4B)\ne0$, let
\[
F':\quad y^2=x^3-2Ax^2+(A^2-4B)x
\]
be the curve $2$-isogenous to $F$. The descent map on $F$, with values
in $\Q^\times/(\Q^\times)^2$, sends a point
with $x\ne0$ to the squareclass $[x]$, sends $(0,0)$ to $[B]$, and sends
the identity to $[1]$; define the map on $F'$ analogously. A squarefree
signed divisor $d$ of $B$ lies in the first image exactly when
\begin{equation}
N^2=dU^4+AU^2V^2+\frac{B}{d}V^4
\label{q2:eq:fixed-descent-quartic}
\end{equation}
has a solution $N,U,V\in\Z$ with $\gcd(U,V)=1$, including the
cases $UV=0$. For $UV\ne0$, such a solution gives the point
\[
x=\frac{dU^2}{V^2},\qquad y=\frac{dUN}{V^3}
\]
on $F$, with squareclass $[x]=[d]$. If $V=0$, primitivity gives
$U=\pm1$ and $N^2=d$, representing the identity class; if $U=0$,
it gives $V=\pm1$ and $N^2=B/d$, so $[d]=[B]$, the class of $(0,0)$.
Every image class has such a signed squarefree divisor representative.
The same description applies to $F'$.
If the two images have orders $m$ and $m'$, then
\begin{equation}
2^{\operatorname{rank}F(\Q)}=\frac{mm'}4.
\label{q2:eq:fixed-descent-rank}
\end{equation}

For $E$ we have $A=337$ and $B=20736=144^2$, so the possible
squareclasses are
\[
\mathord\pm1,\quad\mathord\pm2,\quad
\mathord\pm3,\quad\mathord\pm6.
\]
The classes $1,-1,6,-6$ occur: besides the identity, they are represented
by
\[
(-81,0),\qquad (24,840),\qquad (-96,480).
\]
The remaining classes $d=\pm2,\pm3$ do not occur. Indeed, modulo $5$
the coefficients of $U^4$ and $V^4$ in
\eqref{q2:eq:fixed-descent-quartic} are $2$ and $3$, in some order,
whereas $337\equiv2\pmod5$. Primitivity excludes $U\equiv V\equiv0$
modulo $5$. If exactly one vanishes, the right side is $2$ or $3$;
if neither vanishes, then $U^4=V^4=1$ and $U^2V^2\in\{1,4\}$
modulo $5$, giving the same two residues. Neither is a square.
Thus the first descent image has order $4$.

The isogenous curve is
\[
E':\quad y^2=x^3-674x^2+30625x=x(x-49)(x-625).
\]
Its possible squareclasses are
\[
\mathord\pm1,\quad\mathord\pm5,\quad
\mathord\pm7,\quad\mathord\pm35.
\]
For $x<0$, the product $x(x-49)(x-625)$ is negative, so $E'(\R)$
has no such affine point. The exceptional descent values at infinity
and $(0,0)$ are also positive squareclasses. Thus every negative
class is excluded, leaving only $1,5,7,35$.
For $d=5,35$, the descent quartic reduces modulo $3$ to
\[
N^2\equiv2U^4+U^2V^2+2V^4\equiv2\pmod3
\]
for every primitive pair $(U,V)$: primitivity excludes both coordinates
being divisible by $3$. This is impossible.
For $d=7$ its coefficients modulo $16$ are $7,14,7$.
Primitivity ensures that at least one of $U,V$ is odd.
If both $U,V$ are odd the right side is $12$ modulo $16$;
if exactly one is odd, it is $7$ or $15$. None is a square.
Hence the second descent image consists only of $1$, and
\eqref{q2:eq:fixed-descent-rank} gives $\operatorname{rank}E(\Q)=0$.

It remains to identify the torsion subgroup. Direct substitution gives
$S=(24,840)\in E(\Q)$. The
tangent slope at $S$ is $23$, so the addition law gives
\[
2S=(144,-3600).
\]
Lemma~\ref{lem:common-duplication} with $B=144^2$, applied to
$2S$, gives $4S=(0,0)$. Thus $S$ has exact order $8$.
The point $H=(-81,0)$ does not belong to $\langle S\rangle$, whose
unique nonzero point of order $2$ is $(0,0)$. Hence
\[
\langle S,H\rangle\cong\Z/8\Z\times\Z/2\Z.
\]
By Mazur's theorem \cite{Mazur}, a noncyclic rational torsion group has
order at most $16$, so this subgroup is all of $E(\Q)_{\mathrm{tors}}$.
The rank calculation and the Mordell--Weil theorem then give
$E(\Q)=\langle S,H\rangle$.
Direct substitution verifies the sixteen distinct points in the
statement; their number proves that the list is complete.

For comparison, the change $X=4u-112$, $Y=8v+4u$ gives the integral model
\[
v^2+uv=u^3-1070u+7812,
\qquad \Delta=2^8\cdot3^8\cdot5^4\cdot7^2.
\]
This model is minimal: an integral Weierstrass model with smaller
discriminant valuation at a prime would decrease that valuation by a
positive multiple of $12$. This is impossible because every prime
valuation of $\Delta$ is less than $12$ and an integral model has
nonnegative discriminant valuations. It is recorded with LMFDB label $210.e6$ and Cremona
label $210e2$ \cite{Cremona,LMFDB210e6}.
\end{proof}

{\bf Use of AI}:  The author proved all the non-degenerate quadrilateral cases manually and used ChatGPT to verify and find the missing gaps. The author fixed those missing gaps and iteratively verified the proofs with ChatGPT and then manually. For the degenerate quadrilateral cases, the author used ChatGPT to prove the {\em three-collinear quadrilateral case} and verified and expanded the proof manually. The examples and diagrams were generated by ChatGPT. The author is solely responsible for the correctness of all the proofs.

\end{document}